\documentclass[EJP]{ejpecp} % replace ECP by EJP if needed.
\usepackage{nicematrix}

\SHORTTITLE{A General Coupling for Ising Models and Beyond}
\usepackage{tikz-cd}
\usepackage{graphics}
\usepackage[]{todonotes}
\usepackage{physics}
\usepackage{setspace}
\usepackage{caption}
\usepackage{subcaption}
\usepackage{floatrow}
\usepackage{xpatch}

\usepackage{graphicx}
\usepackage{mathrsfs}
\usepackage{verbatim}

\DeclareMathOperator*{\UC}{\mathtt{UC}}
\DeclareMathOperator*{\UEG}{UEG}

\DeclareMathOperator{\XOR}{XOR}

\newcommand{\N}{\mathbb{N}}

\newcommand{\Z}{\mathbb{Z}}

\newcommand{\cc}{\leftrightarrow}
\newcommand{\Prb}{\mathbb{P}}
\newcommand{\XORI}{\mu^{\XOR}}
\newcommand{\Ising}{\mu}
\newcommand{\Prbcur}{\mathbf{P}}
\newcommand{\Bernoulli}{\mathbb{P}}
\newcommand{\loopmeasure}{\ell}

\newcommand{\Even}{\mathcal{E}_{\emptyset}}

\usepackage{tikz}
\usepackage{booktabs}

\usepackage{tikz-3dplot} 
\usetikzlibrary{positioning}% To get more advances positioning options
\usetikzlibrary{arrows}% To get more arrow heads

\newcommand{\id}{1\! \!1}

\newcommand{\dc}{\Prbcur^{\emptyset, \emptyset}}

\TITLE{A General Coupling for Ising Models and Beyond}
    
\AUTHORS{%
Ulrik Thinggaard Hansen\footnote{Department of Mathematics,
Universität Innsbruck, Technikerstrasse 13, 6020 Innsbruck, Austria.
    \EMAIL{ulrik.hansen@uibk.ac.at}}%\orcid{0000-0002-2697-5725}
  \and Jianping Jiang \footnote{Yau Mathematical Sciences Center, Tsinghua University, Beijing 100084, China.
    \EMAIL{jianpingjiang@tsinghua.edu.cn}}%% remove this line and below if single author 
\and Frederik Ravn Klausen\footnote{DPMMS, University of Cambridge, United Kingdom.
    \EMAIL{frk23@cam.ac.uk}}}%%AUTHORS %\orcid{0000-0001-7815-4315}
 % Separate items with ;

\AMSSUBJ{82B20; 60K35; 82B27; 60G60} % Edit. Separate items with ;
\SUBMITTED{October 15, 2025} % Edit.
\ACCEPTED{July 13, 2026} % Edit.

\VOLUME{31}
\YEAR{2026}
\PAPERNUM{117}
\DOI{10.1214/26-EJP1573}

\usepackage[english]{babel}
\usepackage[utf8]{inputenc}

\usepackage{amsmath, amssymb, amsthm}
\usepackage{xcolor}
\usepackage{tikz-cd}
\usepackage{graphics}
\usepackage[]{todonotes}
\usepackage{physics}
\usepackage{setspace}

\usepackage{caption}
\usepackage{subcaption}
\usepackage{floatrow}
\usepackage{xpatch}

\usepackage{graphicx}
\usepackage{mathrsfs}
\usepackage{verbatim}
\usepackage{adjustbox}
\newcommand{\nocontentsline}[3]{}
\newcommand{\tocless}[2]{\bgroup\let\addcontentsline=\nocontentsline#1{#2}\egroup}

\usepackage{tikz}
\usepackage{pgfplots}
\usepackage{tikz-3dplot} 
\usetikzlibrary{positioning}% To get more advances positioning options
\usetikzlibrary{arrows}% To get more arrow heads

\usepackage[nameinlink,noabbrev]{cleveref}

\makeatletter
\newcommand{\RedeclareTheoremEnv}[2]{%
  \expandafter\let\csname #1\endcsname\relax
  \expandafter\let\csname end#1\endcsname\relax
  \newtheorem{#1}[theorem]{#2}%
}

\theoremstyle{ejpecpbodyit}
\RedeclareTheoremEnv{proposition}{Proposition}
\RedeclareTheoremEnv{lemma}{Lemma}
\RedeclareTheoremEnv{corollary}{Corollary}
\RedeclareTheoremEnv{definition}{Definition}

\theoremstyle{ejpecpbodyrm}
\RedeclareTheoremEnv{remark}{Remark}
\RedeclareTheoremEnv{example}{Example}
\makeatother

\usetikzlibrary{positioning,calc}

\newlength\tindent
\makeatletter
\newcommand{\changeoperator}[1]{%
  \csletcs{#1@saved}{#1@}%
  \csdef{#1@}{\changed@operator{#1}}%
}
\newcommand{\changed@operator}[1]{%
  \mathop{%
    \mathchoice{\textstyle\csuse{#1@saved}}
               {\csuse{#1@saved}}
               {\csuse{#1@saved}}
               {\csuse{#1@saved}}%
  }%
}

\usetikzlibrary{calc,arrows.meta}
\usepackage{graphicx} 
\changeoperator{prod}
\pgfplotsset{compat=1.18}

\ABSTRACT{
The couplings between the Ising model and its graphical representations, the random-cluster, random current and loop $\mathrm{O}(1)$ models, are put on common footing through a generalization of the Swendsen-Wang-Edwards-Sokal coupling. A new special case is that the single random current with parameter $2J$ on edges of agreement of XOR-Ising spins has the law of the double random current. 
The coupling also yields a general mechanism for constructing conditional Bernoulli percolation measures via uniform sampling, providing new perspectives on models such as the arboreal gas, random $d$-regular graphs, and self-avoiding walks.
As an application of the coupling for the Loop-Cluster joint model, we prove that the regimes of exponential decay of the $q$-state Potts model, the random-cluster representation, and its $q$-flow loop representation coincide on the torus, generalizing the Ising result.  By providing a new and equivalent definition of the plaquette random cluster model, we are able to generalize the loop–cluster coupling to lattice gauge theories.
}

\begin{document}
\section{Introduction}
The Swendsen-Wang cluster algorithm \cite{swendsen1987nonuniversal} %utilizing the intimate relation between the FK random cluster model \cite{fortuin1972random} and the Ising model \cite{lenz1920beitrag} 
was a breakthrough for minimizing the effect of critical slowdown of MCMC algorithms. Soon thereafter, it was interpreted  and generalized \cite{chayes1998graphical,edwards1988generalization}. 
This paradigm has since found countless applications and extensions for simulations \cite{sokal1997monte, zhang2020loop}.
%and to obtain rigorous results \cite{DC17,duminil2022100}.
%Highlights include....\cite{??} up to the recent development of the Loop-Cluster algorithm for the random cluster model for general $q$ \cite{zhang2020loop}.
In parallel, the use of the random-cluster and random current representations of the Ising model \textcolor{black}{for obtaining rigorous results} has exploded \cite{aizenman1988discontinuity, DC16,DC17,duminil2022100,Gri06}. \textcolor{black}{This endeavor has been aided by the} many couplings between the graphical representations of the Ising model that have been found - see \Cref{fig:couplings_horizontal} for an overview. \textcolor{black}{It has often proven useful to switch between several different stochastic geometric representations - that are  tied together by  compatibility relations (e.g., the spins are constant on FK clusters). }The following theorem unifies all the mentioned couplings along with some new extensions and illuminates their common structure:
\begin{theorem}\label{thm:generalcoupling}
	Let $\Omega$ and $\Sigma$ be two finite sets, $\rho$ and $\gamma$ be probability measures on $\Omega$ and $\Sigma$ respectively, and $f: \Omega \to 2^{\Sigma}$ be a function\footnote{Here, $2^{\Sigma}$ is the power set of $\Sigma$.}. For each $\eta \in \Sigma$, let
    $f^{-1}(\eta):=\{\omega \in \Omega \mid \eta \in f(\omega)\}.$
     Suppose that $ \sum_{\omega \in \Omega, \eta \in \Sigma} \rho[\omega] \gamma[\eta] \id[\eta \in f(\omega)] > 0$ and then define a coupling probability measure on $\Omega\times\Sigma$ by
	\[\mathscr{P}[\omega,\eta]\propto \rho[\omega] \gamma[\eta] \id[\eta \in f(\omega)],~\forall (\omega, \eta)\in\Omega\times\Sigma. \] 

	\begin{enumerate}
    \item[(a)] The marginal of $\mathscr{P}$ on $\Omega$, denoted by  $\mathscr{P}_{\Omega}$, satisfies $ \mathscr{P}_{\Omega}[\omega]  \propto \gamma[f(\omega)]\rho[\omega]$ for each $\omega \in \Omega$. 
		\item[(b)] The marginal of $\mathscr{P}$ on $\Sigma$, denoted by $\mathscr{P}_{\Sigma}$, satisfies $ \mathscr{P}_{\Sigma}[\eta] \propto  \rho[f^{-1}(\eta)] \gamma[\eta] $ for each $\eta \in \Sigma$.   

\item[(c)] For each $\omega\in\Omega$ with $\mathscr{P}_{\Omega}[\omega]\neq 0$,  the conditional measure $\mathscr{P}[\cdot | \omega]$ on $\Sigma$ is $\gamma[\cdot \mid f(\omega)]$.
        \item[(d)]   For each $\eta\in\Sigma$ with $\mathscr{P}_{\Sigma}[\eta]\neq 0$, the conditional measure $\mathscr{P}[\cdot|\eta] $ on $\Omega$ is $\rho[\cdot \mid f^{-1}(\eta)]$.
	\end{enumerate}
\end{theorem}

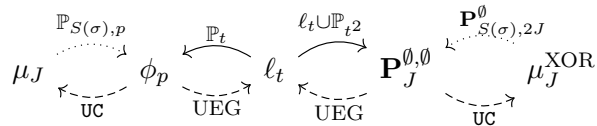
\begin{figure} 
\vspace{-0.1cm}
\captionsetup{width=.95\linewidth}
\adjustbox{scale=1,center}{\begin{tikzcd}
    \Ising_{J}\arrow[ r,bend left, dotted, "\Prb_{S(\sigma),p}"] 
    &   \phi_p      \arrow[l, dashed,bend left, "\mathtt{UC}"]  \arrow[r, bend right, dashed, "\UEG"{yshift=-10pt}]
      & \ell_t  \arrow[l,bend right,"\Prb_{t}"{yshift=+10pt}]       \arrow[r,bend left, "\ell_t \cup \Prb_{t^2}"{xshift=15pt}]
    &  \dc_J  \arrow[l, bend left, dashed, "\UEG"{xshift=15pt}] \arrow[r, dashed,bend right, "\mathtt{UC}"{yshift=-10pt,xshift=-5pt}] 
    &    \XORI_{J} \arrow[l,bend right, dotted,"\Prbcur_{S(\sigma),2J}^{\emptyset}"{yshift=14pt,xshift=-10pt}]     
\end{tikzcd}}
{\caption{Couplings of the graphical representations of the ferromagnetic Ising model on any finite graph $G=(V,E)$ with coupling constants $J=(J_e)_{e\in E}$. 
Dashed arrows indicate a uniform even subgraph ($\UEG$)  or a uniform coloring ($\mathtt{UC}$) of clusters by $+$ and $-$. Full arrows indicate a union with another percolation measure. Dotted arrows correspond to running the process on the graph of agreement edges of the spin configuration. Here, $\mu_J$ is the Ising model, $\phi_p$ the FK random-cluster model, $\Prbcur^\emptyset_{J}$, $\Prbcur^{\emptyset,\emptyset}_{J}$  the single and double random current measures,  $\ell_t$ the loop $\mathrm{O}(1)$ model (a.k.a. high-temperature expansion) and $\mathbb{P}$ is Bernoulli percolation - all to be defined below. $\XORI_{J}$ is the XOR of two independent Ising models. Original references are \cite{aizenman2019emergent,duminil2021conformal2,edwards1988generalization,evertz2002new,grimmett2007random,klausen2021monotonicity,lis2017planar,Lis,lupu2016note,swendsen1987nonuniversal}
%\cite{swendsen1987nonuniversal,edwards1988generalization, evertz2002new, grimmett2007random,lupu2016note,klausen2021monotonicity,Lis,lis2017planar,duminil2021conformal2, aizenman2019emergent}
and \Cref{prop:doublecurrent-Xor}. \Cref{thm:generalcoupling} unifies all these couplings and the corresponding couplings with sources (\Cref{fig:couplings_with_sources}). 
} 
\label{fig:couplings_horizontal}}
\end{figure}   
\begin{remark}\label{rem:gamma}
The corresponding algorithm: $\mathscr{P}_{\Sigma}$ can be sampled from a $\mathscr{P}_{\Omega}$-random configuration $\omega$ as $\gamma$ conditioned to lie in $f(\omega)$. $\mathscr{P}_{\Omega}$ can be sampled from a $\mathscr{P}_{\Sigma}$-random configuration $\eta$ as $\rho$ conditioned on $f^{-1}(\eta)$. 
In most applications, $\gamma$ is the uniform measure on $\Sigma$. Then, $\mathscr{P}[\;\cdot \mid \omega]$ is uniform on $f(\omega)$.
\end{remark}

 For a finite graph $G= (V,E)$ with coupling constants $(J_e)_{e\in E}$, the Swendsen-Wang algorithm for the Ising model is recovered by letting $\Omega = \{0,1\}^E$, $\Sigma = \{-1,+1\}^V$, $\rho = \Bernoulli_p$ be Bernoulli percolation with parameter $p_e:=1-e^{-2J_e}$, $\gamma$ be the uniform measure on $\Sigma$. Define the satisfied edges by $S(\sigma) = \{uv  \in E \mid \sigma_u = \sigma_v\} $ and $f:\Omega \rightarrow 2^{\Sigma}$ by 
$
      f(\omega) = \{ \sigma \in \Sigma \mid \omega_e = 0, \forall e \not \in S(\sigma) \}, 
$
 the map that sends $\omega$ to the set of spin configurations where all spins have the same sign on each open cluster of $\omega$. The marginal $\mathscr{P}_{\Omega}$ is the random-cluster model $\phi_p$, which satisfies $\phi_p[\omega] \propto \abs{f(\omega)} \Bernoulli_p[\omega] = 2^{\kappa(\omega)}\Bernoulli_p[\omega]$, where $\kappa(\omega)$ is the number of open clusters of $\omega$ and thus, $2^{\kappa(\omega)}$ is the total number of $\pm$-colorings of $\omega$. The marginal $\mathscr{P}_{\Sigma}$ is the Ising model. The conditional measure $\mathscr{P}[\cdot | \omega]$ is the uniform coloring by $\pm$ of the open clusters of $\omega$, and as $f^{-1}(\sigma) = \{ \omega \in \Omega \mid \omega \subset S(\sigma)\}$, then $\mathscr{P}[\cdot | \sigma] = \Bernoulli_{S(\sigma)}[\cdot]$,  Bernoulli percolation on $S(\sigma)$.

\begin{table}[ht!]
\centering
\caption{Non-exhaustive list of couplings unified in \Cref{thm:generalcoupling}. For spin models we write $\sigma$ instead of $\eta$.  In all but five cases, $\gamma$ is the uniform measure on $\Sigma$. 
The notation $[q] = \{0, \dots, q-1\}$, and $\mathcal{C}(\omega)$ is the set of clusters of $\omega$.    $\mathcal{E}_{A}(\omega)$ denotes the set of subgraphs $F$ of $\omega$ with $\partial F=A$. In particular, $\mathcal{E}_{\emptyset}$ is the set of even graphs.
$\Bernoulli$ is Bernoulli percolation, $\phi$ is the random-cluster model, $\mu$ is the Ising model, $\ell$ the loop $\mathrm{O}(1)$ model, $\Prbcur$ the (traced) single random current measure, and $\Prbcur^{A,B}$ the double random current measure with sources $A$ and $B$. 
$\Sigma^{\downarrow}(\omega)= \{\eta \in \Sigma: \eta \subset \omega\}$. 
\(\UC_\omega\) and \(\UC^q_\omega\), is the uniform $\pm$ (respectively $q$) coloring of the clusters of $\omega$,  \(\UEG_\omega\) is the uniform even subgraph of $\omega$,  \(\mathtt{UG}^A_\omega\) is the uniform subgraph with sources $A$ of $\omega$,   $\mathtt{UG}_\omega[\cdot \mid \Sigma]$ is the uniform $\Sigma$ subgraph of $\omega$, $\mathtt{U}^q_\omega$ is a uniform divergence free coloring of the edges of $\omega$ with $q$ colors. 
The lattice gauge notation is introduced in \Cref{sec:Lattice_Gauge}. For the last two couplings, notation is explained in the links in the last column. Couplings which are either new or substantially extended in this paper are highlighted with blue.}
\label{tab:coupling_overview}
\resizebox{\textwidth}{!}{\begin{NiceTabular}{|c|c|c|c|c|c|c|c|}
\hline
$\Omega$ & $\Sigma$ & $\rho[\cdot \mid f^{-1}(\eta)]$ & $\gamma[\cdot \mid f(\omega)]$ & $f(\omega)$ 
  & $\mathscr{P}_{\Omega}$ & $\mathscr{P}_{\Sigma}$ & \textbf{Names/References} \\
\specialrule{1.5pt}{0pt}{0pt}
$\{0,1\}^E$ & $[q]^V$  
 & \( \Bernoulli_{S(\sigma)}\) 
 & \(\UC^q_\omega\) 
 & $[q]^{\mathcal{C}(\omega)}$
 & \(\phi\)
 & \(\mu\)
 & Swendsen–Wang 
   \cite{swendsen1987nonuniversal}, \cite{edwards1988generalization} \\ 
\hline
$\{0,1\}^E$ & $\Even$
 & \(\Bernoulli \cup \delta_{\eta}\) 
 & \(\UEG_\omega\) 
 & $\Even(\omega)$
 & \(\phi\)
 & $\ell$
 & Ising Loop-Cluster 
   \cite{evertz2002new, grimmett2007random}  \\ 
\hline

\rowcolor{blue!10}$\{0,1\}^E$ & $\{\pm1\}^V$
 & $\Prbcur_{S(\sigma)}^{\emptyset}$
 & \(\UC_\omega\) 
 & $[2]^{\mathcal{C}(\omega)}$
 & \(\dc\)
 & \(\XORI\)
 & XOR-DC, \cite{Lis},Prop.~\ref{prop:doublecurrent-Xor} \\ 
\hline
$\{0,1\}^E$ & $\Even$
 & \scalebox{0.8}{\(\ell_t \cup \Bernoulli_{t^2}\cup \delta_{\eta}\)} 
 & \(\UEG_\omega\) 
 & $\Even(\omega)$
 & \(\dc\)
 & $\ell$
 & Loop-Current 
\cite{DC16,klausen2021monotonicity} \\ 
\specialrule{1.5pt}{0pt}{0pt}
$\{0,1\}^E$ & $\{\pm 1\}^V$
 & \scalebox{0.8}{\(\Prb_{S(\sigma)}[\cdot \mid \mathcal{F}_A]\)} 
 & \(\mathtt{UC}_\omega\)  
 & $[2]^{\mathcal{C}(\omega)}$
 & \scalebox{0.8}{\(\phi[\cdot\mid\mathcal{F}_A]\)}
 & $\mu^A$
 &  SW w. sources, Prop.~\ref{prop:randomcluster-Ising} \\ 
\hline
$\{0,1\}^E$ &$\mathcal{E}_A$
 & \(\Prb \cup \delta_{\eta}\) 
 & \(\mathtt{UG}^A_\omega\)  
 & $\mathcal{E}_A(\omega)$
 & \scalebox{0.8}{\(\phi[\cdot\mid\mathcal{F}_A]\)}
 & $\ell^A$
 & LC w. sources \cite{aizenman2019emergent}, Prop.~\ref{prop:randomcluster-loop} \\ 
\hline

\rowcolor{blue!10} $\{0,1\}^E$ & $\{\pm 1\}^V$
 & \scalebox{0.8}{$\Prbcur_{S(\sigma)}^{A\triangle B}[\cdot\mid\mathcal{F}_A]$}
 & \(\mathtt{UC}_\omega\)  
 & $[2]^{\mathcal{C}(\omega)}$
 & $\Prbcur^{A,B}$
 & \(\mu^{A,B}\)
 & XOR-DC w. sources,  Prop.~\ref{prop:doublecurrent-Xor} \hspace{-5pt} \\ 
\hline

$\{0,1\}^E$ & $\mathcal{E}_A$
 & \scalebox{0.8}{$\ell^B \cup \Prb_{t^2} \cup \delta_{\eta}$}
 & \(\mathtt{UG}^A_\omega\)  
 & $\mathcal{E}_A(\omega)$
 & $\Prbcur^{A,B}$
 & $\ell^A$
 & L-Cur. w. sources, Prop.~\ref{prop:doublecurrent-loop} \\ 
\specialrule{1.5pt}{0pt}{0pt}

\rowcolor{blue!10}$\{0,1\}^E$ & \hspace{-10pt} \scalebox{.8}{\hspace{2pt}\tiny{any}}  \scalebox{.8}{$\Sigma \subset \Omega$}
 & \( \mathbb{P} \cup \delta_{\eta}\)
 & \scalebox{0.7}{$\Prb_{\omega,\frac{x}{p+x}}[\cdot \mid \Sigma]$}
 & $\Sigma^{\downarrow}(\omega)$
 & $\mathscr{P}_{\Sigma} \cup \mathbb{P}_p$
 & $\mathbb{P}_{\frac{x}{1+x}}[\cdot \mid \Sigma]$
 & Partial conditional, Prop.~\ref{prop:partial_coupling} \\ 
\hline

\rowcolor{blue!10}$\{0,1\}^E$ & \hspace{-10pt} \scalebox{.8}{\hspace{2pt}\tiny{any}}  \scalebox{.8}{$\Sigma \subset \Omega$}
 & \( \mathbb{P} \cup \delta_{\eta} \)
 & $\mathtt{UG}_\omega[\cdot \mid \Sigma]$
 & $\Sigma^{\downarrow}(\omega)$
 & $\mathscr{P}_{\Sigma} \cup \mathbb{P}_x$
 & $\mathbb{P}_{\frac{x}{1+x}}[\cdot\mid \Sigma]$
 & Conditional perc.\ Cor.~\ref{prop:conditioned_Bernoulli1} \\ 
\hline

\rowcolor{blue!10}$\{0,1\}^E$ & \hspace{-10pt} \scalebox{.8}{\hspace{2pt}\tiny{any}}  \scalebox{.8}{$\Sigma \subset \Omega$}
 & \( \mathbb{P} \cap \delta_{\eta}\)
 & \scalebox{0.6}{$\Prb_{\frac{x(1-p)}{1+x(1-p)}}[ \cdot  \mid \Sigma^{\uparrow}(\omega)]$}
 & $\Sigma^{\uparrow}(\omega)$
 & $\mathscr{P}_{\Sigma} \cap \mathbb{P}_{p}$
 & $\mathbb{P}_{\frac{x}{1+x}}[\cdot \mid \Sigma]$
 & Partial upward, Prop.~\ref{prop:conditioned_Bernoulli2} \\ 
\specialrule{1.5pt}{0pt}{0pt}

$\{0,1\}^E$ & $\ker(\partial)$ 
 & \(\Bernoulli \cup \delta_{\eta} \) 
 & $\mathtt{U}^q_\omega$
 & $\ker(\partial^\omega)$
 & $\phi^q$
 & $\ell^q$
 & Loop-Cluster \cite{zhang2020loop}, Prop.~\ref{prop:uniform flows} \\ 
\hline

$\{0,1\}^{\Lambda_k}$ & \hspace{-8pt} $C_{k-1}$
 & $\mathbb{P}_{S(\sigma)}$
 &  $\mathtt{UC}^q_{\omega,k}$
 & $\ker(d_{k}^{\omega})$
 & $\phi^q_{\Lambda_k}$
 & $\mu_{\Lambda_k}$
 & Lat. Gauge SW, \cite{ben1990critical,duncanPRCM1,HS16} \\ 
\hline

\rowcolor{blue!10}
$\{0,1\}^{\Lambda_k}$ & $\ker (\partial_k)$ 
 & $\mathbb{P} \cup \delta_{\textrm{supp}(\eta)}$
 &  $\mathtt{U}^q_{\omega,k}$
 & $\ker(\partial_k^{\omega})$
 & $\phi^q_{\Lambda_k}$
 & $\ell^q_{\Lambda_k}$
 & Lat. Gauge LC, Prop.~\ref{prop:lattice_potts} \\ 
\specialrule{1.5pt}{0pt}{0pt}
\small{$\textrm{Im}(\partial_{d-k+1})$} & $C^{k-1}$
 & $\delta_{(d^k(\sigma))^*}$
 & $\mathtt{U}_{(*\circ d^k)^{-1}(\omega)}$
 & \small{$(*\circ d^k)^{-1}(\omega)$}
 & $\mu_{\Lambda_{d-k}}$
 & $\mu^*_{\Lambda^{k-1}}$
 & Spin model/domain wall, \ref{prop:KW} \\ 
\hline

$\Even$ & $\Even$
 & \small{$\ell_{\eta^c, n-m}\cup \delta_{\eta}$} %\scalebox{0.8}{$(n-m)^{\kappa(\omega)}x^{\abs{\omega}}$}
 & $\mathbb{P}_{\mathcal{C}(\omega),\frac{m}{n}}$ %\scalebox{0.8}{$\bigl(\tfrac{m}{n-m}\bigr)^{\kappa(\eta)}$}
 & $\Even(\omega)$ 
 & $\ell_{n}$
 & \tiny{ $\propto Z_{\eta^c,n-m}\cdot \textrm{d}\ell_{m}$}
 &\small{Loop $\mathrm{O}(n)$ 1+1=2 } \cite{Glazman2021Log}, \ref{sec:loopO(n)} \hspace{-3pt} \\
\hline
\end{NiceTabular}}
\end{table}

 The proof of Theorem \ref{thm:generalcoupling} is a short calculation:
\begin{proof}
	The marginal measure on $\Sigma$ is
	\[\mathscr{P}_{\Sigma}[\eta]=\sum_{\omega\in\Omega} \mathscr{P}[\omega,\eta]\propto \sum_{\omega\in\Omega} \rho[\omega] \gamma[\eta] \id[\eta\in f(\omega)]= \gamma[\eta] \rho[\{\omega\in\Omega \mid \eta\in f(\omega)\}]=\gamma[\eta] \rho[f^{-1}(\eta)].\]
	Similarly, the marginal measure on $\Omega$ is
	\[\mathscr{P}_{\Omega}[\omega]=\sum_{\eta\in \Sigma} \mathscr{P}[\omega,\eta] \propto \sum_{\eta\in \Sigma} \rho[\omega] \gamma[\eta] \id[\eta\in f(\omega)]=\rho[\omega] \gamma[\{\eta \in \Sigma \mid \eta \in  f(\omega)\}]=\rho[\omega] \gamma[f(\omega)].\]
	Furthermore, the conditional measures are straightforwardly calculated: 
	\[\mathscr{P}[\eta \mid \omega]=\frac{\mathscr{P}[\omega,\eta]}{\mathscr{P}_{\Omega}[\omega]} \propto \frac{\rho[\omega] \gamma[\eta] \id[\eta\in f(\omega)]}{\rho[\omega] \gamma[f(\omega)]}=\gamma[\eta \mid f(\omega)],\]
    \[ \mathscr{P}[\omega \mid \eta]=\frac{\mathscr{P}[\omega, \eta]}{\mathscr{P}_{\Sigma}[\eta]} \propto \frac{\rho[\omega] \gamma[\eta] \id[\eta\in f(\omega)]}{\gamma[\eta] \rho[f^{-1}(\eta)]}=\rho[\omega \mid f^{-1}(\eta)],\]
which was what we wanted.
\end{proof}
%\tableofcontents 

%The structure of the paper goes as follows: In \Cref{sec:Ising_sec}, generalized versions of the couplings of the Ising model in \Cref{fig:couplings_horizontal} are stated (see also the generalization to the case with sources in \Cref{fig:couplings_with_sources}). 
%See \Cref{sec:Ising_sec} for details on the novelty. 

The structure of the paper goes as follows: In \Cref{sec:Ising_sec}, generalized versions of the couplings of the Ising model in \Cref{fig:couplings_horizontal} involving the double random current are stated (see also the generalization to the case with sources in \Cref{fig:couplings_with_sources}). Part (b) of \Cref{prop:doublecurrent-Xor} is new and so is \Cref{prop:doublecurrent-loop} with $A\neq \emptyset$ and $B \neq \emptyset$, see \Cref{sec:Ising_sec} for further details. 
In \Cref{sec:Bernoulli}, we discuss some new couplings of general conditioned Bernoulli percolation. In \Cref{sec:Potts}, we show how the loop representation from \cite{zhang2020loop} encodes the phase transition of the Potts model and an extension of the coupling to lattice gauge theories is given in \Cref{sec:Lattice_Gauge} with the new lattice gauge loop cluster coupling in \Cref{prop:lattice_potts} as a highlight. Finally, some open problems are discussed in \Cref{sec:questions}. Some additional proofs, details and other well-known couplings that fit into \Cref{thm:generalcoupling} are postponed to the appendices.

There is an overview of some of the couplings unified by \Cref{thm:generalcoupling} in \Cref{tab:coupling_overview} - including the novel ones from this paper (in blue). Through \Cref{thm:generalcoupling}, the table highlights the common structure of the couplings, which can hopefully both help put them to use and aid in the derivation of new couplings. 

In \Cref{fig:couplings_horizontal} and \Cref{tab:coupling_overview}, we used the union and intersection of two probability measures   $\pi$ and $\nu$ defined on the same space $\{0,1\}^E$, their union $\pi \cup \nu$ (intersection $\pi \cap \nu$, respectively) is defined to be the union (intersection, respectively) of open edges from independent samples of $\pi$ and $\nu$.

Two recurring themes in the application of \Cref{thm:generalcoupling} which might not be apparent from the theorem statement go as follows: In both, $\Omega$ is a set of graphs and the measure $\rho = \mathbb{P}$ is Bernoulli percolation. First, if $f^{-1}(\eta)=\{\omega \in \Omega \mid \omega \subset S(\eta) \}$ for some function $S: \Sigma \to \Omega,$ then the conditional measure $\mathscr{P}[\;\cdot \mid \eta]$ may be written $\mathbb{P}_{S(\eta)}[\cdot]$ - that is, Bernoulli percolation on the set $S(\eta)$.
In the second case, $\Sigma \subset \Omega$ and $f^{-1}(\eta) = \{\omega \in \Omega \mid \omega \supset \eta \}$, then $\mathscr{P}[\;\cdot \mid \eta] =\mathbb{P}  \cup \delta_{\eta}$, 
where $\delta_{\eta}$ is the Dirac measure on $\eta$. By the law of total probability, $\mathscr{P}_{\Omega} = \Bernoulli \cup \mathscr{P}_{\Sigma}$, i.e. $\mathscr{P}_{\Omega}$ is obtained by Bernoulli-sprinkling $\mathscr{P}_{\Sigma}$. Both cases appear several times through the paper: See, e.g., \Cref{fig:couplings_horizontal} and \Cref{tab:coupling_overview}. The above observations are further elaborated upon in \Cref{sec:Ising_proofs}.

\section{Graphical representations of the Ising model}\label{sec:Ising_sec} 
We first recall the definitions of the Ising model and its graphical representations - see \cite{DC17} for a more complete introduction. 
Let \( G = (V, E) \) be a finite undirected graph with pairwise interactions \( \{ J_e > 0 \mid e \in E \} \). 
The Hamiltonian of the Ising model is defined through 
$$
H(\sigma) = - \sum_{uv \in E} J_{uv} \sigma_u \sigma_v. 
$$
The Ising probability measure, $\Ising_{G,J}$, is a measure on $\{-1,1\}^V$ defined  by 
\begin{align*}
	\Ising_{G,J}[\sigma] = \frac{\exp\left[- H(\sigma)\right]}{Z_{G,J}},
\end{align*}
where we have absorbed the temperature into \( \{ J_e > 0 \mid e \in E \} \) and $Z_{G,J}$ is the partition function
$$
Z_{G,J} = \sum_{\sigma \in \{-1,1\}^V} \exp\left[- H(\sigma)\right].
$$

For $\omega\subset E,$ let $\partial \omega$ denote the set of $v\in V$ with odd degree in the spanning subgraph $(V,\omega)$ (see \Cref{sec:Potts} for a more general definition). We refer to $\partial \omega$ as the \textit{sources} of $\omega$.
%For a spanning subgraph $(V,F)\subset (V,E)$, let $\partial F $ be the set $v\in V$ with odd degree in $F$ (see  \Cref{sec:Potts} for a more general definition).
Let $\mathcal{E}_A(G) = \{ \omega \subset E \mid \partial \omega = A\}$ be the graphs with sources $A$. The sourceless configurations $\mathcal{E}_\emptyset(G)$ will be referred to as the even subgraphs of $G$. Let
\begin{align}\label{eq:FA}
\mathcal{F}_A = \{\omega \subset E \mid \exists F \subset \omega, \partial F = A \} 
\end{align}
denote the event consisting of all subgraphs that have a subgraph with sources $A$ or equivalently  the set of bond configurations $\omega$ for which each open cluster of $\omega$ intersects $A$ an even number of times (could be $0$). For any $\gamma$ with $\partial \gamma = A,$ considering the bijection $F \mapsto F \triangle \gamma$ from $\mathcal{E}_\emptyset(\omega)$  to $\mathcal{E}_A(\omega)$ proves the useful switching principle which appeared in \cite{griffiths1970concavity}:
\begin{align}\label{eq:switching_principle}
\abs{\mathcal{E}_A(\omega)} = \id[\omega \in \mathcal{F}_A]\abs{\mathcal{E}_\emptyset(\omega)}
\end{align} 
Let us recall the high-temperature expansion of the Ising model, 
\begin{align*}
	Z_{G,J} &:= \sum_{\sigma} \exp \left( \sum_{uv \in E}  J_{uv} \sigma_u \sigma_v \right)
	= \sum_{\sigma} \prod_{uv \in E} (\cosh( J_{uv}) + \sigma_u \sigma_v \sinh(J_{uv}))\\
	&= \left(\prod_{e \in E} \cosh(J_{e})\right) \sum_{\sigma} \prod_{uv \in E} (1 + \sigma_u \sigma_v \tanh(J_{uv}))
	= \prod_{e \in E} \cosh(J_{e}) \, 2^{|V|} \hspace{-10pt}\sum_{F \subset E, \partial F = \emptyset}\hspace{-2pt} \prod_{e \in F} \tanh(J_{e}).
\end{align*}

Since the expression $\tanh(J_e)$ will appear repeatedly, we follow the convention from \cite{Lis} (but use $t$ instead of $x$ as a shorthand for $\tanh$) and let $t_e= \tanh(J_e)$.

The expression for the partition function motivates the definition of the high-temperature expansion (also known as the loop $\mathrm{O}(1)$) probability measure $\ell_{G,t}$ on $\{0,1\}^E$ defined by 
\begin{align}\label{eq:loopO(1)def}
	\ell_{G,t}[\omega] \propto   \prod_{e \in \omega}t_e 
\id[\partial \omega = \emptyset]. 
\end{align}
%where $\partial \omega$ is the set of vertices with odd degree. 
For $t_e>1$, the definition also makes sense and will be related to the anti-ferromagnetic Ising model in \Cref{sec:anti_ferromagnetic_Ising}. 
More generally, for a set of sources $A \subset V$ (with $\abs{A}$ even) define $\ell^A$ on $\{0,1\}^E$ 
%(sometimes, naturally, $\mathcal{E}_A$) 
through  
\begin{align*}
	\ell_{G,t}^A[\omega] \propto  \prod_{e \in \omega}t_e \id[\partial \omega = A].%,~\forall \omega \subset \
\end{align*}
We will often consider ourselves free to consider $\ell_{G,t}^A$ as a measure on its support, $\mathcal{E}_A(G),$ instead.

We will be concerned with the two-parameter family $\ell^A_{G,t} \cup \mathbb{P}_{G,p}$, that is, the loop $\mathrm{O}(1)$ model with parameter $t=(t_e)_{e\in E}$ union Bernoulli percolation with parameter $p=(p_e)_{e\in E}$ - see \Cref{sec:l1Bf} for more details on this family. 
The random current model can be defined by taking independent $\text{Poisson}(J_e)$-multi-edges on each edge $e\in E$ conditioned on the degree of every vertex being even. More generally, for the single random current with sources $A\subset V,$ one has to condition on the set of vertices with odd degree to be $A$. We are mostly interested in the open (i.e., nonzero) edges from random current models. We write $\Prbcur^{A}_{G,J}$ for the push forward of the random current with sources  $A$  under the map which sends all positive integers to $1$ and $0$ to $0$. One may check (see, e.g., \cite[Theorem 3.2]{aizenman2019emergent}) that
\begin{equation}\label{eq:definition_single_current}
    \Prbcur^A_{G,J}=\ell^A_{G,t} \cup \Prb_{G,1-\sqrt{1-t^2}},
\end{equation}
 where $1-\sqrt{1-t^2}$ is the vector $(1-\sqrt{1-t_e^2})_{e\in E}$.
The double random current with sources $A$ and $B$ is the union of two independent single random currents
\begin{align}\label{eq:definition_double_current}
	\Prbcur_{G,J}^{A,B} :=\Prbcur_{G,J}^{A} \cup \Prbcur_{G,J}^{B}= \loopmeasure_{G,t}^A \cup \loopmeasure_{G,t}^B \cup \Bernoulli_{G,t^2}.
\end{align}

The ($q=2$) random-cluster model, $\phi_{G,p}$, can also be obtained as $\ell_{G,t} \cup \mathbb{P}_{G,t}$, since 
\begin{align}\label{eq:definition_random_cluster}
\begin{split}
\ell_{G,t} \cup \mathbb{P}_{G,t}[\omega] &\propto \sum_{F\in \Even(\omega)} \prod_{e \in F}t_e   \prod_{e \in \omega \setminus F} t_e  \prod_{e \in E \setminus \omega}(1-t_e)  =   \prod_{e \in \omega} t_e \prod_{e \in E \setminus \omega}  (1-t_e) \sum_{F\in \Even(\omega)} 1 \\
&%=  \prod_{e \in \omega} t_e \prod_{e \in E \setminus \omega}  (1-t_e) \abs{\Even(\omega)}
\propto 2^{\kappa(\omega)}\prod_{e \in \omega} p_e \prod_{e \in E \setminus \omega} (1-p_e) \propto \phi_p[\omega],      
\end{split}
\end{align}
where  $p=(p_e)_{e\in E}:=(1-e^{-2J_e})_{e\in E}$, and we used the well known formula for the cyclomatic number (see also \Cref{lemma:Counting_divergence})
\begin{align}\label{eq:even_formula}
	\abs{\Even(\omega)} = 2^{\kappa(\omega) + \abs{\omega} - \abs{V}}.
\end{align}
%which will be generalized to $\Z/ q\Z$-valued directed graphs in \Cref{lemma:Counting_divergence}. 

\subsubsection{Notation}
We always fix a finite graph $G=(V,E)$ with coupling constants $(J_e)_{e\in E}$ and consider various spin-type measures (i.e., defined on $S^V$ for some finite set $S$) or bond-type measures (i.e., defined on $\{0,1\}^E$). Following the convention in the literature, random-cluster measures are parametrized by $p=(p_e)_{e\in E}$ with $p_e:=1-e^{-2J_e}$, and loop O(1) measures are parametrized by $t=(t_e)_{e\in E}$ with $t_e:=\tanh(J_e)$; all other measures related to the Ising model are parametrized by $J=(J_e)_{e\in E}$ directly. From now on, we drop the dependence on $G$ and $J$ (or $p$, or $t$) when these are clear from the context. If a measure is defined on a proper subgraph of $G$ and/or the couplings are different from $(J_e)_{e\in E}$, we write the subgraph and corresponding parameter explicitly. We do not distinguish between a subset of edges $F\subset E$ and the spanning subgraph $(V,F)$ it induces in $G$. Each $\omega\in\{0,1\}^E$ is identified with $\omega^{-1}(\{1\})\subset E$.  If $F_1\subset F_2 \subset E$, we may view $F_2\setminus F_1$ as the subgraph $(V,F_2\setminus F_1)$.

\subsection{Ising coupling measures}

\begin{figure}[ht] {\begin{tikzcd}
    \Ising_{J}^A \arrow[d,bend left, dotted, "\Prb_{S(\sigma),p}\lbrack \cdot \mid \mathcal{F}_A\rbrack"] & & \mu_{J}^{A,B} \arrow[d,bend left, dotted, "\Prbcur_{S(\sigma),2J}^{A \triangle B}\lbrack\cdot \mid \mathcal{F}_A\rbrack"{yshift=-4pt}] \\
    \phi_p[\;\cdot \mid \mathcal{F}_A] \arrow[rd, dashed, "\mathtt{UG}^A", swap] \arrow[u, dashed,bend left, "\mathtt{UC}"]  & \Prbcur_J^{A} \arrow[r] \arrow[l] & \Prbcur_J^{A,B} \arrow[ld, dashed, "\mathtt{UG}^A"] \arrow[u,bend left, dashed, "\mathtt{UC}"] &  \hspace{-1.2cm} = \Prbcur_J^{A \triangle B,\emptyset}[\;\cdot \mid \mathcal{F}_A] \\
                             & \ell_t^A \arrow[u]        &                                  \\
\end{tikzcd}}
{\caption{Generalized coupling measures between the graphical representations of the Ising model in \Cref{fig:couplings_horizontal} to include sources $A \subset V$. The measures $\mu_J^A$ and $\mu_J^{A,B}$ have, to our knowledge, not been studied in the literature before, but they reduce to the Ising model and XOR model when $A=\emptyset$ and $B =\emptyset$. $\mu^A[\sigma] \propto \Bernoulli_{S(\sigma),p}[\mathcal{F}_A] \mu[\sigma]$ and 
$
 \mu^{A,B}[\sigma] \propto \langle \tau_{A \triangle B} \rangle_{S(\sigma),2J} \Prbcur^{A\triangle B}_{S(\sigma),2J}[\mathcal{F}_A]  \XORI[\sigma]
$. $\mathtt{UC}$ is uniform $\pm$-coloring of clusters. 
$\mathtt{UG}^A$ is the uniform graph with sources $A$. We exhibit the couplings as a special case of \Cref{thm:generalcoupling} in Propositions 
\ref{prop:doublecurrent-Xor}, \ref{prop:doublecurrent-loop}, \ref{prop:randomcluster-Ising}, \ref{prop:randomcluster-loop}. 
\label{fig:couplings_with_sources}}}
\end{figure}
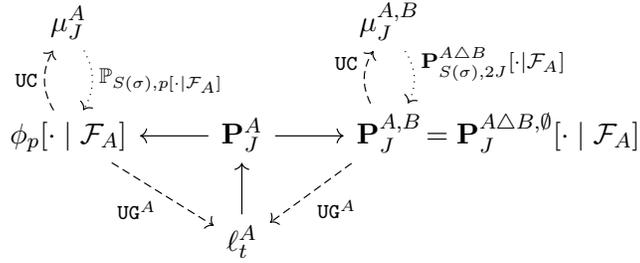

In this section, we state two new couplings for graphical representations of the Ising model which are consequences of Theorem \ref{thm:generalcoupling}.    Their proofs are given in \Cref{sec:Ising_proofs} along with two more couplings. See \Cref{fig:couplings_with_sources} for an illustration of the four couplings.
All the known coupling measures relating the Ising model, random-cluster model, loop $\mathrm{O}(1)$ model, double random current model and XOR-Ising model are special cases of \Cref{thm:generalcoupling}.
Our first application concerns the double random current model and the XOR Ising model.

\begin{proposition}\label{prop:doublecurrent-Xor}
	Let $G=(V,E)$ be a finite connected graph. Let $\Omega:=\{0,1\}^E$ be the space of bond configurations and $\Sigma:=\{-1,+1\}^V$ be the space of spin configurations. Let $f:\Omega\to 2^{\Sigma}$ be given by $f(\omega)=\{\sigma\in \Sigma\mid \omega\subset S(\sigma)\}$ (so $|f(\omega)|=2^{\kappa(\omega)}$). For any $A, B\subset V$ with $|A|, |B|$ even, let
	\[\rho[\omega] \propto \Prbcur^{A\triangle B}_{2J}[\omega \mid \mathcal{F}_A],~\forall \omega\in\Omega,~\gamma[\sigma]\propto 1, ~\forall \sigma\in \Sigma.\]
	Then, under the coupling measure $\mathscr{P}$ defined in Theorem \ref{thm:generalcoupling}, 
	\begin{enumerate}
		\item[(a)] The marginal of $\mathscr{P}$ on $\Sigma$ is $\mu^{A,B}$ defined by
		\[\mu^{A,B}[\sigma] \propto \langle \tau_{A \triangle B} \rangle_{S(\sigma),2J} \Prbcur^{A\triangle B}_{S(\sigma),2J}[\mathcal{F}_A]  \XORI[\sigma],\]
		where $ \langle \tau_{A \triangle B} \rangle_{S(\sigma),2J}$ is the expectation of the product of Ising spins  $\tau_{A \triangle B} = \prod_{x\in A \triangle B}\tau_x$, $\tau_x \in \{-1,1\}$  under the Ising measure defined on $S(\sigma)$ with couplings $(2J_e)_{e\in E}$. For each $\omega\in\Omega$ with  $\omega\in\mathcal{F}_A \cap \mathcal{F}_B$, the conditional measure $\mathscr{P}[\;\cdot \mid \omega]$ on $\Sigma$ is realized by tossing independent fair coins to get $\pm$ for each open cluster of $\omega$.
		\item[(b)] The marginal of $\mathscr{P}$ on $\Omega$ is $\Prbcur^{A,B}$.  For each $\sigma\in\Sigma$ with $\mu^{A,B}[\sigma] \neq 0$, the conditional measure $\mathscr{P}[\;\cdot \mid \sigma] = \Prbcur^{A\triangle B}_{S(\sigma),2J}[\;\cdot \mid \mathcal{F}_A]$, the single random current measure on $S(\sigma)$ with sources $A\triangle B$ and couplings $2J=(2J_e)_{e\in E}$ conditioned on $\mathcal{F}_A$.% ; so each edge in $U(\sigma)$ is closed.
	\end{enumerate}
\end{proposition}

When $A=B=\emptyset$, part (a) of Proposition \ref{prop:doublecurrent-Xor} says that one can obtain a XOR-Ising configuration by independently assigning $\pm 1$ to the open clusters of a double random current measure; this is proved in \cite[Corollary 3.11 and Remark 3.12]{Lis} when $G$ is a planar graph (see also \cite[Corollary 3.3]{duminil2021conformal}). Part (b) of Proposition \ref{prop:doublecurrent-Xor} is new; Part (a) with $A\neq \emptyset$ and/or $B\neq \emptyset$ is also new. 

Let $\Prbcur^{\emptyset}_{\mathbb{Z}^d,\beta}$ (respectively, $\Prbcur^{\emptyset,\emptyset}_{\mathbb{Z}^d,\beta}$) be the single (respectively, double) current defined on $\mathbb{Z}^d$ with coupling constants $J_e=\beta$ for each nearest neighbor edge $e$ in $\mathbb{Z}^d$. Proposition \ref{prop:doublecurrent-Xor} together with the law of total probability imply that $$\Prbcur^{\emptyset,\emptyset}_{\mathbb{Z}^d,\beta}[x\longleftrightarrow y] = \XORI[\sigma_x \sigma_y] =\langle \sigma_x\sigma_y\rangle_{\mathbb{Z}^d,\beta}^2,$$ as is well-known from the switching lemma. This proof only uses the switching principle \eqref{eq:switching_principle} on the traced graph and not the multi-graph. One straightforward conclusion from this is that $\Prbcur^{\emptyset,\emptyset}_{\mathbb{Z}^d,\beta}$ percolates if and only if $\beta>\beta_c(d)$ where $\beta_c(d)$ is the critical inverse temperature. Curiously, part (b) of Proposition~\ref{prop:doublecurrent-Xor}  suggests that $\Prbcur^{\emptyset}_{\mathbb{Z}^d,\beta}$ percolates whenever $\beta>2\beta_c(d)$. More formally, if percolation of the sourceless single random current is monotone in the domain for subgraphs\footnote{Percolation of the single random current is in general not monotone in the domain. For a counterexample, for any constant $\beta,$ one may amend the construction in  
\cite{hansen2024nonuniquenessphasetransitionsgraphical} to find (countable, locally finite) graphs $\mathbb{G}' \subset \mathbb{G}$ such that the single current percolates on $\mathbb{G}'$, but not on $\mathbb{G}$. It also follows that there exists a finite graph $G$ and a set of sources $A$ such that $\Prbcur^{A}[v \cc w] < \Prbcur^{\emptyset}[v \cc w]$.} of $\mathbb{Z}^d$, then $\Prbcur^{\emptyset}_{\mathbb{Z}^d,\beta}$  percolates if $\beta>2\beta_c(d)$. See \Cref{sec:questions} for a discussion of related questions. 
        
The next application of Theorem \ref{thm:generalcoupling} is about uniform subgraphs:
\begin{proposition}\label{prop:doublecurrent-loop}
	Let $G=(V,E)$ be a finite connected graph and fix     $A\subset V$ with $|A|$ even, $\Omega:=\{0,1\}^E$ and $\Sigma:=\{F\subset E \mid \partial F=A\}$ be the set of subgraphs of $G$ with sources $A$. Let $f:\Omega\to 2^{\Sigma}$ be defined by $f(\omega)=\{F\subset \omega \mid \partial F=A\}$ and
	\[\rho[\omega]\propto \id[\omega\in\mathcal{F}_A] (\ell^{A \triangle B} \cup \Prb_{t^2})[\omega],~\forall \omega\in \Omega, ~\gamma[\eta]\propto 1, ~\forall \eta\in \Sigma.\]
	Then, under the coupling measure $\mathscr{P}$ defined in Theorem \ref{thm:generalcoupling}, 
\begin{enumerate}
		\item[(a)] The marginal of $\mathscr{P}$ on $\Sigma$ is $\ell^{A}$. For each $\omega\in\Omega$ with $\omega\in\mathcal{F}_A \cap \mathcal{F}_B$, the conditional measure is given by  $\mathscr{P}[\;\cdot \mid \omega] = \mathtt{UG}^A[\omega]$, the uniform $\omega$-subgraph with sources $A$.
		\item[(b)] The marginal of $\mathscr{P}$ on $\Omega$ is $\Prbcur^{A,B}$.  For each $\eta\in\Sigma$, the conditional measure $\mathscr{P}[\;\cdot \mid \eta]$ is $\ell^B \cup \Prb_{t^2} \cup \delta_{\eta}$ where $\delta_{\eta}$ is the Dirac measure concentrated at $\eta$.
	\end{enumerate}
In particular, the uniform graph with sources $A$ of the double random current measure with sources $A$ and $B$ is the loop $\mathrm{O}(1)$ model with sources $A$. 
\end{proposition}
Proposition \ref{prop:doublecurrent-loop} gives a coupling between the loop $\mathrm{O}(1)$ model and the double random current model which appeared implicitly in the proof of \cite[Thm 3.2]{lis2017planar} for $A=\emptyset, B \neq \emptyset$. For $A=\emptyset, B = \emptyset$  one direction is proven in \cite[Remark 3.4]{DC16} and the other in \cite[Theorem 4.1]{klausen2021monotonicity}. 
%When $A=B=\emptyset$, Proposition \ref{prop:doublecurrent-loop} gives a coupling between the loop $\mathrm{O}(1)$ model and the double random current model; one direction is proven in \cite[Remark 3.4]{DC16} and the other is implicit  in the proof of \cite[Thm 3.2]{lis2017planar} and proved in \cite[Theorem 4.1]{klausen2021monotonicity}. 
Proposition \ref{prop:doublecurrent-loop} with $A \neq \emptyset$ and $B \neq \emptyset$ is new.
The attentive reader with \Cref{fig:couplings_with_sources} in mind could ponder whether the UEG of $\phi[\;\cdot \mid \mathcal{F}_A]$ is the loop $\mathrm{O}(1)$ model and will quickly be convinced that it is not the case. 

\begin{remark}
    A peculiar consequence of Proposition \ref{prop:doublecurrent-loop} is that the uniform even subgraph of $\Prbcur^{\emptyset,B}$ is the loop $\mathrm{O}(1)$ model regardless of $B$. 
 This again implies that the probability that $\Prbcur_{\beta}^{\emptyset,B}$ has a path "wrapping around" the torus an odd number of times is independent of $B$, using the arguments of \cite{hansen2023uniform}, where \eqref{eq:coupling_use} in the proof of  \Cref{thm:q_state_potts_torus_trick} becomes an equality when considering paths of all odd windings instead of just the simple ones. 
Thus, each of the  measures $\Prbcur^{\emptyset,B}$ encodes the phase transition of the Ising model.  More generally, for any non-trivial finite graph, the convex set %$\mathcal{U}$ 
of measures that have $\ell$ as their UEG has more than 2 extremal points (cf. \Cref{lem:non_trivial_convex_combination_of_double_current}). 
%has at least $2^{\abs{V}}-1$ elements (cf. \Cref{lem:non_trivial_convex_combination_of_double_current})  %(which we conjecture to be the convex dimension of $\mathcal{U}$ cf. Conjecture \ref{prop:convex_UEG_is_loop_O(1)})
This is be contrasted with the operation of adding or deleting Bernoulli edges where there is at most one such measure  (cf. Claim \ref{claim:Bernoulli_uniqueness}). We conjecture that the convex set of measures that have $\ell_{\mathbb{Z}^d, t}$ as their UEG has uncountably many extreme points for any $d\geq 2, t \in (0,1)$. 
\end{remark} 

%Since the double random current $\Prbcur^{\emptyset,xy}$ represents the truncated correlations of the Ising model \cite{aizenman1982geometric}, we were motivated by a hope that \Cref{prop:doublecurrent-loop}  could be relevant to gain further understanding of the near-critical two-dimensional Ising model in a magnetic field, expanding the results in \cite{camia2020gaussian} by combining methods of \cite{camia2020exponential, klausen2022mass}.  \Cref{prop:doublecurrent-loop} exhibits more of the rich interplay between $\Prbcur^{\emptyset,xy}$ and $\phi[\;\cdot \mid x \cc y]$, which was utilized to study the FK-Ising IIC in \cite{panis2024incipient} and the near-critical Ising model in \cite{klausen2022mass}.

\section{Coupling measure for conditional Bernoulli percolation}\label{sec:Bernoulli}
Several models appear in the literature which admit interpretations as conditioned versions of Bernoulli percolation. Indeed, the loop O($1$) measure may be viewed as Bernoulli percolation conditioned to output an even graph. Other natural examples include the arboreal gas model \cite{Bauerschmidt2024PercoForest,Bauerschmidt2021PlaneForest,CaraccioloForests} and $H$-free random graphs \cite{BaloghKrFree,GerkeNoK4}, where the conditioning is on outputting, respectively, a forest and a subgraph of the complete graph containing no copies of a fixed graph $H$. For instance, triangle-free graphs can be proven to be bi-partite with high probability \cite{OsthusNoTriangles,PromelNoTriangles}. Less intuitive examples include the uniform spanning tree \cite{KenyonUST,PemantleUST} and self-avoiding walk models \cite{DCSAW,MadrasBook}.
The following coupling works for all such models:
\begin{proposition}\label{prop:partial_coupling}
Let $G=(V,E)$ be a finite graph, $p_e\in(0,1), x_e>0$ for each $e\in E$. Let $\Omega = \{0,1\}^E$, $\Sigma \subset \Omega$ and $f:\Omega\to 2^{\Sigma}$ be defined by $f(\omega)= \Sigma^{\downarrow}(\omega):=\{\eta \in \Sigma\mid \eta \subset \omega\}$ and
	\[\rho[\omega] = \Prb_{p}[\omega], ~\forall \omega\in \Omega, ~\gamma[\eta]= \Prb_{\frac{x}{p+x}}[\eta \mid \Sigma]\propto 
   \displaystyle \prod_{e \in \eta} \frac{x_e}{p_e} , ~\forall \eta\in \Sigma. \]
   	Then, under the coupling measure $\mathscr{P}$ defined in Theorem \ref{thm:generalcoupling}, 
    \begin{enumerate}
        \item[(a)] The marginal of $\mathscr{P}$ on $\Sigma$ is $\mathscr{P}_{\Sigma}[\eta] \propto \prod_{e\in \eta} x_e$. 
        For each $\omega\in\Omega$ with $\mathscr{P}_\Omega[\omega] \neq 0$, the  conditional measure $\mathscr{P}[\;\cdot \mid \omega]% = \mathbb{P}_{\frac{x}{p+x}}[\cdot \mid f(\omega)] 
        = \mathbb{P}_{\frac{x}{p+x}}[\;\cdot \mid \Sigma^{\downarrow}(\omega)]$. 
        \item[(b)] The marginal of $\mathscr{P}$ on $\Omega$ is $\mathscr{P}_\Omega= \mathscr{P}_{\Sigma} \cup \mathbb{P}_p$.
        For each $\eta\in\Sigma$ with $\mathscr{P}_{\Sigma}[\eta]\neq 0$, the conditional measure $\mathscr{P}[\;\cdot \mid \eta] = \mathbb{P}_p \cup \delta_{\eta}$, Bernoulli percolation on $E$ with parameter $(p_e)_{e\in E}$ conditioned on all edges in $\eta$ being open.
    \end{enumerate}            
\end{proposition}
\begin{proof}
The $\Omega$ marginal is by \Cref{thm:generalcoupling}, 
%By Theorem \ref{thm:generalcoupling}, we have
    \begin{align*}
        \mathscr{P}_\Omega[\omega] &\propto \gamma[f(\omega)] \rho[\omega] \propto \sum_{\eta \in \Sigma^{\downarrow}(\omega)}\gamma[\eta] \Prb_p[\omega] \propto \sum_{\eta \in \Sigma^{\downarrow}(\omega)} \prod_{e \in \eta} \frac{x_e}{p_e} \prod_{e\in \omega} p_e \prod_{e\in E \setminus \omega}(1-p_e)\\
        &\propto \sum_{\eta \in \Sigma^{\downarrow}(\omega)} \prod_{e\in \eta} x_e \prod_{e\in \omega \setminus \eta} p_e \prod_{e\in E \setminus \omega}(1-p_e) \propto (\mathscr{P}_{\Sigma} \cup \mathbb{P}_p)[\omega]. 
    \end{align*}  
The remaining part of the proof also follows from \Cref{thm:generalcoupling}. 
\end{proof}
Since we take the case $p_e= x_e \in (0,1)$ for each $e\in E$ to be the most useful, we state it separately.
\begin{corollary}
    \label{prop:conditioned_Bernoulli1}
Let $G=(V,E)$ be a finite graph, $x_e\in (0,1)$ for each $e\in E$. Let $\Omega = \{0,1\}^E$, $\Sigma \subset \Omega$ , $f:\Omega\to 2^{\Sigma}$ be defined by $f(\omega)= \Sigma^{\downarrow}(\omega):=\{\eta \in \Sigma: \eta \subset \omega\}$ as well as
	$$\rho[\omega] \propto \Prb_{x}[\omega], ~\forall \omega\in \Omega, ~\gamma[\eta]\propto 1, ~\forall \eta\in \Sigma. $$
    	Then, under the coupling measure $\mathscr{P}$ defined in Theorem \ref{thm:generalcoupling}, 
    \begin{enumerate}
        \item[(a)] The marginal of $\mathscr{P}$ on $\Sigma$ is $\mathscr{P}_{\Sigma}[\eta]=\mathbb{P}_{\frac{x}{1+x}}[\eta \mid \Sigma] \propto \prod_{e\in \eta} x_e$. 
        For each $\omega\in\Omega$ with $\mathscr{P}_\Omega[\omega] \neq 0$, the  conditional measure $\mathscr{P}[\;\cdot \mid \omega]% = \mathbb{P}_{\frac{x}{p+x}}[\cdot \mid f(\omega)] 
        = \mathtt{UG}[\;\cdot \mid \Sigma^{\downarrow}(\omega)]$- the uniform subgraph of $\omega$ which belongs to $\Sigma$.
        \item[(b)]The marginal of $\mathscr{P}$ on $\Omega$ is $\mathscr{P}_\Omega= \mathscr{P}_{\Sigma} \cup \mathbb{P}_x$.
        For each $\eta\in\Sigma$ with $\mathscr{P}_{\Sigma}[\eta]\neq 0$, the conditional measure $\mathscr{P}[\;\cdot \mid \eta] = \mathbb{P}_x \cup \delta_{\eta}$. 
    \end{enumerate}
\end{corollary}
%Setting $x=p$ one recovers \Cref{prop:conditioned_Bernoulli1}. 
Flipping the picture, with \Cref{prop:conditioned_Bernoulli1} in mind, we will refer to \Cref{prop:partial_coupling} as a \textit{partial} coupling for the following reason: In the uniform case, sprinkling $\mathscr{P}_{\Sigma}$ by $\mathbb{P}_{x}$ yields a measure from which $\mathscr{P}_{\Sigma}$ may be recovered by taking a uniform subgraph from $\Sigma$. \Cref{prop:conditioned_Bernoulli1} then says that a similar statement holds if we only sprinkle by a smaller density $\mathbb{P}_{x'}$ with $x'_e<x_e$ for all $e$, corresponding to only adding some of the edges from $\mathbb{P}_x.$

\begin{remark}
    If $\Sigma =\mathcal{E}_{\emptyset}(G)$ is the set of all even subgraphs of $G$, then $\mathscr{P}_{\Sigma}=\ell_x$ is the loop O(1) measure and the marginal $\mathscr{P}_\Omega= \ell_x \cup \mathbb{P}_x$ is the random-cluster measure $\phi_{2x/(1+x)}$. In that case, \Cref{prop:conditioned_Bernoulli1} is the $A=\emptyset$ case of \Cref{prop:randomcluster-loop} discovered in \cite{evertz2002new, grimmett2007random}.
    We speculate that most of the relations proven using \Cref{thm:generalcoupling} admit some partial coupling in the style of \Cref{prop:partial_coupling}. Such relations were explored for the Swendsen-Wang algorithm in \cite{higdon1998auxiliary}. 
\end{remark}     
In a sense, the corresponding algorithm gives a reduction of weighted approximate sampling to uniform sampling. Such reductions are already well known in the computer science literature under the name of alias tables \cite{knuth1997art}. While uniform sampling algorithms are efficient in some of our cases of interest (the uniform even graph or the uniform $d$-regular graph), it is not always the case. In general, the hardness of uniform sampling is reducible to approximate counting \cite{jerrum1986random}.

    The condition $x_e\in (0,1)$ is essential for \Cref{prop:conditioned_Bernoulli1} to carry through in general:

\begin{lemma} \label{lemma:exclusion} Suppose that $\Sigma$ contains two elements $\eta'\subset \eta$ with $\prod_{e\in \eta'} x_e<\prod_{e\in \eta} x_e$. Then, there exists \textit{no} measure $\mathscr{P}$ on $\Omega\times \Sigma$ such that the marginal $\mathscr{P}_{\Sigma}[\eta] \propto \prod_{e\in \eta} x_e$ and such that  the conditional measure $\mathscr{P}[\;\cdot \mid \omega]$ is uniform on $\Sigma^{\downarrow}(\omega)$.
\end{lemma}
\begin{proof}
We prove the statement by contraposition. Suppose that such a measure $\mathscr{P}$ exists. Then, 
\[\frac{\mathscr{P}_{\Sigma}[\eta]}{\mathscr{P}_{\Sigma}[\eta']}=\frac{\sum_{\omega}\mathscr{P}[\eta \mid \omega] \mathscr{P}_{\Omega}[\omega]}{\sum_{\omega'}\mathscr{P}[\eta' \mid \omega'] \mathscr{P}_{\Omega}[\omega']}=\frac{\sum_{\omega: \eta\subset \omega} \mathscr{P}_{\Omega}[\omega]/\abs{\Sigma^{\downarrow}(\omega)}}{\sum_{\omega': \eta'\subset \omega'} \mathscr{P}_{\Omega}[\omega']/\abs{\Sigma^{\downarrow}(\omega')}} \leq 1,\]
where we used that $\eta\subset \omega$ implies $\eta'\subset \omega$. We conclude that $\prod_{e\in \eta} x_e\leq \prod_{e\in \eta'} x_e$ for all pairs $\eta\subset \eta'$ in $\Sigma$. 
\end{proof}

One particularly natural instance where the assumption of \Cref{lemma:exclusion} applies is when $0\in\Sigma$, $|\Sigma|\geq 2$ and $x_e>1$ for all $e\in E.$ For example, this implies that a generalization of the coupling from \cite{grimmett2007random}
of the loop O($1$) model to the regime $x>1$ requires that one exchange the UEG for something else. Fortunately, % one gets a straightforward corollary to \Cref{prop:conditioned_Bernoulli1}.
if one applies the map $(\eta,\omega)\mapsto (E \setminus \eta,E \setminus \omega)$ to the coupling from \Cref{prop:partial_coupling}, one gets, by push-forward:

\begin{proposition}\label{prop:conditioned_Bernoulli2}
Let $G=(V,E)$ be a finite graph, $p_e\in(0,1), x_e>0$ for each $e\in E$. Let $\Omega = \{0,1\}^E$, $\Sigma \subset \Omega$ and $f(\omega)= \Sigma^{\uparrow}(\omega):=\{\eta \in \Sigma \mid \eta \supset \omega\}$, and 
	$$\rho[\omega] = \Prb_{p}[\omega], ~\forall \omega\in \Omega, ~\gamma[\eta] = \Prb_{\frac{x(1-p)}{1+x(1-p)}}[\eta \mid \Sigma]\propto \prod_{e\in \eta} x_e(1-p_e) , ~\forall \eta\in \Sigma.$$
    	Then, under the coupling measure $\mathscr{P}$ defined in Theorem \ref{thm:generalcoupling}, 
    \begin{enumerate}
        \item[(a)]  The marginal of $\mathscr{P}$ on $\Sigma$ is $\mathscr{P}_{\Sigma}[\eta]\propto \prod_{e\in \eta} x_e$. For each $\omega\in\Omega$ with $\mathscr{P}_\Omega[\omega] \neq 0$, the  conditional measure $\mathscr{P}[\;\cdot \mid \omega]=\Prb_{\frac{x(1-p)}{1+x(1-p)}}[\; \cdot  \mid \Sigma^{\uparrow}(\omega)]$.
        \item[(b)] The marginal of $\mathscr{P}$ on $\Omega$ is $\mathscr{P}_{\Omega} = \mathscr{P}_{\Sigma} \cap \mathbb{P}_{p}$.  
        For each $\eta\in\Sigma$ with $\mathscr{P}_{\Sigma}[\eta] \neq 0$, the conditional measure $\mathscr{P}[\;\cdot \mid \eta] = \Prb_p \cap \delta_{\eta} = \Prb_{\eta,p}. $ %Bernoulli percolation on $\eta$ with parameters $(p_e)_{e\in E}$.  
    \end{enumerate}
\end{proposition}
\begin{proof}
    By Theorem \ref{thm:generalcoupling}, we have
    \begin{align*}
        \mathscr{P}_{\Sigma}[\eta] &\propto \rho[f^{-1}(\eta)] \gamma[\eta] \propto \Prb_{p}[\text{each edge in } E\setminus \eta \text{ is closed}] \prod_{e\in \eta} x_e(1-p_e)\\
        &\propto \prod_{e\in E \setminus \eta}(1-p_e) \prod_{e\in \eta} x_e(1-p_e) \propto \prod_{e\in \eta} x_e,
    \end{align*}
    \begin{align*}
        \mathscr{P}_{\Omega}[\omega] &\propto \gamma[f(\omega)] \rho[\omega] \propto \sum_{\eta \in \Sigma^{\uparrow}(\omega)} \gamma[\eta] \Prb_p[\omega] \propto \sum_{\eta \in \Sigma^{\uparrow}(\omega)} \prod_{e \in \eta} x_e (1-p_e) \prod_{e\in \omega} p_e \prod_{e \in E \setminus \omega}(1-p_e)\\
        &\propto \sum_{\eta \in \Sigma^{\uparrow}(\omega)} \prod_{e\in \eta} x_e \prod_{e\in \omega} p_e \prod_{e\in \eta\setminus \omega} (1-p_e) \propto (\mathscr{P}_{\Sigma} \cap \mathbb{P}_{p})[\omega].
    \end{align*}
    The computations of conditional measures are straightforward.
\end{proof}
Whenever, $x_e> 1, p_e = 1- \frac{1}{x_e}$ for each $e \in E$, $\gamma$ becomes the uniform measure and the coupling simplifies analogously to \Cref{prop:conditioned_Bernoulli1}. 
In this case, for the  conditional measure $\mathscr{P}[\;\cdot \mid \omega]$, instead of taking uniform subgraphs belonging to $\Sigma$, one takes uniform supergraphs belonging to $\Sigma$. 

In \Cref{sec:coupling_examples} some examples are given of how one may extract information about conditional Bernoulli percolation using the results of this section.

%\begin{remark}
 %   Naturally, one could construct more complicated mixtures $\Sigma^{\uparrow}$ or $\Sigma^{\downarrow}$ to get a similar statement for weights $(x_e)_{e\in E}$ with some $x_e<1$ and some $x_e>1.$ The proof will similarly go by applying an appropriate map to some $(\Sigma')^{\downarrow}$.
%\end{remark}

\section{The loop representation of the Potts model} \label{sec:Potts}

The loop O($1$) model is but one of many models of random geometric objects satisfying something like a divergence-free constraint. In the following, we show that the couplings of \cite{grimmett2007random,zhang2020loop} generalize to a loop representation of the lattice gauge Potts models in full generality. To ease ourselves into an appropriate mind-set (and to prove that the   $q$-flow (loop) representation has exponential decay if and only if the Potts model has in \Cref{thm:q_state_potts_torus_trick}), we start by treating the edge models before we delve into the lattice gauge theories.

In the case of the loop O($1$) model, we may identify $\{0,1\}$ with the group
$\mathbb{Z}/2\mathbb{Z}$. For a finite graph $G=(V,E),$ there is a boundary map $\partial:\{0,1\}^E\to \{0,1\}^V$ given by linearly extending $\id_{vw}\mapsto \id_w+\id_v.$ The loop O($1$) model is then Bernoulli percolation conditioned to output an element in the kernel of $\partial$.

The loop-cluster representation from \cite{zhang2020loop} may be defined similarly:
For any finite graph $G = (V,E)$ and natural number $q\geq 2$, let $\mathcal{O}(E)$ denote the set of oriented edges - i.e. the set of ordered pairs $(v,w)$ with $v$ and $w$ adjacent in $G$. In general, for $\mathfrak{e}=(v,w)\in \mathcal{O}(E),$ we write $-\mathfrak{e}=(w,v)$. Let  $\Omega^q(E)\subset (\mathbb{Z}/q\mathbb{Z})^{\mathcal{O}(E)}$ denote the space of $1$-chains, that is:
$$
\Omega^q(E)=\{ \sigma \in (\mathbb{Z}/q\mathbb{Z})^{\mathcal{O}(E)}\mid \forall \mathfrak{e}\in \mathcal{O}(E):\sigma_{-\mathfrak{e}}=-\sigma_\mathfrak{e}\}
$$ Thus, for any fixed orientation $\overrightarrow{E}$ of the edges, $\sigma\in \Omega^q(E)$ is completely determined by $\sigma|_{\overrightarrow{E}}$. % - and such a representation may be referred to as gauge fixing. 
In particular, the dimension\footnote{Or rank, rather, in the case where $q$ is not prime.} of $\Omega^q(E)$ is $\abs{E}$. %See \Cref{fig:gauge-fixing}.   
When defining chains in this section, we will do so in a fixed orientation as in \cite{zhang2020loop}.
For instance, for $\mathfrak{e}\in \mathcal{O}(E),$ we will let $\mathbf{1}^{\mathfrak{e}}$ denote the $1$-chain satisfying $\mathbf{1}^{\mathfrak{e}}_{\mathfrak{e}}=1,$ $\mathbf{1}^{\mathfrak{e}}_{-\mathfrak{e}}=-1$ and $\mathbf{1}^\mathfrak{e}_{\mathfrak{e}'}=0$ for $\mathfrak{e}'\in \mathcal{O}(E)\setminus \{\mathfrak{e},-\mathfrak{e}\}$. In different situations, it is convenient to choose our orientations according to context, and thus, the given definition allows for flexibility. 
In the sequel, we will be considering representations of lattice gauge theories, where working without a fixed orientation is even more helpful.

Just like before, there is a boundary map $\partial: \Omega^q(E)\to (\mathbb{Z}/q\mathbb{Z})^V$ given by a linear extension\footnote{Again, in the case $q=2,$ this is consistent with the definition given before.} of $\mathbf{1}^{(v,w)}\mapsto \id_w -\id_v$.

Then, for positive real numbers $x_e$ indexed by $e \in E$, the $q$-flow measure is defined on $\Sigma=\ker(\partial)$ through \begin{align}\label{eq:q_flow_def}\ell_{G,x}^q[\eta] \propto \prod_{e\in \text{supp}(\eta)} x_e,
\end{align} where $\text{supp}(\eta)=\{vw\in E\mid \eta_{(v,w)}\neq 0\}$, i.e.  parallel oriented edges which are both non-zero are not double-counted. 

In the following, we let $\Omega=\{0,1\}^E,$ and $\ker(\partial^\omega)=\{\eta \in \Sigma\mid \text{supp}(\eta)\subset \omega\}$ be the kernel of the boundary map restricted to $\omega\in \Omega.$ To prepare ourselves to prove that the $q$-flow measure bears the same relation to the random-cluster model as the loop O($1$) model, we need an analogue of the formula for the cyclomatic number \eqref{eq:even_formula} for general chains. The following was  observed in \cite{zhang2020loop}, but we include a proof for completeness. The proof is completely analogous to the one given for $q$ prime in \cite{kavitha2009cycle} and is also given for $q=2$ in \cite[Proposition 2.3]{grimmett2007random}:
\begin{lemma}[\cite{zhang2020loop}] \label{lemma:Counting_divergence}
For any $\omega\in \Omega,$ 
$
\abs {\ker(\partial^\omega)}=q^{|\omega|+\kappa(\omega)-|V|}.
$
\end{lemma}
\begin{remark} \label{rem:nothing_special_edge}
Nothing is special about $\mathbb{Z}/q\mathbb{Z}$. For any finite, abelian group $\mathbb{G},$ $\abs {\ker(\partial^\omega)}=|\mathbb{G}|^{|\omega|+\kappa(\omega)-|V|}$.
\end{remark}
\begin{proof}
Let $F\subset \omega$ be a spanning forest of $\omega$ (i.e. the edges form a spanning tree for every component) and let $\eta,\eta'\in \ker(\partial^\omega)$ with $\eta_{\mathfrak{e}}=\eta_{\mathfrak{e}'}$ for all $\mathfrak{e}\in \mathcal{O}(\omega\setminus F).$ Then, $\eta-\eta'\in \Sigma^{\downarrow}(F)=\{0\},$ which may be checked inductively, since the constraint $\partial \eta=0$ forces the value at any edge adjacent to a leaf to be $0$. This gives $|\ker(\partial^\omega)|\leq q^{|\omega\setminus F|}$.

Conversely, for every $k\in \mathbb{Z}/q\mathbb{Z}$ and $e=vw\in \omega\setminus F,$ there is an oriented simple loop $\Gamma$ in $\omega$ given by following $\mathfrak{e}=(v,w)$ and then taking the unique oriented path in $\mathcal{O}(F)$ from $w$ to $v$. Setting $\eta^{k,\mathfrak{e}}|_{\Gamma}=k$ and $0$ on edges outside of $\Gamma$ and $-\Gamma$ yields an element $\eta^{k,\mathfrak{e}}\in \ker(\partial^\omega)$ with $\eta^{k,\mathfrak{e}}_{\mathfrak{e}}=k$ and $\eta^{k,\mathfrak{e}}_{\mathfrak{e}'}=0$ for all $\mathfrak{e}'\in \mathcal{O}(\omega\setminus (F\cup \{e\}))$. Adding together such paths, we see that $|\ker(\partial^\omega)|\geq q^{|\omega\setminus F|},$ and we conclude that $|\ker(\partial^\omega)|= q^{|\omega\setminus F|}=q^{|\omega|+\kappa(\omega)-|V|}$.
\end{proof}

The lemma combined with \Cref{thm:generalcoupling} proves the coupling from \cite{zhang2020loop}, which generalized \cite{evertz2002new, grimmett2007random}.

\begin{proposition}[\cite{zhang2020loop}]\label{prop:uniform flows} Let $G=(V,E)$ be a finite graph, $q\geq 2$ be an integer, $x_e\in (0,1)$ for each $e\in E,$ $\Omega=\{0,1\}^E,$ $\Sigma=\ker(\partial)$ and $f:\Omega\to 2^{\Sigma}$ be defined by $f(\omega)=\ker(\partial^\omega)$ and 
$$
\rho[\omega]\propto \mathbb{P}_x[\omega],\quad\forall \omega\in \Omega,\qquad \gamma[\eta]\propto 1,\quad\forall \eta\in \Sigma.
$$
	Then, under the coupling measure $\mathscr{P}$ defined in Theorem \ref{thm:generalcoupling}, 
	\begin{enumerate}
		\item[(a)] The marginal of $\mathscr{P}$ on $\Sigma$ is $\ell^q_{G,x}$.
        For each $\omega\in\Omega$, the conditional measure $\mathscr{P}[\;\cdot \mid \omega]$ on $\Sigma$ is the uniform  measure on $\ker(\partial^\omega)$, i.e. a uniform divergence free $q$-coloring of $\omega$.
		\item[(b)] The marginal of $\mathscr{P}$ on $\Omega$ is the random-cluster model with cluster-weight $q$ and edge-weights $(p_e)_{e\in E}$  satisfying $\frac{x_e}{1-x_e}=\frac{p_e}{q(1-p_e)},$ 
        \[
        \phi^q_{G,p}[\omega]\propto \left(\prod_{e\in \omega} \frac{p_e}{1-p_e}\right)q^{\kappa(\omega)}.
        \]
        For each $\eta\in\Sigma$, the %\textcolor{orange}{\so{(marginal of the)}}
        conditional measure $\mathscr{P}[\;\cdot \mid \eta] = \Bernoulli_{x} \cup \delta_{\eta}$ on $\Omega$ is Bernoulli percolation on $E$ with parameter $(x_e)_{e\in E}$ conditioned on all edges in $\eta$ being open.
	\end{enumerate}
\end{proposition}
\begin{proof}
The only part which is not an immediate consequence of \Cref{thm:generalcoupling} is that the marginal on $\Omega$ is the random-cluster model. We calculate using \Cref{lemma:Counting_divergence}:
\begin{align}\label{eq:proof}
\prod_{e\in \omega} \left(\frac{p_e}{1-p_e}\right)q^{\kappa(\omega)}
&\propto \prod_{e\in \omega} \left(\frac{x_e}{1-x_e} \frac{p_e(1-x_e)}{x_e(1-p_e)} \right)q^{\kappa(\omega)} \\
&=\prod_{e\in \omega} \left(\frac{x_e}{1-x_e}\right)q^{\abs{\omega}+\kappa(\omega)}\propto \prod_{e\in \omega} \left(\frac{x_e}{1-x_e}\right) \abs{\ker(\partial^\omega)},
\end{align}
as claimed.
%where \Cref{lemma:Counting_divergence} was used in the last identity.
\end{proof}

For the rest of the section, we consider the case where $p_e\equiv p.$
Let $\ell^q_{G,x}[v\cc B]$ denote the probability that the vertex $v$ is connected to the set $B$ by a path of edges in the $\textrm{supp}(\eta)$. 
Let $B_n = [-n,n]^d \cap \Z^d$ be the box of size $n$ in dimension $d$.  
Define, furthermore, $x_c=x_c(d,q)=\frac{p_c(d,q)}{p_c(d,q)+q(1-p_c(d,q))},$ where $p_c(d,q)$ is the transition point of the $q$-state random-cluster model in $d$ dimensions, corresponding to the critical temperature of the Potts model.

\begin{corollary} \label{cor:immediate_exp_decay}
Let $d\in \mathbb{N}$ and $x<x_c.$ 
Then, there exists $C>0$ such that for any sequence of finite graphs $(G_n)_{n\in \mathbb{N}}$ extending\footnote{Meaning that the vertex boundary of $B_{n-1}$ inside $G_n$ coincides with its vertex boundary inside $\mathbb{Z}^d$.} $B_{n-1}\subset \mathbb{Z}^d,$ $\ell^q_{G_n,x}[0\cc \partial B_{n-1}]\leq \exp(-Cn)$. 
\end{corollary}
\begin{proof}

    By \Cref{prop:uniform flows}, there exists a coupling $\mathscr{P}$ of  $\eta \sim \ell_{G_n,x}^q$ and $\omega\sim \phi^q_{G_n,p}$ such that $\textrm{supp}(\eta)\subset \omega$. Thus, the statement follows from the fact that the same is true for $\phi^q_{G_n,p}$ when $p<p_c$ because of sharpness \cite{DCsharpness}.
\end{proof}

Perhaps less obviously, \Cref{prop:uniform flows} also implies that the conclusion of \Cref{cor:immediate_exp_decay} does not hold for $x>x_c$ by taking $G_n$ to be a sequence of tori. This is explained in \Cref{sec:polynomial_lower_bound_q_flow}. 

\subsection{Loop-Cluster coupling measure for Potts lattice gauge theory} \label{sec:Lattice_Gauge} We now turn our attention to graphical representations of the Potts lattice gauge theories. These were introduced by Wegner as examples of thermodynamic systems which lack local order parameters, yet exhibit phase transitions \cite{WegDual} and have later been taken up as toy models of Euclidean gauge theories \cite{KogPotts}. In the following, we show that the $q$-flow measure extends to a divergence-free representation of all lattice gauge Potts models. However, we first need a crash course in exterior differential calculus.

Let $\Lambda$ be a finite cubical complex, i.e. a tuple $(\Lambda_0, \dots, \Lambda_N)$, such that
\begin{enumerate}
    \item For each $k\in \mathbb{N}_0,$ a collection $\Lambda_k$ of $k$-cells, which are copies of the unit hypercube $\{0,1\}^k$. 
    \item For each $k\geq 1$ and $\mathbf{c}\in \Lambda_k$, each of its $k-1$-dimensional faces lies in $\Lambda_{k-1}$.
    \item The set $\cup_{k=0}^N \Lambda_k$ is finite.
\end{enumerate}
 A natural example arises from a finite subgraph $G\subset \mathbb{Z}^d$, where inductively, the $0$-cells are the vertices of $G$, the $1$-cells are the edges of $G$ and where a $k$-cell is any embedded copy of $\{0,1\}^k$ in $G$ (as a graph).

For the unit cube $\{0,1\}^k$ and $1\leq j\leq k$, we let $F_{j}^+(\{0,1\}^k)$ denote the codimension-$1$ face $\{0,1\}^{j-1}\times \{1\}\times \{0,1\}^{k-j}$ and accordingly $F^-_{j}(\{0,1\}^k)=\{0,1\}^{j-1}\times \{0\}\times \{0,1\}^{k-j}$. These definitions then induce face maps $F^+_{j},F^-_{j}:\Lambda_k\to \Lambda_{k-1}$.
 For $q\in\mathbb{N}\backslash\{1\},$ we define a boundary map by $\partial_k:\Lambda_{k}\to (\mathbb{Z}/q\mathbb{Z})^{\Lambda_{k-1}}$ by\footnote{The usual purpose of the signs is to ensure that $\partial \circ \partial=0,$ which in and of itself does not matter to our present story. However, the signs are still important to ensure that the dual map $d_k$ defined in \Cref{sec:elementary} has the correct relation to the Potts lattice gauge theory.} $\partial \mathbf{c}=\sum_{j=1}^k (-1)^{j+1} \left(F_{j}^+(\mathbf{c})-F_{j}^-(\mathbf{c})\right)$.

The two possible orientations of the continuum hypercube $[0,1]^d$ induce a notion of orientation for the cells of a cubical complex. Equivalently, for $k\geq 2,$ an orientation $\mathfrak{c}$ of $\mathbf{c}\in \Lambda_k$ rooted at the vertex $v$ is  given by an ordering $e_{i_1}\wedge e_{i_2}\wedge...\wedge e_{i_k}$  of the $k$ edges adjacent to $v$ modulo even permutations. If $w$ and $v$ are neighbors connected by $e_{i_j}$, the same orientation rooted at $w$ is represented by the opposite orientation to $\tilde{e}_{i_1}\wedge \tilde{e}_{i_2}\wedge...\wedge e_{i_j}\wedge...\tilde{e}_{i_k},$ where $\tilde{e}_{i_l}$ is the unique edge adjacent to $w$ and parallel to $e_{i_l}.$ We denote by $\mathcal{O}(\Lambda_k)$ the set of oriented $k$-cells and for $\mathfrak{c}\in \mathcal{O}(\Lambda_k),$ we denote by $-\mathfrak{c}$ its opposite orientation.
 An oriented $k$-cell $\mathfrak{c}$ represented by $e_{1}\wedge e_{2}\wedge...\wedge e_{k}$ rooted at $v$ then naturally yields the oriented $(k-1)$-cell $F^{\pm}_j(\mathfrak{c})$ rooted at $v$ given by $e_{1}\wedge e_{2}\wedge... \wedge e_{j-1}\wedge e_{j+1}\wedge... \wedge e_{k},$ where the sign $\pm$ is chosen such that $v\in F^{\pm}_j(\mathbf{c})$. For $k=1$, we keep the previous definition of orientation.

The Potts lattice gauge theories (the definition of which is postponed to \eqref{eq:lattice_gauge_Hamiltonian}) live on spaces of $k$-chains on $\mathcal{O}(\Lambda_k)$. For $k\in \mathbb{N},$ the space of $\mathbb{Z}/q\mathbb{Z}$-valued $k$-\textit{chains} $C_k$ is the set of maps 
\begin{align}
   C_k = \left \{ \sigma\in (\mathbb{Z}/q\mathbb{Z})^{\mathcal{O}(\Lambda^{k})} \mid  \sigma_{\mathfrak{c}}=-\sigma_{-\mathfrak{c}} \text{ for every } \mathfrak{c}\in \mathcal{O}(\Lambda_k) \right \}.
\end{align}

 Again, $C_k$ is spanned by elements of the form $\mathbf{1}^{\mathfrak{c}}$ satisfying $\mathbf{1}^{\mathfrak{c}}_\mathfrak{c}=1,$ $\mathbf{1}^{\mathfrak{c}}_{-\mathfrak{c}}=-1$ and $\mathbf{1}^{\mathfrak{c}}_\mathfrak{c'}=0$ otherwise. For each $k\geq 2$, the boundary map $\partial_k$ extends linearly to a map $\partial_k:C_k\to C_{k-1}$ such that 
\begin{align}\label{eq:boundary_general}
\partial_k(\mathbf{1}^{\mathfrak{c}})=\sum_{j=1}^{k} (-1)^{j+1} \left( \mathbf{1}^{F_j^+(\mathfrak{c})}-\mathbf{1}^{F_j^-(\mathfrak{c})} \right).
\end{align}
 We let the boundary map $\partial_1$ agree with the previously given definitions.

 For natural numbers $q\geq 2$ and $k$, and positive real numbers $(x_\mathbf{c})_{\mathbf{c}\in \Lambda_k}$, define the $q$-flow model with cell dimension $k$ on $\textrm{ker}(\partial_k)$ by 
\begin{align}\label{eq:def_lattice_gauge_q_flow}
\ell^q_{\Lambda_k,x}[\eta]\propto \prod_{\mathbf{c}\in \text{supp}(\eta)}x_{\mathbf{c}},\end{align}
where $\textrm{supp}(\eta)=\{\mathbf{c}\in \Lambda_k\mid \forall \mathfrak{c}\in \mathcal{O}(\mathbf{c}): \eta_{\mathfrak{c}}\neq 0\}$ (that is, once again, we do not double count the support).

 Analogously to how we only considered subgraphs containing all vertices in the previous sections, the $k$-dimensional plaquette random-cluster model will be defined on subcomplexes that contain all lower-dimensional cells. We say that a subcomplex $\omega$ of $\Lambda$ is \textit{$k$-spanning} if $\Lambda_j(\omega)=\Lambda_j$ for $0\leq j\leq k-1$ and $\Lambda_{j}(\omega)=\emptyset$ for $j>k$. Thus, a spanning subgraph is a $1$-spanning subcomplex. For a subcomplex $\omega$, we let $\partial_k^{\omega}$ denote the induced maps on $\omega$ as a complex (this might, in general, both restrict the domain and co-domain).
The plaquette random-cluster model of cell dimension $k$ with coefficients in $\mathbb{Z}/q\mathbb{Z}$ and parameters $p=(p_{\mathbf{c}})_{\mathbf{c}\in \Lambda_k}$ on $\Lambda$ is then the random $k$-spanning subcomplex which to a given $k$-spanning subcomplex $\omega$ assigns the probability
\begin{align}\label{eq:new_definition_plaquette_random_cluster}
\phi^q_{\Lambda_k,p}[\omega]\propto \frac{|\ker (\partial_{k-1}^{\omega})|}{|\textrm{Im}(\partial_{k}^{\omega})|}\prod_{\mathbf{c}\in \Lambda_k (\omega)} \frac{p_{\mathbf{c}}}{1-p_{\mathbf{c}}}. 
\end{align}

It is worth noting that for $k=1,$ we recover the usual definition (see item (b) of \Cref{prop:uniform flows}). This is essentially proven along the way in \Cref{lemma:random_cluster_lattice_gauge_with_kernel} below.

Several historic attempts at defining a random-cluster model on general plaquettes were made in the 80s (see e.g. \cite{HistoricPRCM2,HistoricPRCM3,HistoricPRCM4}) but were shown to suffer from problems such as topological anomalies \cite{HistoricPRCM1} (see \cite{duncanPRCM1} for further explication). The plaquette random-cluster model we study here was introduced by Hiraoka and Shirai \cite{HS16} for prime $q$ and extended by Duncan and Schweinhart \cite{duncanPRCM2} to general $q$. In \cite{duncanPRCM2}, the Radon-Nikodym derivative of $\phi^q_{\Lambda_k,p}$ with respect to the Bernoulli product measure is the cardinality of the reduced cohomology of $\omega$ with coefficients in $\mathbb{Z}/q\mathbb{Z}$, $|H^{k-1}(\omega,\mathbb{Z}/q\mathbb{Z})|$. By the universal coefficient theorem for cohomology (see also Corollary 58 of \cite{duncanPRCM2}),
\[H^{k-1}(\omega,\mathbb{Z}/q\mathbb{Z}) \cong H_{k-1}(\omega,\mathbb{Z}/q\mathbb{Z}):=\frac{\ker (\partial_{k-1}^{\omega})}{\textrm{Im}(\partial_{k}^{\omega})}\]
under reasonable assumptions, which justifies our definition of $\phi^q_{\Lambda_k,p}$ in all natural cases of interest (e.g. boxes in $\mathbb{Z}^d$ and $\mathbb{T}_n^d).$ However, there are cubical complexes which do not satisfy the assumptions of \cite[Corollary 58]{duncanPRCM2}.  In \Cref{sec:elementary}, we give an elementary proof showing that the two Radon-Nikodym derivatives give the same measure.

As we saw in \Cref{prop:uniform flows}, when $k=1$, for $\omega\sim \phi^q_{\Lambda_1,p},$ the uniform element of $\ker(\partial^\omega)$ %$=\{\eta \in \Sigma\mid \text{supp}(\eta)\subset \omega\}$ 
has the distribution of $\ell^q_{\Lambda_1,x}$ and that $\phi^q_{\Lambda_1,p}$ can be obtained from $\ell^q_{\Lambda_1,x}$ via sprinkling. 
%We next argue that sampling a divergence free chain on the plaquette random-cluster model always yields a flow model. First, we prove the following lemma, generalizing the rewriting of the random-cluster model in the proof of \Cref{prop:uniform flows}:
The new way of writing the plaquette random-cluster model in \eqref{eq:new_definition_plaquette_random_cluster}, enables us to couple it to the $q$-flow measure analogously to the case $k=1$. That is, sampling a divergence free chain on the plaquette random-cluster model always yields a flow model. This highlights the topological background for the original Loop-Cluster coupling. 
 A key step is the following generalization of \eqref{eq:proof}.
\begin{lemma} \label{lemma:random_cluster_lattice_gauge_with_kernel}
For $x_{\mathbf{c}}=\frac{p_{\mathfrak{c}}}{p_{\mathfrak{c}}+q(1-p_{\mathfrak{c}})},$ it holds that 
$
\phi^q_{\Lambda_k,p}[\omega]\propto |\ker(\partial^{\omega}_k)|\prod_{\mathfrak{c}\in \Lambda_k (\omega)} \frac{x_{\mathfrak{c}}}{1-x_{\mathfrak{c}}}.
$
\end{lemma}
\begin{proof}
Since $\omega$ is $k$-spanning,  $\ker(\partial_{k-1}^{\omega})=\ker(\partial_{k-1})$ is independent of $\omega$. Hence,
\begin{equation}\label{eq:phidef}
\phi^q_{\Lambda_k,p}[\omega] \propto \frac{1}{|\textrm{Im} (\partial_{k}^{\omega})|} \prod_{\mathbf{c}\in \Lambda_k (\omega)} \frac{p_{\mathbf{c}}}{1-p_{\mathbf{c}}}.
\end{equation}
Applying the first isomorphism theorem to the homomorphism $\partial_k^{\omega}: C_k(\omega)\rightarrow C_{k-1}(\omega)$, we get
\[|\ker(\partial_k^{\omega})| \cdot |\textrm{Im}(\partial_k^{\omega})|=|C_k(\omega)|=q^{|\Lambda_k(\omega)|},\]
where the last equality comes from counting the $k$-chains in $C_k(\omega)$. Plugging this into \eqref{eq:phidef} yields the lemma.
\end{proof} 
\Cref{lemma:random_cluster_lattice_gauge_with_kernel} yields a generalization of \Cref{prop:uniform flows}.  
\begin{proposition}\label{prop:lattice_potts}
Let $\Lambda_k$ be a collection of $k$-cells, $\Omega = \{0,1\}^{\Lambda_k}$, $\Sigma =\ker(\partial_k)$,  $x_{\mathbf{c}}\in (0,1)$ for all $\mathbf{c}\in \Lambda_k$,  $f:\Omega\to 2^{\Sigma}$ be defined by $f(\omega)= \Sigma^{\downarrow}(\omega):=\{\eta \in \Sigma\mid \operatorname{supp}(\eta) \subset \omega\}$  and
	\[\rho[\omega] \propto \prod_{\mathbf{c}\in \Lambda_k (\omega)} \frac{x_{\mathbf{c}}}{1-x_{\mathbf{c}}}, ~\forall \omega\in \Omega, ~\gamma[\eta]\propto 1, ~\forall \eta\in \Sigma. \]
    	Then, under the coupling measure $\mathscr{P}$ defined in Theorem \ref{thm:generalcoupling}, 
    \item[(a)] The marginal  $\mathscr{P}_\Sigma = \ell^q_{\Lambda_k,x}$.
          For $\omega\in\Omega,$ the  conditional measure $\mathscr{P}[\;\cdot \mid \omega]$  is uniform on $\Sigma^{\downarrow}(\omega)$, i.e. a uniform divergence free $q$-coloring of $\omega$.
		\item[(b)] The marginal $\mathscr{P}_\Omega= \phi^q_{\Lambda_k,p}$, with $x_{\mathbf{c}}=\frac{p_{\mathbf{c}}}{p_{\mathbf{c}}+q(1-p_{\mathbf{c}})}$.
        For each $\eta\in\Sigma$, the conditional measure $\mathscr{P}[\;\cdot \mid \eta] = \mathbb{P}_x \cup \delta_{\operatorname{supp}(\eta)}$,  Bernoulli percolation on $\Lambda_k$ with parameter $(x_{\mathbf{c}})_{\mathbf{c}\in \Lambda_k}$ conditioned on all $k$-cells in $\operatorname{supp}(\eta)$ being open.
\end{proposition}

\begin{remark}[Additional couplings]
In \Cref{prop:conditioned_Bernoulli1}, we saw that the relation between the loop O(1) and random-cluster model, does not need the structure of even graphs. Similarly, \Cref{prop:lattice_potts} generalizes upon replacing $\ker(\partial_k)$ with any subset $\Sigma \subset \{0,1\}^{\Lambda_k}$ and there also exists a partial coupling generalizing \Cref{prop:partial_coupling} to the lattice gauge case. Let us also note that Ben-Av. et. al. \cite{ben1990critical},  Hiraoka and Shirai \cite{HS16} and  Duncan and Schweinhart \cite[Proposition 21]{duncanPRCM1} generalized the  coupling between the Potts and random-cluster models to the lattice gauge framework (cf. \Cref{tab:coupling_overview}).  
\end{remark}
 \Cref{sec:duality} reviews how the (generalized) $q$-flow model has the law of the domain wall in a lattice gauge theory generalizing the direct relation between the loop O(1) and Ising models that exist for planar graphs, and puts the coupling into the context of \Cref{thm:generalcoupling}. In \cite[Theorem 18]{duncanPRCM1}, a related but different duality of Potts lattice gauge theories is discussed. The relation is exactly analogous to the relation between the duality of planar random-cluster models and the duality between Ising models and interfaces (i.e. the loop O($1$) model).
 
Let us finally note that the random current expansion has been generalized to lattice gauge Ising models in \cite{aizenman1982geometric, forsstrom2025current}. We believe that it should be possible to prove suitable versions of \Cref{prop:doublecurrent-Xor} and \Cref{prop:doublecurrent-loop} as well. %In the next section, we consider some additional open questions. 

\section{Open questions}\label{sec:questions}
We believe that the couplings discussed in this paper can be applied in a myriad of ways and we record some instances that seem particularly promising. 

\subsection{Potential model specific applications}
One of our motivations for generalizing the couplings between the graphical representations was the recent use of the coupling of the loop $\mathrm{O}(1)$ model as the uniform even subgraph of FK-Ising in \cite{ aizenman2019emergent,angel2021uniform, hansen2023uniform}. 
%It still remains to see the limits of the use of the couplings.  
We suspect that the antiferromagnetic loop $\mathrm{O}(1)$ model, which can be studied using \Cref{prop:conditioned_Bernoulli2}, is low-hanging future usecase. 
\begin{question}
For $d\geq 2$, is $\ell_{\mathbb{Z}^d,t}[0 \cc \infty] > 0$ for all $t \geq 1$? 
\end{question}
A positive answer to the question would settle a missing point in the phase diagram for $d = 2$ in \cite{GMM18}. In \cite{hansen2023uniform}, it was proven\footnote{In fact, the same arguments prove percolation for $t \in [1-\varepsilon,1+\varepsilon]$ on $\mathbb{Z}^d.$ Indeed, similarly to \Cref{sec:anti_ferromagnetic_Ising}, one may prove that the marginal on $\Omega$ obtained  in \Cref{prop:conditioned_Bernoulli2} is the complement of a random-cluster model conditioned on $\mathcal{F}_A,$ where $A$ is the set of vertices of odd degree in the full graph.
Since $\mathbb{Z}^d$ is even, this conditioning is trivial and the previous argument carries through.
One can also prove percolation for large $t$ by noting that for $t<1$, $\ell_t$ is dominated by the random-cluster measure $\phi$, which is dominated by Bernoulli percolation. Since $\Z^d$ is itself even, the complement of $\ell_{t}$ is $\ell_{\frac{1}{t}}$ and dominated below by a Bernoulli percolation, which is supercritical if $t$ is large. } in an interval $[1-\varepsilon,1]$ for all $d \geq 2$. But since the loop $\mathrm{O}(1)$ model in general \cite{GMM18,klausen2021monotonicity} and its percolative properties in particular \cite{hansen2024nonuniquenessphasetransitionsgraphical} are not monotone, the question is still open for $t \geq 1$ to our knowledge. Similarly, one may ask whether $\ell_{\mathbb{Z}^d,t}^q[0 \cc \infty] > 0$, for all $t \geq 1, q \geq 2$.

A more challenging question is whether the single random current percolates throughout the supercritical regime of the Ising model as was asked by Duminil-Copin in \cite[Question 1]{DC16}. This question was motivated by the exponential decay of truncated correlation, which was later proven using other methods in \cite{duminil2020exponential}.
\begin{question}[\cite{DC16}]
On $\mathbb{Z}^d$, $d\geq 3$ and $\beta > \beta_c(d)$, does the single random current percolate? 
\end{question}
We refer the interested reader to  related open questions in \cite[Section 6]{hansen2023uniform}, where a partial answer is also given. 

In \Cref{thm:q_state_potts_torus_trick}, it was proven that the $q$-flow model on the torus has exponential decay if and only if the $q$-state Potts model has.  It is natural to ask the follow-up question.

\begin{question}
 For the $q$-flow measure on $\mathbb{Z}^d$, $\ell_{\mathbb{Z}^d,x}^q$, and $x>x_c(d,q)$, is the expected cluster size of $0$ infinite? 
Do non-zero edges percolate?
\end{question}
One could also ask whether \Cref{thm:q_state_potts_torus_trick} has a natural extension to real $q$ using the tools from \cite{chayes1998graphical,deng2007cluster}. 
It also seems promising that the analogue for Potts lattice gauge models could be resolved.  
\begin{question}
Does $\ell^q_{\Lambda_k,x}$ have perimeter law if and only if $\phi^q_{\Lambda_k,p(x,q)}$ has?
\end{question}
%The XOR-trick on the torus employed in \cite{hansen2023uniform} extends to $q$-flow model implying a polynomial lower bound for connection events of non-zero edges on the $n$ torus in $d$-dimensions $\ell_{\beta,\mathbb{T}_n^d}^q[0 \cc \partial \Lambda_n] $, whenever $\beta > \beta_c(d,q)$.

\subsection{General questions on extensions of the framework}
The couplings discussed above related the lattice gauge Potts model to its graphical representations. 
However, it is believable  that a more general program can be carried out, cf. Remarks  \ref{rem:nothing_special_edge} and \ref{rem:nothing_special_gauge}. 

\begin{question}
For a general spin model (defined on a compact group), construct its FK-Edwards-Sokal representation and its high-temperature expansion. 
What are the assumptions on the group necessary to get a Loop-Cluster algorithm?
\end{question}

Propositions
\ref{prop:partial_coupling} and \ref{prop:conditioned_Bernoulli2} generalize the Swendsen-Wang algorithm from the space of even graphs to arbitrary subsets of all graphs.  Much work has been devoted to the mixing time of Swendsen-Wang dynamics  (see e.g., \cite{guo2017random,levin2017markov} and references therein).  The ease of sampling the uniform even graph as well as Bernoulli percolation suggests that the  loop-cluster algorithm is a good starting point and it was investigated in \cite{zhang2020loop}.
\begin{question}
    Which of the new algorithms from the general coupling are computationally efficient? 
\end{question}

Finally, from a more abstract point of view, we only stated the main theorem \Cref{thm:generalcoupling} in the case of finite sets, but we know that generalizations exist, as they were already achieved for the XY model in \cite{edwards1988generalization} and a coupling was used in the infinite volume limit in \cite{angel2021uniform}. 

However, we believe that the Ising couplings in \Cref{fig:couplings_horizontal} could be expanded to a more general framework of probability kernels, which could include the (continuum) random-cluster and current representations of the quantum Ising model that have been used to extend several results from the classical to the quantum Ising model  \cite{bjornberg2015vanishing,bjornberg2009phase,Ioffe2009stochastic}. 
\begin{question}
    Let $(\Omega,\mathcal{B}(\Omega),\mathscr{P}_{\Omega})$ and $(\Sigma,\mathcal{B}(\Sigma),\mathscr{P}_{\Sigma})$ be Polish measure spaces and let $f:\Omega\to \mathcal{B}(\Sigma)$ be measurable\footnote{There is a natural $\sigma$-algebra on $\mathcal{B}(\Sigma)$ generated by the Borel measures.}.
    What assumptions on $\mathscr{P}_{\Omega}$, $\mathscr{P}_{\Sigma}$ and $f$ are needed for there to exist a probability measure $\mathscr{P}$ on $\Omega\times \Sigma$ with $\mathscr{P}_{\Omega}$ and $\mathscr{P}_{\Sigma}$ as marginals and such that the conditional measure $\mathscr{P}[\;\cdot \mid \omega]$ is supported on $f(\omega)$ almost surely? Do the graphical representations of the quantum Ising model satisfy any such assumptions? 
\end{question}

\begin{appendix}
\section{The double random current and the XOR Ising model}
The random currents encode correlations of the Ising model and have similarity to both random walk and percolation. They were first introduced in \cite{griffiths1970concavity} and brought to power in \cite{aizenman1982geometric} and have been fueling much of the recent rigorous progress on the Ising model \cite{duminil2022100}, which has made them an object of independent study.

\subsection{Preliminaries: The loop $\mathrm{O}(1)$ model sprinkled with Bernoulli family} \label{sec:l1Bf} 
To ease notation here and in the following, define 
$ t^F :=  \prod_{e \in F} t_e = \prod_{e \in F}\tanh(J_e)$ for each $F \subset E$. 
The following formula for the double random current model will be crucial for our proofs of Propositions \ref{prop:doublecurrent-Xor} and \ref{prop:doublecurrent-loop}. The formula was proved by Lis in \cite[Theorem 3.2]{lis2017planar} for the case $B=\emptyset$. For completeness, we include a proof.

 \begin{proposition}[Lis's formula]\label{prop:Lisformula}
\begin{align}\label{eq:Lisformula}
    \Prbcur^{A,B}[\omega] \propto \id_{\mathcal{F}_A}(\omega) \abs{\Even(\omega)} (\ell^{A\triangle B}\cup \mathbb{P}_{t^2})[\omega]. 
    \end{align}
\end{proposition}
\begin{proof}
Defining briefly $s_e:= \sqrt{1-t_e^2}$ and $s^{2\omega}:=\prod_{e\in \omega} s_e^2=\prod_{e \in \omega}(1-t_e^2) =  \prod_{e \in \omega} \hspace{-5pt}  \sech(J_e)^2$ yields, 
\begin{align*}
\Prbcur^{A,B}[\omega] = \loopmeasure^A \cup \loopmeasure^B \cup \Bernoulli_{t^2}[\omega] 
&\propto \sum_{\partial \eta_1 = A, \partial \eta_2 = B} t^{\eta_1}t^{\eta_2} t^{2(\omega \setminus (\eta_1 \cup \eta_2))} s^{2(E \setminus \omega)}  \id[\eta_1, \eta_2 \subset \omega]\\
&= t^{2\omega} s^{2(E\setminus \omega)} \sum_{\partial \eta_1 = A, \partial \eta_2 = B}  t^{- (\eta_1\triangle \eta_2)}\id[\eta_1, \eta_2 \subset \omega]\\
& =  \id_{\mathcal{F}_A}(\omega)t^{2\omega} s^{2(E \setminus \omega)}\sum_{\partial \eta_1 = A, \partial \tilde{\eta}_2 = A \triangle B}  t^{-\tilde{\eta}_2}\id[\eta_1 \subset \omega] \id[ \tilde{\eta}_2 \triangle \eta_1 \subset \omega]\\
%& =  \id_{\mathcal{F}_A}(\omega)t^{2\omega} s^{2(E\setminus \omega)}\sum_{\partial \eta_1 = A, \partial \eta_2 = A \triangle B}  t^{-\eta_2}\id[\eta_1 \subset \omega] \id[ \eta_2 \subset \omega]\\
& =  \id_{\mathcal{F}_A}(\omega) s^{2(E \setminus \omega)} \abs{\mathcal{E}_A(\omega)} \sum_{\partial \tilde \eta_2 = A \triangle B, }  t^{\tilde \eta_2} t^{2(\omega \setminus \tilde \eta_2)}\id[\tilde \eta_2 \subset \omega] \\
& \propto \id_{\mathcal{F}_A}(\omega) \abs{\Even(\omega)} (\ell^{A\triangle B}\cup \mathbb{P}_{t^2})[\omega],
\end{align*}
where $\tilde{\eta}_2=\eta_1 \triangle \eta_2$ in the third line, %$\eta_2=\tilde{\eta}_2$ in the fourth line, 
and the switching principle \eqref{eq:switching_principle} was used in the last line. 
\end{proof}

The two-parameter family $\ell^A_x \cup \Bernoulli_p$ has nice properties under reweighing by $2^{\abs{\omega}}$ illustrated in \Cref{fig:reweighing_dynamics}.
\begin{lemma}
Let $G=(V,E)$ be a finite  graph. Let $A\subset V$ with $|A|$ even (could be $0$). Then, $\forall \omega\in\{0,1\}^E$,
\begin{align}\label{eq:reweight_by_two}
2^{\abs{\omega}} (\loopmeasure^A_x \cup \Bernoulli_p)[\omega] \propto \left(\loopmeasure^A_{\frac{2x}{1+p}} \cup \Bernoulli_{\frac{2p}{1+ p}}\right)[\omega].
\end{align}
   In particular, with the usual $t=(t_e)_{e\in E}$ and $J=(J_e)_{e\in E}$, it holds that
\begin{align}\label{eq:reweight_identity}
       2^{\abs{\omega}} (\loopmeasure^A_t \cup \Bernoulli_{t^2})[\omega] \propto \Prbcur^A_{2J}.
   \end{align}
\end{lemma}

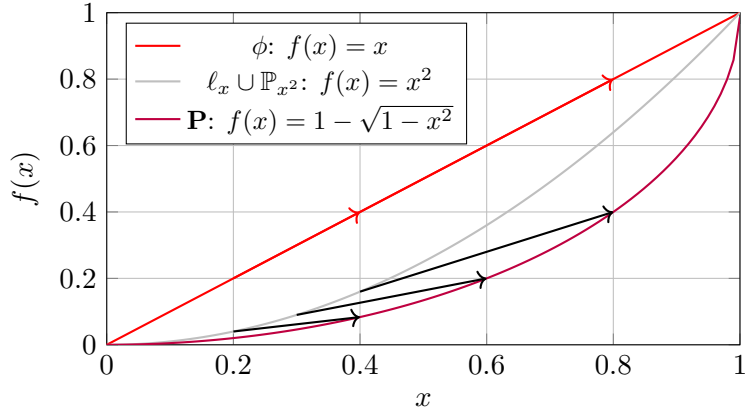
\begin{figure}
    \centering
\begin{tikzpicture}
    \begin{axis}[
        width=10cm, height=6cm,
        grid=both,
        xlabel={$x$}, ylabel={$f(x)$},
        xmin=0, xmax=1,
        ymin=0, ymax=1,
        legend pos=north west,
        legend style={font=\small}
    ]
        % Plot f(x) = 0
        %\addplot[domain=0:1, samples=100, thick, color=blue] {0};
        %\addlegendentry{$\ell$: $f(x) = 0$}

        % Plot f(x) = x
        \addplot[domain=0:1, samples=100, thick, color=red] {x};
        \addlegendentry{$\phi$: $f(x) = x$}

        % Plot f(x) = x^2
        \addplot[domain=0:1, samples=100, thick, color=lightgray] {x^2};
        \addlegendentry{$\ell_x\cup \Prb_{x^2}$: $f(x) = x^2$}

        % Plot f(x) = 1 - sqrt(1 - x^2)
        \addplot[domain=0:1, samples=100, thick, color=purple] {1 - sqrt(1 - x^2)};
        \addlegendentry{$\Prbcur$: $f(x) = 1 - \sqrt{1 - x^2}$}

        % Add error points along f(x) = 0 and f(x) = x
        % Add error lines from (x, x^2) to (2x, 1 - sqrt(1 - (2x)^2))
        \addplot[->, thick, color=red] coordinates {
            (0.2, 0.2) (0.4, 0.4) % Explicit value for 1 - sqrt(1 - 0.16)
        };
       % \addplot[->, thick, color=blue] coordinates {
       %     (0.3, 0) (0.5, 0) % Explicit value for 1 - sqrt(1 - 0.16)
        %};
        %\addplot[->, thick, color=blue] coordinates {
        %    (0.1, 0) (0.3, 0) % Explicit value for 1 - sqrt(1 - 0.16)
        %};
        \addplot[->, thick, color=red] coordinates {
            (0.4, 0.4) (0.8, 0.8) % Explicit value for 1 - sqrt(1 - 0.16)
        };
        \addplot[->, thick, color=black] coordinates {
            (0.2, 0.04) (0.4, 0.0835) % Explicit value for 1 - sqrt(1 - 0.16)
        };
        \addplot[->, thick, color=black] coordinates {
            (0.3, 0.09) (0.6, 0.2) % Explicit value for 1 - sqrt(1 - 0.36)
        };
        \addplot[->, thick, color=black] coordinates {
            (0.4, 0.16) (0.8, 0.4) % Explicit value for 1 - sqrt(1 - 0.64)
        };
    \end{axis}
\end{tikzpicture}
  \caption{The flow of two parameter family $\ell_x \cup \Bernoulli_p$ under reweighing by $2^{\abs{\omega}}$ governed by \eqref{eq:reweight_by_two} is indicated by the arrows: The curve $p=x^2$ flows to the curve of single random current measures, the random-cluster line $p=x$ flows into itself.}
    \label{fig:reweighing_dynamics}
\end{figure}

\begin{proof}
Notice first that, 
  \begin{align*}
         (\loopmeasure^A_x \cup \Bernoulli_p)[\omega]
        \propto \sum_{\eta \subset \omega, \partial \eta = A} x^{\eta} p^{\omega \setminus \eta} (1-p)^{E \setminus \omega} 
        \propto \sum_{\eta \subset \omega, \partial \eta = A} \left( \frac{x}{1-p}\right)^{\eta} \left( \frac{p}{1-p}\right)^{\omega \setminus \eta}.
    \end{align*}
In this reparameterization we see that, 
    \begin{align*}
        2^{\abs{\omega}} (\loopmeasure^A_x \cup \Bernoulli_p)[\omega]
     \propto \sum_{\eta \subset \omega, \partial \eta = A} \left( \frac{2 x}{1-p}\right)^{\eta} \left( \frac{2 p}{1-p}\right)^{\omega \setminus \eta} \propto (\loopmeasure^A_{\tilde x} \cup \Bernoulli_{\tilde p})[\omega]. 
    \end{align*}
where 
$\frac{\tilde x}{1- \tilde p} = \frac{2 x}{1-p}$
    and
$\frac{\tilde p}{1- \tilde p} = \frac{2 p}{1-p}$. 
Then, \eqref{eq:reweight_by_two} follows by solving these equations, and \eqref{eq:reweight_identity} follows from \eqref{eq:reweight_by_two}, \eqref{eq:definition_single_current} and the double angle formulas
\[\tanh(2J_e)=\frac{2\tanh(J_e)}{1+\tanh^2(J_e)},~\hspace{-5pt}  \sech(2J_e)=\frac{1}{1+2\sinh^2(J_e)},\]
which concludes the proof.
\end{proof}

 Starting from \eqref{eq:Lisformula} and using first \eqref{eq:even_formula} and then, \eqref{eq:reweight_identity} yields a formula  which is essential for proving Proposition \ref{prop:doublecurrent-Xor}.
\begin{align}\label{eq:double_J}
\Prbcur_{J}^{A,B}[\omega] \propto \id_{\mathcal{F}_A}(\omega)2^{\kappa(\omega) }  2^{\abs{\omega} }  (\ell_{t}^{A\triangle B}\cup \mathbb{P}_{t^2})[\omega] 
    \propto \id_{\mathcal{F}_A}(\omega)2^{\kappa(\omega) }  \Prbcur_{2J}^{A\triangle B}[\omega]. 
\end{align}

\subsection{The XOR Ising model} The XOR Ising model is the pointwise product of two independent Ising models with couplings \((J_e)_{e \in E}\). 
Recall that $S(\sigma):=\{e=uv\in E: \sigma_u\sigma_v=1\}$ is the set of satisfied edges of a spin configuration $\sigma$ and 
$
  U(\sigma) = E\setminus S(\sigma)  
$
is the set of unsatisfied edges. 

\begin{lemma}\label{lem:XOR_Ising_as_partition_function}
The XOR-Ising measure can be represented as follows, 
\[
\XORI[\sigma] \propto Z_{S(\sigma), 2 J}.% \prod_{e \in U(\sigma) } \hspace{-5pt}  \sech (2\beta J_e) \sum_{F \subset E} \prod_{e \in F } \tanh(2\beta J_e).
\]
\end{lemma}
\begin{proof}
First, the Ising measure can be rewritten as
\begin{align}\label{eq:Ising_model_from_unsatis}
\Ising[\sigma] = \prod_{e\in E} e^{J_e} \prod_{e \in U(\sigma)} e^{- 2 J_e}/ Z_{G,J},
\end{align}
where $
Z_{G,J} = \prod_{e \in E} e^{J_e}  \sum_{\sigma \in \{-1,1\}^V}\prod_{e \in U(\sigma)} e^{- 2 J_e}.
$
Let $\sigma_1 \otimes \sigma_2$ denote the pointwise product of $\sigma_1$ and $\sigma_2$. Using that $U(\sigma \otimes \sigma_2) = U(\sigma) \triangle U(\sigma_2)$, 
    \begin{align*}
\XORI[\sigma] &= \sum_{\sigma_1, \sigma_2} \Ising[\sigma_1] \Ising[\sigma_2] \id[\sigma_1 \otimes \sigma_2 = \sigma]
=  \sum_{\sigma_1, \sigma_2} \Ising[\sigma_1 \otimes \sigma_2] \Ising[\sigma_2] \id[\sigma_1= \sigma] 
= \sum_{\sigma_2} \Ising[\sigma \otimes \sigma_2] \Ising[\sigma_2] \\
&\propto \sum_{\sigma_2} \prod_{e\in U(\sigma \otimes \sigma_2)} e^{-2 J_e} \prod_{e\in U( \sigma_2)} e^{-2 J_e}
\propto \prod_{e\in U(\sigma)}e^{-2 J_e}  \sum_{\sigma_2} \prod_{e\in U( \sigma_2)\setminus U(\sigma)} e^{-4 J_e}\\
&= \prod_{e\in U(\sigma)}e^{-2 J_e}Z_{S(\sigma), 2 J} \prod_{e \in S(\sigma)} e^{- 2 J_e}  \propto Z_{S(\sigma), 2 J}, 
\end{align*}
as claimed.
\end{proof}
    
\section{Additional Ising Couplings and proofs}\label{sec:Ising_proofs}
Before embarking on the proofs of \Cref{prop:doublecurrent-Xor} and \Cref{prop:doublecurrent-loop}, let us state two additional couplings. The four propositions together generalize all couplings in  \Cref{fig:couplings_horizontal}.  Proposition \ref{prop:randomcluster-Ising} with $A=\emptyset$ is the Edwards-Sokal coupling for the Ising model and the random-cluster model with $q=2$ \cite{edwards1988generalization}, which is derived from the Swendsen-Wang Monte Carlo algorithm for the Potts model \cite{swendsen1987nonuniversal}. Proposition~\ref{prop:randomcluster-Ising} with $A \neq \emptyset$ is new.

\begin{proposition}\label{prop:randomcluster-Ising}
	Let $G=(V,E)$ be a finite connected graph. Let $\Omega:=\{0,1\}^E$ be the space of bond configurations and $\Sigma:=\{-1,+1\}^V$ be the space of spin configurations. Define $f:\Omega\to 2^{\Sigma}$ by 
$
      f(\omega) = \{ \sigma \in \Sigma \mid \omega_e = 0, \forall e \not \in S(\sigma) \}.
$ 
For any $A\subset V$ with $|A|$ even, let
	\[\rho[\omega] \propto \id[\omega\in\mathcal{F}_A]\Prb_p[\omega],~\forall \omega\in\Omega,~\gamma[\sigma]\propto 1, ~\forall \sigma\in \Sigma,\]
    where $p=(1-e^{-2J_e})_{e\in E}$.
	Then, under the coupling measure $\mathscr{P}$ defined in Theorem \ref{thm:generalcoupling}, 
	\begin{enumerate}
		\item[(a)] The marginal of $\mathscr{P}$ on $\Sigma$ is $\mu^A$ defined by
		$\mu^A[\sigma] \propto \mu[\sigma] \Prb_{S(\sigma),p}[\mathcal{F}_A].$
		 For each $\omega\in\Omega$ with $\omega\in\mathcal{F}_A$, the conditional measure $\mathscr{P}[\;\cdot \mid \omega]$ on $\Sigma$ is realized by tossing independent fair coins to get $\pm$ for each open cluster.
		\item[(b)] The marginal of $\mathscr{P}$ on $\Omega$ is $\phi_p[\;\cdot \mid \mathcal{F}_A]$.  For each $\sigma\in\Sigma$ with $S(\sigma)\in\mathcal{F}_A$, the conditional measure $\mathscr{P}[\;\cdot \mid \sigma]$ on $\Omega$ is the independent Bernoulli percolation on $S(\sigma)$ with parameter $p$ conditioned on $\mathcal{F}_A$,  $\Prb_{S(\sigma),p}[\;\cdot \mid \mathcal{F}_A]$.
	\end{enumerate}
\end{proposition}

The second proposition is already known in the literature, but we state it for completeness and to show how it follows from \Cref{thm:generalcoupling}.

\begin{proposition}[\cite{aizenman2019emergent,grimmett2007random}]\label{prop:randomcluster-loop}
	Let $G=(V,E)$ be a finite connected graph. Let $\Omega:=\{0,1\}^E$ be the space of bond configurations. For $A\subset V$ with $|A|$ even, let $\Sigma:=\{F\subset E: \partial F=A\}$ be the set of subgraphs of $G$ with sources $A$. Let $f:\Omega\to 2^{\Sigma}$ be defined by $f(\omega)=\{F\subset \omega: \partial F=A\}$. Let
	\[\rho[\omega]\propto \id[\omega\in\mathcal{F}_A] \Prb_t[\omega],~\forall \omega\in \Omega, ~\gamma[\eta]\propto 1, ~\forall \eta\in \Sigma.\]
	Then, under the coupling measure $\mathscr{P}$ defined in Theorem \ref{thm:generalcoupling}, 
	\begin{enumerate}
		\item[(a)] The marginal of $\mathscr{P}$ on $\Sigma$ is $\ell^{A}$. For each $\omega\in\Omega$ with $\omega\in\mathcal{F}_A$, the conditional measure $\mathscr{P}[\;\cdot \mid \omega] = \mathtt{UG}^A[\omega] $  on $\Sigma$ is the uniform subgraph on $f(\omega)$.
		\item[(b)] The marginal of $\mathscr{P}$ on $\Omega$ is $\phi_p[\;\cdot \mid \mathcal{F}_A]$.  For each $\eta\in\Sigma$ with $\partial \eta=A$, the conditional measure $\mathscr{P}[\;\cdot \mid \eta]  = (\Prb_t \cup \delta_{\eta})[\cdot]$, that is,  $\eta$ sprinkled with Bernoulli percolation with parameters $t=(t_e)_{e\in E}$. 
	\end{enumerate}
\end{proposition}

The law of total probability (cf. \Cref{rem:gamma}) implies $\phi_p[\omega \mid \mathcal{F}_A] = (\Prb_t \cup \ell^A)[\omega].$
When $A=\emptyset$, Proposition \ref{prop:randomcluster-loop} says that one may couple the random even subgraph and the random-cluster model with $q=2$; this is proved in \cite[Section 3]{grimmett2007random} (see also \cite{evertz2002new} and \Cref{prop:conditioned_Bernoulli1}). Proposition \ref{prop:randomcluster-loop} with $A\neq \emptyset$ is proved in \cite[Theorem 3.2]{aizenman2019emergent}.

\subsubsection{Simplifications of conditional measures}\label{sec:simplifications_of_conditional} We first note that in Theorem \ref{thm:generalcoupling} with $\gamma$ the uniform measure on $\Sigma$ and thus in Propositions \ref{prop:randomcluster-Ising}-\ref{prop:randomcluster-loop} and Propositions \ref{prop:doublecurrent-Xor}-\ref{prop:doublecurrent-loop}, the conditional measure $\mathscr{P}[\;\cdot \mid \omega]$ is always the uniform measure on $f(\omega)$. Let us extend the comments made in the introduction on conditional measures: For Propositions \ref{prop:randomcluster-Ising}-\ref{prop:randomcluster-loop} and Propositions \ref{prop:doublecurrent-Xor}-\ref{prop:doublecurrent-loop}, it is always the case that
\[\rho[\omega]=\varphi[\cdot | \mathcal{F}_A]\]
for some probability measure $\varphi$ on $\{0,1\}^E$. Then, by part (d) of Theorem \ref{thm:generalcoupling},
\[\mathscr{P}[\;\cdot \mid \eta]=\rho[\;\cdot \mid f^{-1}(\eta)]=\varphi[\;\cdot \mid \mathcal{F}_A \cap f^{-1}(\eta)].\]
In Propositions \ref{prop:randomcluster-Ising} and \ref{prop:doublecurrent-Xor}, $f^{-1}(\sigma)=\{U(\sigma) \text{ is closed}\}$. So we have
\[\mathscr{P}[\;\cdot \mid \sigma]=\varphi[\;\cdot \mid \mathcal{F}_A \cap \{U(\sigma) \text{ is closed}\}]=\varphi_{S(\sigma)}[\;\cdot \mid \mathcal{F}_A].\]
In Proposition \ref{prop:randomcluster-loop} and \ref{prop:doublecurrent-loop}, for each $\eta\subset E$ with $\partial \eta=A$, $f^{-1}(\eta)=\{\eta \text{ is open}\}$. Since $\{\eta \text{ is open}\}\subset \mathcal{F}_A$, 
\[\mathscr{P}[\;\cdot \mid \eta]=\varphi[\;\cdot \mid \mathcal{F}_A \cap \{\eta \text{ is open}\}]=\varphi[\;\cdot \mid \{\eta \text{ is open}\}].\]
%So the interpretations of those two conditional measures in Proposition \ref{prop:randomcluster-Ising}-\ref{prop:doublecurrent-loop} are clear; but 
We will also provide a slightly different interpretation for $\mathscr{P}[\;\cdot \mid \eta]$ in Proposition \ref{prop:doublecurrent-loop}.

\subsection{Proof of Proposition \ref{prop:randomcluster-Ising}: coupling of \texorpdfstring{$\phi[\cdot|\mathcal{F}_A]$}{ph} and \texorpdfstring{$\mu^A$}{ma}}
By \Cref{thm:generalcoupling}, the marginal of $\mathscr{P}$ on $\Sigma$ is
\begin{align*}
    \mathscr{P}_{\Sigma}[\sigma] &\propto \rho[f^{-1}(\sigma)]\gamma[\sigma] \propto \sum_{\omega\in \Omega} \id[\sigma\in \hspace{-1pt}f(\omega)]\id[\omega\in\mathcal{F}_A]\Prb_p[\omega]\\
    &\propto \hspace{-3pt}\prod_{e\in U(\sigma)}\hspace{-4pt}(1-p_e)\hspace{-11pt}\sum_{\omega\in\{0,1\}^{S(\sigma)}} \hspace{-10pt} \id[\omega\in\mathcal{F}_A] \prod_{e\in S(\sigma)} {p_e}^{\omega_e}(1-p_e)^{1-\omega_e} \\
    &\propto \prod_{e\in U(\sigma)} e^{-2J_e} \Prb_{S(\sigma),p}[\mathcal{F}_A] \propto \mu[\sigma] \Prb_{S(\sigma),p}[\mathcal{F}_A], 
\end{align*}
where we used \eqref{eq:Ising_model_from_unsatis} in the last proportionality. Similarly, the marginal of $\mathscr{P}$ on $\Omega$ is
\[\mathscr{P}_{\Omega}(\omega)\propto \gamma[f(\omega)] \rho[\omega] \propto |f(\omega)|\rho[\omega] \propto 2^{\kappa(\omega)} \Prb_p[\omega] \id[\omega\in \mathcal{F}_A]. \]

\subsection{Proof of Proposition \ref{prop:doublecurrent-Xor}: coupling of \texorpdfstring{$\Prbcur^{A,B}$}{dc} and \texorpdfstring{$\mu^{A,B}$}{mab}}\label{sec:XOR_Ising_coupling}
By \Cref{thm:generalcoupling}, the marginal of $\mathscr{P}$ on $\Sigma$ is
\begin{align*}
    \mathscr{P}_{\Sigma}[\sigma] &\propto \rho[f^{-1}(\sigma)]\gamma[\sigma] \propto \sum_{\omega\in \Omega} \id[\sigma\in f(\omega)] \id[\omega\in\mathcal{F}_A]\Prbcur_{2J}^{A \triangle B}[\omega] \\
    &\propto \prod_{e\in U(\sigma)}\hspace{-2pt}\hspace{-5pt}  \sech(2J_e)\hspace{-15pt} \sum_{\omega\in\{0,1\}^{S(\sigma)}}\hspace{-10pt}\id[\omega\in\mathcal{F}_A] \hspace{-20pt}\sum_{F\subset \omega: \partial F=A\triangle B}\prod_{e\in F} \hspace{-3pt} \tanh(2J_e) \hspace{-3pt} \prod_{e\in\omega\setminus F}\hspace{-7pt}(1-\hspace{-5pt}  \sech(2J_e))\hspace{-10pt}\prod_{e\in S(\sigma)\setminus \omega}\hspace{-12pt}  \sech(2J_e)\\
    &\propto \prod_{e\in U(\sigma)} \hspace{-5pt}  \sech(2J_e) \Prbcur_{S(\sigma),2J}^{A \triangle B}[\mathcal{F}_A]
    \frac{Z_{S(\sigma),2J}^{A \triangle B}}{Z_{S(\sigma),2J}} Z_{S(\sigma),2J} \prod_{e\in S(\sigma)} \hspace{-5pt}  \sech(2J_e)\\
    &\propto \Prbcur_{S(\sigma),2J}^{A \triangle B} [\mathcal{F}_A] \langle \tau_{A \triangle B} \rangle_{S(\sigma),2J} Z_{S(\sigma),2J} \propto \Prbcur_{S(\sigma),2J}^{A \triangle B} [\mathcal{F}_A] \langle \tau_{A \triangle B} \rangle_{S(\sigma),2J} \XORI[\sigma],
\end{align*}
using Lemma \ref{lem:XOR_Ising_as_partition_function} in the last proportionality and in the third proportionality that the partition function for $\Prbcur_{S(\sigma),2J}^{C}$ is
\begin{align*}
&2^{-|V|}Z_{S(\sigma),2J}^{C}\prod_{e\in S(\sigma)}\hspace{-5pt}  \sech(2J_e)\\
& \quad = \sum_{\omega \in \{0,1\}^{S(\sigma)}} \sum_{F\subset \omega: \partial F= C} \prod_{e\in F} \hspace{-5pt}\tanh(2J_e) \prod_{e\in \omega\setminus F}(1-\hspace{-5pt}  \sech(2J_e)) \prod_{e\in S(\sigma) \setminus \omega} \hspace{-12pt}  \sech(2J_e)\\
&\text{with } Z^{C}_{S(\sigma),2J}:=\sum_{\sigma\in\{-1,+1\}^V}\sigma_C\exp[\sum_{uv\in S(\sigma)} 2J_{uv}\sigma_u\sigma_v]. 
\end{align*}
Similarly, the marginal of $\mathscr{P}$ on $\Omega$ can be computed using \eqref{eq:double_J} to be the double random current: 
\begin{align*}
    \mathscr{P}_{\Omega}[\omega]\propto \gamma[f(\omega)] \rho[\omega] \propto |f(\omega)|\rho[\omega]  \propto 2^{\kappa(\omega)} \id[\omega\in\mathcal{F}_A]\Prbcur_{2J}^{A \triangle B}[\omega] \propto \Prbcur^{A,B}_{J}[\omega]. 
\end{align*}

\subsection{Proof of Proposition \ref{prop:randomcluster-loop}: coupling of \texorpdfstring{$\phi[\cdot | \mathcal{F}_A]$}{pf} and \texorpdfstring{$\ell^A$}{ela}}
By \Cref{thm:generalcoupling}, for each $\eta\subset E$ with $\partial \eta=A$, the marginal of $\mathscr{P}$ on $\Sigma$ is:
\begin{align*}
    \mathscr{P}_{\Sigma}[\eta] &\propto \rho[f^{-1}(\eta)]\gamma[\eta]\propto \sum_{\omega\in \Omega} \id[\eta\in f(\omega)] \id[\omega\in\mathcal{F}_A]\Prb_t[\omega] \propto \sum_{\omega\in \Omega} \id[\eta \subset \omega] \Prb_t[\omega] \propto t^{\eta} \propto \ell^{A}[\eta].
\end{align*}

Similarly, for each $\omega\in\Omega$ with $\omega\in\mathcal{F}_A$, we have
\begin{align*}
    \mathscr{P}_{\Omega}[\omega]\propto \gamma[f(\omega)] \rho[\omega] \propto |f(\omega)|\rho[\omega] \propto  \Prb_t[\omega] \abs{\mathcal{E}_A(\omega)}= \Prb_t[\omega] \abs{\mathcal{E}_{\emptyset}(\omega)} \propto \phi_p[\omega | \mathcal{F}_A],
\end{align*}
where we used the switching principle \eqref{eq:switching_principle} and \eqref{eq:definition_random_cluster} in the last two equalities respectively.

\subsection{Proof of Proposition \ref{prop:doublecurrent-loop}: coupling of \texorpdfstring{$\Prbcur^{A,B}$}{PbAB} and \texorpdfstring{$\ell^A$}{ela}}
By \Cref{thm:generalcoupling}, the marginal of $\mathscr{P}$ on $\Sigma$ is: for each $\eta\subset E$ with $\partial \eta=A$,
\begin{align*}
    \mathscr{P}_{\Sigma}[\eta] &\propto \rho[f^{-1}(\eta)]\gamma[\eta] \propto \sum_{\omega\in \Omega} \id[\eta\in f(\omega)] \id[\omega\in\mathcal{F}_A](\ell^{A \triangle B} \cup \Prb_{t^2})[\omega]\\
    & = \sum_{\omega\in \Omega} \id[\eta \subset \omega] (\ell^{A \triangle B} \cup \Prb_{t^2})[\omega]\propto \sum_{F\subset E, \partial F= A \triangle B} t^F (t^2)^{\eta\setminus F} \\ 
    &\propto t^{\eta} \sum_{F\subset E, \partial F= A \triangle B} t^{\eta \triangle F} \propto t^{\eta} \sum_{\tilde{F}\subset E, \partial \tilde{F}= B} t^{\tilde{F}} \propto t^{\eta}.% \propto \ell^A[\eta].
\end{align*}
For each $\eta\in\Sigma$ with $\partial \eta=A$, we have
\begin{align*}
    \mathscr{P}[\omega \mid \eta]&\propto \frac{\id[\eta \subset \omega] (\ell^{A \triangle B} \cup \Prb_{t^2})[\omega]}{t^{\eta}} \propto \frac{\id[\eta \subset \omega] (s^2)^{E\setminus \omega}\sum_{F\subset \omega, \partial F= A \triangle B} t^F (t^2)^{\omega\setminus F}}{t^{\eta}} \\
    &= \id[\eta \subset \omega] (s^2)^{E\setminus \omega} \sum_{F\subset \omega, \partial F= A \triangle B} t^{\eta \triangle F} (t^2)^{\omega\setminus (\eta \cup F)} \\
    &= \id[\eta \subset \omega] (s^2)^{E\setminus \omega} \sum_{\tilde{F}\subset \omega, \partial \tilde{F}= B} t^{\tilde{F}} (t^2)^{\omega\setminus (\eta \cup \tilde{F})}\propto  \id[\eta \subset \omega](\ell^B \cup \Prb_{t^2} \cup \delta_{\eta})[\omega].
\end{align*}

Similarly, the marginal of $\mathscr{P}$ on $\Omega$ is: for each $\omega\in\Omega$ with $\omega\in\mathcal{F}_A$,
\begin{align*}
    \mathscr{P}_{\Omega}[\omega] \propto |f(\omega)|\rho[\omega] \propto (\ell^{A \triangle B} \cup \Prb_{t^2})[\omega] \abs{\mathcal{E}_A(\omega)}
    =(\ell^{A \triangle B} \cup \Prb_{t^2})[\omega] \abs{\mathcal{E}_{\emptyset}(\omega)} \propto \Prbcur^{A,B}[\omega],
\end{align*}
where we used the switching principle \eqref{eq:switching_principle} in the last equality and Proposition \ref{prop:Lisformula} in the last proportionality.

Let $G=(V,E)$ be a finite connected graph. For any $t\in (0,1),$ the set of measures $\nu$ such that the uniform even graph of $\nu$ has the law of $\ell_{G,t}$ is convex, which is evident from the interpretation of the convex combination $\alpha \nu_1 + (1-\alpha) \nu_2$ as sampling a $\nu_1$-random configuration with probability $\alpha$  and a  $\nu_2$-random configuration with probability $1-\alpha$.     By \Cref{prop:doublecurrent-loop}, for every set $A \subset V$ of even cardinality, the uniform even graph of $\Prbcur^{\emptyset,A}$ is $\ell_t$. They cannot be written as convex combinations of each other:

\begin{comment}We speculate whether the following should hold true:
\begin{conjecture}\label{prop:convex_UEG_is_loop_O(1)}
   Let $G = (V,E)$ be a finite connected graph. For any $t \in (0,1)$, the set of measures $\nu$ on $\{0,1\}^E$ such that the uniform even graph of $\nu$ has the law of $\ell_{G,t}$ has exactly $2^{\abs{V}-1}$ extreme points. 
\end{conjecture}
\end{comment}

\begin{lemma} \label{lem:non_trivial_convex_combination_of_double_current}
Let $n\in \mathbb{N}$. If $(\alpha_j)_{1\leq j\leq n}\subseteq (0,1),$ $\sum_{j=1}^n\alpha_j=1$  and $\Prbcur^{\emptyset,A} = \sum_{j=1}^n\alpha_j \Prbcur^{\emptyset,A_j},$ then $A=A_j$ for every $1\leq j\leq n.$ In particular, there are no non-trivial convex combinations $\Prbcur^{\emptyset,A} = \sum_{j=1}^n\alpha_j \Prbcur^{\emptyset,A_j}.$
\end{lemma}
\begin{proof}
Suppose that $\Prbcur^{\emptyset,A}=\sum_{j=1}^n\alpha_j\Prbcur^{\emptyset,A_j}$ for some $\alpha_j\in (0,1)$ with $\sum_{j=1}^{n}\alpha_j=1$.
Plugging in the event $\mathcal{F}_A$ yields that $\Prbcur^{\emptyset,A_j}[\mathcal{F}_A] = 1$ for every $j$. Using \eqref{eq:definition_double_current} with the observation that there is a positive probability that $\ell^{\emptyset}\otimes \Bernoulli_{t^2}$ outputs two empty graphs, these relations imply that 
 $\ell^{A_j}[\mathcal{F}_A] =  1 $ for every $j$. Similar considerations prove that 
$\ell^A[\cup_{j=1}^n \mathcal{F}_{A_j}]=1.$ This implies two combinatorial facts: First of all, any configuration $\eta$ with $\partial \eta\in \{A_j\mid 1\leq j\leq n\}$ has a subgraph $\eta'$ with $\partial \eta'=A.$ Conversely, every configuration $\eta$ with $\partial \eta=A$ has a subgraph $\eta'$ with $\partial\eta'\in \{A_j\mid 1\leq j\leq n\}$. Letting $\eta$ be a minimal configuration with $\partial \eta=A,$ we get a subgraph $\eta'\subset \eta$ with $\partial \eta'\in \{A_j\mid 1\leq j\leq n\}$ and a subgraph $\eta''\subset \eta'$ with $\partial \eta''=A.$ By minimality of $\eta,$ we must have $\eta''=\eta'=\eta$ and hence, $A\in \{A_j\mid 1\leq j\leq n\}$, yielding that $\mathbf{P}^{\emptyset,A}$ can be written as a convex combination of the $n-1$ remaining terms $\mathbf{P}^{\emptyset,A_j}$. Proceeding by induction, we conclude that $A=A_j$ for every $1\leq j\leq n$, which is what we wanted.
\end{proof}
The fact that $\mathbf{P}^{\emptyset,A}[\mathcal{F}_A]=1$ gives some evidence that $\mathbf{P}^{\emptyset,A}$ should be extremal in the convex set of measures with $\ell_{G,t}$. Since $\mathbf{P}^{\emptyset,\emptyset}$ is not contained in the convex hull of the other $\mathbf{P}^{\emptyset,A},$ this indicates that the set should contain at least $2^{\abs{V}-1}$ points. It is natural to ask about the number of extreme points. 
\begin{question}
       Let $G = (V,E)$ be a finite connected graph. For any $t \in (0,1)$, let $\mathcal{U}$ denote the convex set of measures $\nu$ on $\{0,1\}^E$ such that the uniform even graph of $\nu$ has the law of $\ell_{G,t}$. How many extreme points does $\mathcal{U}$ have?
\end{question}

Finally, let us prove a claim about disintegration by Bernoulli percolation that was made in the main text. 
\begin{claim}\label{claim:Bernoulli_uniqueness}
    For any finite graph $G =(V,E)$ and any $p \in (0,1)$ and any probability measure $\mu$ on $\{0,1\}^E$, there is at most one probability measure $\nu$ on $\{0,1\}^E$, such that 
    $
    \mu = \nu \cap \Bernoulli_p
    $
    and at most one probability measure $\gamma$ on $\{0,1\}^E$ such that 
       $
    \mu = \gamma \cup \Bernoulli_p.
    $
\end{claim}
\begin{proof}
    To prove the first uniqueness statement assume   $
    \mu = \nu \cap \Bernoulli_p
    $ and consider the full graph $E$. It holds that $\mu[E] = p^{\abs{E}} \nu[E],$ which uniquely determines $\nu[E]$.   Now, iteratively, let $\omega\subseteq E$ and suppose that $\nu[\eta]$ is determined for all supergraphs $\omega \subsetneq \eta$. Then, $\nu[\omega]$ is uniquely determined by the equation
   % $\mu[\eta] = \sum_{\omega \supset \eta} p^{|\eta|}(1-p)^{|\omega\setminus \eta|}\nu[\omega].$ 
    $\mu[\omega] = \sum_{\eta \supset \omega} p^{|\omega|}(1-p)^{|\eta\setminus \omega|}\nu[\eta].$The second statement is proven similarly, by iterating starting from the empty graph or by applying the first statement to the complement.  
\end{proof}

\section{Coupling measures for conditional Bernoulli percolation: Examples}\label{sec:coupling_examples}

There is a myriad of other potential models beyond the ones mentioned above where a coupling in the form of \Cref{prop:conditioned_Bernoulli1} or \Cref{prop:conditioned_Bernoulli2} could be applied. Examples include the random $d$-regular graph, the hard-core model,  random graphs with forbidden vertex degrees \cite{grimmett2010random}, the maximum planar subgraph, $k$-out subgraphs \cite{frieze2017random}, corner percolation \cite{pete2008corner}, dimers \cite{kenyon2014conformal} and weighted, minimal or degree-constrained spanning  trees. In \cite{bauerschmidt2017local}, the local law for random $d$-regular graphs was computed via a local modification of the configuration, which may be seen as a much more complicated relative of \Cref{prop:conditioned_Bernoulli1}.
Let us go slightly more into detail with two such models. 
 %See also \cite{barbu2005generalizing} for some applications of the Swendsen-Wang type algorithms to computer vision.  

\subsection{The antiferromagnetic Ising model} \label{sec:anti_ferromagnetic_Ising}
In \Cref{prop:conditioned_Bernoulli2}, let $x_e> 1, p_e=1-1/x_e$ for each $e\in E$ and $\Sigma:=\{\eta \in\{0,1\}^E: \partial \eta=\emptyset\}$ be the set of all even subgraphs of $G$. Then, $\mathscr{P}_{\Sigma}$ is the (antiferromagnetic) loop measure $\ell_{G,x}$ with edge weight $x_e>1$ defined through the formula \eqref{eq:loopO(1)def}. \Cref{prop:conditioned_Bernoulli2} says that $\ell_{G,x}$ has the following resampling property: $\ell_{G,x}$ can be obtained by sampling $\omega\sim \ell_{G,x}\cap\mathbb{P}_{G,1-1/x}$ and, conditionally on $\omega$, sampling a uniform even supergraph of $\omega$. 
\\
\\ Assume in the following that $G$ is planar with nearest neighbor interactions, we now explain that if $\omega \sim \ell_{G,x}\cap \mathbb{P}_{G,1-1/x},$ then the dual edges corresponding to (the open edges in) $\omega$ are distributed like (the open edges from) the random-cluster representation of the antiferromagnetic Ising model with couplings $(-\log(x_e)/2)_{e\in E^*}$ which can be found in \cite{kasai1988percolation,Newman1990Rep,Newman1994Rep}; see also \cite{aizenman2025geometric} for more geometric representations related to frustration.
Let $G^*=(V^*,E^*)$ be the dual graph of $G$. Note that since $G$ is finite, $V^*$ contains the unique infinite face of $G$. The dual edge of $e\in E$ is denoted by $e^*$, and $\omega_{e^*}=1$ if and only if $\omega_e=0$. First note that 
%\todo[inline]{F: I'd write $\mathbb{P}_{G,1-1/x}[\omega \text{ open},\eta\setminus \omega \text{ closed}]$}
\begin{align}\label{eq:psi}
\ell_{G,x}\cap \mathbb{P}_{G,1-1/x}[\omega]& \hspace{-2pt}\propto \sum_{\eta: \omega \subset \eta, \partial \eta=\emptyset} \hspace{-12pt}\ell_{G,x}[\eta]\mathbb{P}_{G,1-1/x}[\text{}\omega \text{ is open}, \text{}\eta\setminus \omega \text{ is closed}]\nonumber\\
    &\propto \prod_{e\in \omega}(x_e-1) \abs{\{\eta \mid \omega\subset \eta, \partial \eta=\emptyset\}}\propto \prod_{e\in \omega}(x_e-1) \abs{\mathcal{E}_{\partial \omega}(\omega^c)}\nonumber\\
    &=\prod_{e\in \omega}(x_e-1) \abs{\mathcal{E}_{\emptyset}(\omega^c)}\id[\omega \text{ has an even supergraph}],
\end{align}
where $\omega^c:=E\setminus \omega$ and we used the switching principle \eqref{eq:switching_principle} in the last equality. Euler's formula says
\[\kappa(\omega^c)=\abs{V}-\abs{\omega^c}+f(\omega^c)-1,\]
where $f(\omega^c)$ is the number of faces of $\omega^c$ (including the unique infinite face). Since there is a one-to-one correspondence between the faces in $\omega^c$ and the clusters of
\newline \noindent $(\omega^c)^*:=\{e^*\in E^* \mid e\in \omega\}$, we have $f(\omega^c)=\kappa((\omega^c)^*)$. Combining these with \eqref{eq:even_formula}, we get
\begin{equation}\label{eq:even_formula_1}
    \abs{\mathcal{E}_{\emptyset}(\omega^c)}=2^{\kappa((\omega^c)^*)-1}.
\end{equation}

We claim that $\omega$ has an even supergraph if and only if $(\omega^c)^*$ is bipartite. Suppose $\omega$ has an even supergraph $\tilde{\omega}$. Then,  $\tilde{\omega}$ can be viewed as Ising interfaces living on $G$ (so Ising spins live on $G^*$). So $(\tilde{\omega}^c)^*=\{e^*\in E^* \mid e\in \tilde{\omega}\}$  is bipartite (the set of $+$ spins and its complement form a partition of $V^*$). Therefore, $(\omega^c)^*\subset (\tilde{\omega}^c)^*$ is also bipartite. On the other hand, suppose $(\omega^c)^*$ is bipartite. Then, one may assign $+$ or $-$ to each vertex in $V^*$ such that each edge in $(\omega^c)^*$ has endpoints with different signs. Let $\eta:=\{uv \in E^*: u \text{ and } v \text{ have different signs}\}$ be the set of all edges with endpoints having different signs. Then, $(\eta^c)^*$ can be viewed as the Ising interfaces living on $G$ and thus $(\eta^c)^*$ is even. Since $(\omega^c)^*\subset \eta$, we get $\omega \subset (\eta^c)^*$, which completes the proof of the claim.

Combining the claim with \eqref{eq:psi} and \eqref{eq:even_formula_1},
\[\ell_{G,x}\cap\mathbb{P}_{G,1-1/x}[\omega]\propto \prod_{e\in \omega^*}(1/x_e)\prod_{e\in (\omega^c)^*}(1-1/x_e)2^{\kappa((\omega^c)^*)}\id[(\omega^c)^* \text{ bipartite}]. \]
That is, for each $\omega\sim \ell_{G,x}\cap\mathbb{P}_{G,1-1/x}$, $(\omega^c)^*=\{e^*\in E^* \mid e\in \omega\}$ has the distribution of $\phi_{G^*,1-1/x}[\cdot\mid \text{bipartite}]$.

\subsection{The arboreal gas} One could hope that  \Cref{prop:conditioned_Bernoulli2} could shed some light on the arboreal gas $\mathcal{A}$, where $\Sigma = \{F \subset \Omega \mid F \text{ is a forest} \}$ and conventionally $x_e = \beta$, for each $e \in E$. 
It is conjectured that the two-dimensional model should have exponential decay for all values of $\beta$ \cite{CaraccioloForests} (with a Polyakov type scaling as $\beta \to \infty$).  As of now, only absence of percolation has been proven \cite{Bauerschmidt2021PlaneForest}. In three dimensions and above, the measure has a regime of percolation \cite{Bauerschmidt2024PercoForest}. 
If we fix a $\beta\gg1$ and set $p = 1- \frac{\varepsilon}{\beta}$ for some $\varepsilon\ll 1$,
then, to obtain $\eta\sim \mathcal{A}_{\beta}\cap \mathbb{P}_{p},$ a tiny fraction of the edges are deleted from the original forest $F\sim \mathcal{A}_{\beta}$. But since every cluster of $F$ is a tree, the measure $\mathcal{A}_\beta \cap \mathbb{P}_{p}$ nevertheless has exponential decay; $\mathcal{A}_\beta \cap \mathbb{P}_{p}[v \cc w] \leq p^{\text{dist}(v,w)}$, for all $v, w \in V$. The coupling tells us that to recover the full arboreal gas $\mathcal{A}_\beta,$ one would need to take an arboreal gas with parameter $\frac{\varepsilon}{1+\varepsilon}$ on the graph where all clusters of $\eta$ are collapsed to single vertices. Since the set of forests forms a decreasing event, by FKG, this conditional measure is stochastically dominated by Bernoulli percolation with parameter $\mathbb{\varepsilon}$. In conclusion, edges deleted independently with probability $\frac{\varepsilon}{\beta}$ can be more than compensated by a follow-up independent addition of edges with probability $\varepsilon$.

In two dimensions, one can potentially also use that the dual graph to a forest is a connected graph, which also fits into the framework of \Cref{prop:conditioned_Bernoulli1} and \Cref{prop:partial_coupling} by letting $\Sigma$ be the set of connected graphs.

\section{Polynomial lower bound for the supercritical $q$-flow model}\label{sec:polynomial_lower_bound_q_flow}
In this appendix, we show how the loop-cluster coupling \Cref{prop:uniform flows} can be used to show an equality of critical points for the $q$-state Potts model and the $q$-flow model. 
 The following follows the arguments made in the case $q=2$ from \cite[Section 4]{hansen2023uniform} fairly closely, and we note that similar tricks of exploiting the topology of the torus are well-established in the literature, see e.g., \cite{crawford2020macroscopic,duncanPRCM1,kozma2013lower}.

Let $\mathbb{T}_n^d=(\mathbb{Z}^d/2n\mathbb{Z}^d)$ denote the $d$-dimensional torus of period $2n$ and denote by $\mathfrak{C}_n$ the event on $\{0,1\}^{E(\mathbb{T}_n^d)}$ that $\omega$ contains a simple loop winding all the way around the torus once\footnote{A more formal definition might say that the pre-image of $\omega$ under the quotient map $\mathbb{Z}^d\to \mathbb{T}_n^d$ satisfies, that there is some cardinal direction $e_j$ and some vertex $v$ such that $v$ and $v+2ne_j$ are connected to each other inside the box $B_n\cup(2n e_j+B_n)$}. Since the event $\mathfrak{C}_n$ is translation-invariant, the sharp threshold theorem for monotonic measures \cite[Theorem 3.2]{GG06} implies that there exists a $\tilde{p}_c$ such that for all $p>\tilde{p}_c,$ there exists $C>0$ such that \begin{align} \label{eq:C_n}
    \phi^{q}_{\mathbb{T}_n^d,p}[\mathfrak{C}_n]\geq C
\end{align}for all $n$ and for $p<\tilde{p}_c$, there exists $C>0$ such that $\phi^{q}_{\mathbb{T}_n^d,p}[\mathfrak{C}_n]\leq \exp(-Cn)$. We claim that $\tilde{p}_c=p_c.$ Indeed, $\tilde{p}_c\geq p_c$  by sharpness of the phase transition on $\mathbb{Z}^d$ \cite{DCsharpness}. Conversely, if $p>p_c,$ one may prove that $\phi^q_{\mathbb{T}_n^d,p}[\mathfrak{C}_n]$ decays at worst polynomially, implying that $p_c\geq \tilde{p}_c.$

Roughly speaking, the event $\mathfrak{C}_n$ is topologically rigid: If $\omega\in \mathfrak{C}_n,$ then a density of its divergence free subgraphs will also lie in $\mathfrak{C}_n$. Below, $x_c=\frac{p_c}{p_c+q(1-p_c)}.$
\begin{theorem} \label{thm:q_state_potts_torus_trick}
Let $d\in \mathbb{N}$, $q\geq2$ be integer  and $x\in (x_c,1].$ Then, there exists $C>0$ such that for every $n\in \mathbb{N}$: 
$$
\ell^q_{\mathbb{T}_n^d,x}[0 \cc \partial B_{n-1}]\geq \frac{C}{n^{d-1}}.
$$
\end{theorem}
\begin{proof}

Let $p$ satisfy $x=\frac{p}{p+q(1-p)}$ and couple $\eta\sim \ell^q_{\mathbb{T}_n^d,x}$ and $\omega\sim \phi^q_{\mathbb{T}_n^d,p}$ via \Cref{prop:uniform flows}.

We first note that if $q$ is even, then sampling $\eta$ as a uniform divergence free form on $\omega$ and then taking all edge weights mod $2$ yields a uniform even subgraph of $\omega$. Thus, the theorem follows by the argument for $q=2$ case from \cite{hansen2023uniform}. We next present a proof which works for all $q$.

Fix an arbitrary co-dimension 1 hyperplane $H\subset \mathbb{T}_n^d$ with orthogonal direction $e_H$.
For any $\eta \in \ker(\partial)$, define the flow\footnote{Using that $\eta\in \ker(\partial)$, the flow is invariant under translating the hyperplane $H$ in the direction of $e_{H}$.} around the torus  in direction $e_H$ by $\mathcal{F}(\eta) =  \sum_{v \in H} \eta_{(v,(v+e_H))}.$  Let $\mathscr{P}$ be the coupling from \Cref{prop:uniform flows} for $G=\mathbb{T}_n^d$ and fix a random-cluster configuration $\omega\in \mathfrak{C}_n$ with a simple open oriented loop $\Gamma\subset\mathcal{O}(\omega)$ wrapping around the torus once and let $\gamma=\sum_{\mathfrak{e}\in \Gamma}\mathbf{1}^{\mathfrak{e}}.$

Then, $\gamma \in \ker(\partial)$ and
$\mathcal{F}(\gamma) \in \{\pm 1\}$. By linearity, we may assume $\mathcal{F}(\gamma)=1$.  
Furthermore, for any $k\in \N$,  $\eta\mapsto \eta+ k\cdot \gamma$ is a bijection from $\ker(\partial^{\omega})$ into itself and accordingly, conditionally on $\omega$, its uniform divergence-free coloring $\eta$ has the same distribution as $\eta+k\cdot \gamma.$ By linearity,
$
\mathcal{F}(\eta + k\cdot \gamma)  = \mathcal{F}(\eta) + k\mathcal{F}(\gamma) = \mathcal{F}(\eta) + k,
$
which means that the conditional probability that $\eta$ has non-zero flow is $\frac{q-1}{q}$, implying
\begin{align}
 \ell^q_{\mathbb{T}_n^d,x}[\mathcal{F}(\eta) \neq 0]
& = \sum_{\omega \in \Omega, \eta \in \Sigma} \mathscr{P}[\eta, \omega] \id[\mathcal{F}(\eta) \neq 0] \\
 &\geq \label{eq:coupling_use}
   \sum_{\omega \in \mathfrak{C}_n} \phi^q_{\mathbb{T}_n^d,p}[\omega]\sum_{\eta \in \Sigma}\mathscr{P}[\eta \mid \omega] \id[\mathcal{F}(\eta) \neq 0]=\frac{q-1}{q} \phi^q_{\mathbb{T}_n^d,p}[\mathfrak{C}_n].  
\end{align}
Now, let us argue that any $\eta \in \ker(\partial)$ with $\mathcal{F}(\eta) \neq 0$ has long loops. Let 
$$\tilde \eta(\mathfrak{e}) = \eta(\mathfrak{e}) \id[\mathfrak{e} \not \in  \cup_{v\in H}\{(v,v+e_H),(v+e_H,v)\}]
$$
be the configuration where edges connecting to one side of the hyperplane are set to 0. 
For every cluster $\mathcal{C}$ of $\textrm{supp}(\tilde \eta)$, $\sum_{v \in \mathcal{C}} (\partial \tilde\eta)_v = 0$, since every $(v,w)\in \mathcal{O}(\mathcal{C})$ contributes with $\tilde \eta_{(v,w)}$ at $w$ and $-\tilde \eta_{(v,w)}$ at $v$. Therefore, $\sum_{v } \id[v \overset{\tilde \eta}{\cc} H](\partial \tilde \eta)_v = 0$, as it is the sum over all clusters intersecting $H$. Since $\sum_{v \in H} (\partial \tilde \eta)_v =\mathcal{F}(\eta) \neq 0,$  there exists at least one vertex $w  \not \in H$ with $(\partial \tilde{\eta})_w\neq 0$ and $w\overset{\tilde{\eta}}\cc H$.
Since $\eta \in \ker(\partial)$, $(\partial \tilde \eta)_w = 0$ for all $w \not \in H \cup (H+e_H)$, such a vertex must lie in $H+e_H$ and thus
$H \overset{\tilde \eta}{\cc} (H+e_H)$. In particular, there is a long path in $\textrm{supp}(\eta)$ from $H$ to $H+e_H$. 
 Thus, by translation invariance, \eqref{eq:C_n} and \eqref{eq:coupling_use}, for some constants $c,C >0$, 
$$
\ell^q_{\mathbb{T}_n^d,x}[0\cc \partial B_{n-1}]\geq \frac{\ell^q_{\mathbb{T}_n^d,x}[\mathcal{F}(\eta) \neq 0]}{c n^{d-1}}\geq\frac{q-1}{qn^{d-1}}C,
$$
which finishes the proof.
\end{proof}

\section{Equivalence of two definitions of the plaquette random-cluster model}\label{sec:elementary}
In our rewriting of the plaquette random-cluster model, we relied on the universal coefficient theorem. In general, this needs a certain homology group to be a free $\mathbb{Z}/q\mathbb{Z}$-module, which might fail if $q$ is not prime. However, the appeal to the universal coefficient theorem turns out to be unnecessary. In this section, we give an elementary proof that the plaquette random-cluster model weights may be given in terms of the boundary map. Throughout, we use the notation for cubical complexes introduced in \Cref{sec:Lattice_Gauge}. The space $C_{k}$ comes equipped with a natural, symmetric, non-degenerate bilinear form
$$
\langle\sigma, \sigma'\rangle =\sum_{\mathbf{c}\in \Lambda_k} \sigma_{\mathfrak{c}} \sigma'_{\mathfrak{c}},
$$
where the product is taken in $\mathbb{Z}/q\mathbb{Z}$ and we pick\footnote{The reason for not just taking the average of the two choices is that the form would not be well-defined for $q=2$.} any of the two orientations of $\mathbf{c}$ at which to evaluate $\sigma$ (the quantity $\sigma_{\mathfrak{c}} \sigma'_{\mathfrak{c}}$ is invariant under flipping the orientation).

The boundary map $\partial_k$ admits an adjoint $d_k:C_{k-1}\to C_k$ characterized by the property that $\langle \partial_k \sigma,\sigma'\rangle=\langle \sigma,d_k\sigma'\rangle$ for all $\sigma\in C_k$ and $\sigma'\in C_{k-1}$. Loosely speaking, $d_k$ attaches to a $(k-1)$-cell all the $k$-cells for which the $(k-1)$-cell is in the boundary.\footnote{With \Cref{sec:duality} in mind, this immediately allows one to check that $d_k$ agrees with $\sigma\mapsto (\partial^{d-k+1}(\sigma^*))^*$ in the cases where the dual complex has defined boundary maps.}
Furthermore, this pairing respects rank-nullity, as we prove below. In the following, for abelian groups $H$ and $G$, $\textrm{Hom}(H,G)$ denotes the abelian group of homomorphisms from $H$ to $G$ (with addition given by pointwise addition).

\begin{lemma} \label{lemma:rank-nullity}
As abelian groups (and hence, as $\mathbb{Z}/q\mathbb{Z}$-modules) for all $k\geq 1$, 
$$
C_{k-1}/\ker(d_k)\cong \mathrm{Im}(\partial_{k}).
$$
In particular, $|C_{k-1}|=|\ker(d_k)|\cdot |\mathrm{Im}(\partial_k)|.$
\end{lemma}
\begin{remark}
Note that for $q$ prime, $\mathbb{Z}/q\mathbb{Z}$ is a field and therefore, the lemma is usual rank-nullity for adjoints.
\end{remark}
\begin{proof}
Let $\mathbb{G}$ be a finite $\mathbb{Z}/q\mathbb{Z}$ module (i.e. every element of $\mathbb{G}$ has order dividing $q$). By the classification of finitely generated abelian groups,  $\mathbb{G}\cong \oplus_{i\in I} G_i,$ where each $G_i$ is a cyclic group and a $\mathbb{Z}/q\mathbb{Z}$ module - in particular, $|G_i|$ divides $q$.
Since $G_i$ is cyclic, any $\phi\in \textrm{Hom}(G_i,\mathbb{Z}/q\mathbb{Z})$ is uniquely determined by $\phi(g)$ for a given generator $g$ of $G_i$. Accordingly, $\textrm{Hom}(G_i,\mathbb{Z}/q\mathbb{Z})\cong H$, where $H<\mathbb{Z}/q\mathbb{Z}$ is the subgroup of elements with order dividing $|G_i|.$
By the Euclidean division algorithm, $H$ is cyclic and generated by $k\cdot 1$ for the smallest $k\geq 1$ such that $k\in H,$ which is $k=q/|G_i|$. In particular, $G_i\cong H\cong \textrm{Hom}(G_i,\mathbb{Z}/q\mathbb{Z}).$ Extending over the direct sum, we get $\mathbb{G}\cong \textrm{Hom}(\mathbb{G},\mathbb{Z}/q\mathbb{Z}).$ In particular, this holds for $\mathbb{G}=C_{k-1}/\ker(d_k)$ and $\mathbb{G}=\textrm{Im}(\partial_k).$

Now, in our case, by non-degeneracy, 
$$
\sigma\in \ker(d_k)\quad \Longleftrightarrow \quad\langle d_k\sigma,\sigma'\rangle=0\quad \forall \sigma'\in C_{k} \quad \Longleftrightarrow \quad \langle \sigma,\partial_k\sigma'\rangle=0\quad  \forall \sigma'\in C_{k}.
$$
Accordingly, there is a well-defined bilinear form $C_{k-1}/\textrm{ker}(d_k)\times  \textrm{Im}(\partial_k)\to \mathbb{Z}/q\mathbb{Z}$ given by $$
(\sigma+\ker (d_k),\partial _k\sigma')\mapsto \langle\sigma,\partial_k\sigma'\rangle,
$$ and this bilinear form is nondegenerate. In particular, we get injective homomorphisms
$$
C_{k-1}/\ker d_k\hookrightarrow \textrm{Hom}(\textrm{Im}(\partial_k),\mathbb{Z}/q\mathbb{Z}) \qquad \qquad \textrm{Im}(\partial_k)\hookrightarrow \textrm{Hom}(C_{k-1}/\ker (d_k),\mathbb{Z}/q\mathbb{Z}),
$$
implying $|C_{k-1}/\ker(d_{k})|\leq |\textrm{Im}(\partial_k)|$ and $|\textrm{Im}(\partial_k)|\leq |C_{k-1}/\ker(d_k)|.$ In particular, the sizes are the same, and the injections are, in fact, isomorphisms.
\end{proof}
\begin{remark}\label{rem:nothing_special_gauge}
Just like in    \Cref{lemma:Counting_divergence}, we would get a similar statement for an arbitrary finite abelian coefficient group $\mathbb{G}$. The reason is again the classification of finitely generated abelian groups. One simply checks that if $C^{\mathbb{G}}_k$ denotes the $k$-chains with coefficients in $\mathbb{G},$ then $C^{\mathbb{G}\oplus \mathbb{G}'}_k\cong C_k^{\mathbb{G}}\oplus C_k^{\mathbb{G}'},$ and $\partial_k$ respects this splitting.
\end{remark}

Similarly to before, for a subcomplex $\omega$, we let $d_k^{\omega}$ denote the corresponding co-boundary map on $\omega$. The cohomological definition of the plaquette random-cluster model due to Duncan and Schweinhart \cite{duncanPRCM2} is given by the weights:
$$
\phi^q_{\Lambda_k,p}[\omega]\propto \left(\prod_{\mathbf{c}\in \Lambda_k (\omega)} \frac{p_{\mathbf{c}}}{1-p_{\mathbf{c}}}\right) \frac{|\ker (d_{k}^{\omega})|}{|\textrm{Im}(d_{k-1}^{\omega})|}=\left(\prod_{\mathbf{c}\in \Lambda_k (\omega)} \frac{p_{\mathbf{c}}}{1-p_{\mathbf{c}}}\right)|H^{k-1}(\omega,\mathbb{Z}/q\mathbb{Z})|.
$$
\begin{lemma} The plaquette random-cluster model satisfies
$$
\phi^q_{\Lambda_k,p}[\omega]\propto \left(\prod_{\mathbf{c}\in \Lambda_k (\omega)} \frac{p_{\mathbf{c}}}{1-p_{\mathbf{c}}}\right) \frac{|\ker (\partial_{k-1}^{\omega})|}{|\mathrm{Im}(\partial_{k}^{\omega})|}.
$$
\end{lemma}
\begin{proof}
Just like in the proof of \Cref{lemma:random_cluster_lattice_gauge_with_kernel}, $d_{k-1}^{\omega}=d_{k-1},$ $\partial_{k-1}^{\omega}=\partial_{k-1}$ and $C_{k-1}(\omega)=C_{k-1}$ for any $k$-spanning subcomplex $\omega$. So by \Cref{lemma:rank-nullity}, 
\begin{align*}
\phi^q_{\Lambda_k,p}[\omega] &\propto \left(\prod_{\mathbf{c}\in \Lambda_k (\omega)} \frac{p_{\mathbf{c}}}{1-p_{\mathbf{c}}}\right) {|\ker d_k^{\omega}|}=\left(\prod_{\mathbf{c}\in \Lambda_k (\omega)} \frac{p_{\mathbf{c}}}{1-p_{\mathbf{c}}}\right) \frac{|C_{k-1}(\omega)|}{ |\textrm{Im}(\partial^{\omega}_k)|}\\
&\propto  \left(\prod_{\mathbf{c}\in \Lambda_k (\omega)} \frac{p_{\mathbf{c}}}{1-p_{\mathbf{c}}}\right) \frac{|\ker(\partial^{\omega}_{k-1})|} {|\textrm{Im}(\partial^{\omega}_k)|},
\end{align*}
which was what we wanted.
\end{proof}

\section{Further examples}
The zoo of Swendsen-Wang type algorithms via Edwards-Sokal type couplings is enormous as was already noted by Edwards and Sokal \cite{edwards1988generalization} and later extended by Chayes and Machta \cite{chayes1998graphical}.  Recently, such generalizations were used to study the Blume-Capel \cite{gunaratnam2024existence} and $\phi^4$ models \cite{gunaratnam2025supercritical}.  Let us mention some examples here that do not quite as easily lend themselves to the Edwards-Sokal level of generality.

\subsection{Spin Models and Domain Walls}\label{sec:duality}
Kramers-Wannier duality was used to predict the value of the critical point of the two-dimensional Ising model in 1941 \cite{kramers1941statistics}. In 1971, Wegner generalized the duality \cite{WegDual} to higher dimensions, where it turns out that the dual model to an Ising model is a lattice gauge theory. Since then, it was generalized to more general groups, see the review \cite{savit1980duality} and references therein. The duality goes via identifying a high-temperature expansion of the model with the domain walls of the dual model at low temperature. We will argue that the identification of a spin model with its domain walls both fits into \Cref{thm:generalcoupling} and the cubical complexes above. %Since then, the method has been generalized countless times \cite{baxter2016exactly}.

To do that, fix a finite abelian coefficient group $\mathbb{G}$. The (infinite) cubical complex $\mathbb{Z}^d$ has a dual cubical complex $(\mathbb{Z}^d)^*,$ the cells of which are the cells of $\mathbb{Z}^d+(1/2,1/2,...,1/2),$ for instance the $0$-cells are the centres of the $d$-cells of $\mathbb{Z}^d$ . The continuous extension of each $k$-cell  $\mathbf{c}$ of $\mathbb{Z}^d$ then intersects a unique $d-k$-cell $\mathbf{c}^*$ of $(\mathbb{Z}^d)^*$ and to every finite cubical complex $\Lambda\subset \mathbb{Z}^d$, we attach a dual complex $\Lambda^*,$ where $\mathbf{c}^*\in \Lambda^{*}$ if and only if $\mathbf{c}\in \Lambda$. We denote the $k$-chains on $\Lambda^*$ by $C^k.$ Note that $\Lambda^*$ is typically not a complex in the sense of the previous definition since it is typically not closed under taking the boundary map. However, when it makes sense (either partially or fully), we will denote the boundary maps on $\Lambda^*$ by $\partial^k$.

 The duality of cells extends to a duality map\footnote{Note that in the case $\mathbb{G}=\mathbb{Z}/2\mathbb{Z}$, this is not the usual duality of planar percolation models. A dual edge would be declared open \textit{when} the primal edge is.} $*: C_k\to C^{d-k}$. %given by $\sigma^*_{\mathfrak{c}}=\sigma_{\mathfrak{c^*}}^*$ for all $\mathfrak{c}^*\in C^{d-k}.$ 
 This induces co-boundary operators $d^k:C^{k-1}\to C^k$ given by $d^k(\sigma^*)= \partial_{d-k+1}(\sigma)^*$. In particular, $\sigma\in \textrm{Im}(\partial_{d-k+1})\Longleftrightarrow \sigma^*\in \textrm{Im}(d^k)$. 
 Suppose now that $\Lambda^*$ is the dual complex of a finite cubical complex $\Lambda,$ fix $1\leq k\leq  d$, weights $(J_{\mathbf{c}})_{\mathbf{c}\in \Lambda^{*,k}}$ and consider now an even function $h:\mathbb{G}\to \mathbb{R}$ (that is, $h(-g)=h(g)$ for all $g\in \mathbb{G}$). 

 We define a Hamiltonian $H:C^{k-1}\to \mathbb{R}$ by\footnote{Since we defined $q$-flow measure on the primal lattice, we have to define the Hamiltonian of the lattice gauge theory on the dual.}
\begin{align}\label{eq:lattice_gauge_Hamiltonian}
H(\sigma^*)=-\sum_{\mathbf{c}\in \Lambda^{*,k}} J_{\mathbf{c}} \cdot h(d^k(\sigma^*)_{\mathfrak{c}}),
\end{align}
 where we choose an arbitrary orientation of $\mathbf{c}$ at which to evaluate $d^k(\sigma^*)_{\mathfrak{c}}$, which is well-defined by evenness of $h$.
 Since $d^1(\id_{v})=\sum_{w\sim v} \mathbf{1}_{(w,v)}$ for any $0$-cell $v$, if $k=1,$ $\mathbb{G}=\mathbb{Z}/q\mathbb{Z}$ and $h(g)=2\id_{g= 0},$ we recover (a rescaled version of) the ordinary Potts Hamiltonian.

 We get an associated probability measure
 $$
\mu^*_{\Lambda^{k-1}}[\sigma^*]\propto \exp(- H(\sigma^*)).
 $$
 Then, we can attach a dual model $\mu_{\Lambda_{d-k}}$ on $C_{d-k}$ with
\begin{align}%\label{eq:dual_def}
 \mu_{\Lambda_{d-k}}[\eta] &\propto  \mu^{*}_{\Lambda^{k-1}}[(d^k)^{-1}(\eta^*)] %\propto |(d^{k^{-1}})(\eta^* )| \exp\left(\sum_{\mathbf{c}\in \Lambda^{*,k}} J_{\mathbf{c}} \cdot h\left( \left. d^k\left((d^k)^{-1}(\eta^*)\right) \right|_{\mathfrak{c}}\right)\right)\nonumber \\
    \propto %\id_{\eta^*\in \textrm{Im}(d^k)} 
    |(d^k)^{-1}(\eta^*)|
    \exp\left(\sum_{\mathbf{c}\in \Lambda^{*,k}} J_{\mathbf{c}} \cdot h\left(\eta^*_{\mathfrak{c}}\right)\right)\nonumber \\
    &\propto \id_{\eta\in \textrm{Im}(\partial_{d-k+1})}\exp\left(\sum_{\mathbf{c}\in \Lambda_{d-k}} J_{\mathbf{c}} \cdot h\left( \eta_{\mathfrak{c}}\right)\right),
\end{align}
 where, in the third proportionality, we used that $d^{k}$ is a group homomorphism, so that for any $\eta\in C_{d-k}$, $(d^{k})^{-1}(\eta^*)$ is either empty or in bijection with $(d^k)^{-1}(0).$ The latter is the case exactly when $\eta^* \in \textrm{Im}(d^k)$, which definitionally is equivalent to $\eta \in \textrm{Im}(\partial_{d-k+1})$.
When $\mathbb{G} = \mathbb{Z}/q\mathbb{Z}$ and $h(g)=2\id_{g= 0} = \id_{g= 0}-\id_{g\neq 0}+1$ the generalized $q$-flow model is a domain wall measure,
\begin{align}\label{eq:q_flow_is_dual_lattice}
     \mu_{\Lambda_{d-k}}[\cdot] = \ell_{\Lambda_{d-k},x}^q[\; \cdot \mid \Im(d^k)],
 \end{align} 
 where $x_\mathbf{c}:=\exp[-2J_{\mathbf{c}} \id_{\eta_{\mathfrak{c}} \neq 0}]$
 and thus in the Ising case of $q=2$, $\mu$ has the law of the classical domain walls.
 
 It is not a fluke that the domain wall measure $\mu$ is only supported on $\textrm{Im}(\partial_{d-k+1})\subset \ker (\partial_{d-k}).$
 Indeed, if the cubical complex has trivial homology in degree $k$, (as is always the case for e.g. a box in $\mathbb{Z}^d$),
 then the two sets coincide (and $\mu_{\Lambda_{d-k}} = \ell_{\Lambda_{d-k}}^q$),  but otherwise, there are even configurations which are not the domain walls of any spin configurations. For instance, the simple loop in the torus $\mathbb{T}_n^2$ which is the image of $(-n,-n+1,...,n-1,n)$ under the quotient map is an even subgraph of the torus, but it is not the domain walls of an Ising configuration, since the number of times the spin changes as one goes around the torus in the vertical direction must be even (so that you end up with the same spin as you started with).

We summarize the above discussion below in a form compatible with the statements of our other couplings.
\begin{proposition}[Spin Models and Domain Walls]\label{prop:KW}
    Let $\Lambda$ be a finite cubical subcomplex of $\mathbb{Z}^d$, $\mathbb{G}$ a finite abelian group, $1\leq k\leq d$, $(J_{\mathbf{c}})_{\mathbf{c}\in \Lambda^{*,k}}$ be positive coupling constants and $h:\mathbb{G}\to \mathbb{R}$ be even. Let $\Omega:=\mathrm{Im}(\partial_{d-k+1})$  and $\Sigma:=C^{k-1}$. Define $f:\Omega\to 2^{\Sigma}$ by $f(\eta) = \{\sigma \in \Sigma \mid d^k(\sigma) = \eta^* \} = (*\circ d^k)^{-1}(\eta)$ and
     $$
     \rho[\eta]\propto \prod_{\mathbf{c}\in {\Lambda_{d-k}}}\exp(J_{\mathbf{c}} h(\eta_{\mathfrak{c}})), \;\forall \eta\in \Omega, \qquad \gamma[\sigma]\propto 1,\;\forall\sigma\in\Sigma.
     $$
     Then, under the coupling measure $\mathscr{P}$ defined in Theorem \ref{thm:generalcoupling}, 
    \begin{enumerate}
		\item[(a)] The marginal of $\mathscr{P}$ on $\Sigma$ is the spin model $\mu^{*}_{\Lambda^{k-1}}$. 
        For each $\eta \in\Omega$,  the conditional measure $\mathscr{P}[\;\cdot \mid \eta]$ is the uniform measure on $f(\eta)$, that is the uniform assignment of a spin configuration consistent with the domain walls $\eta$.
		\item[(b)] The marginal of $\mathscr{P}$ on $\Omega$ is the domain wall measure $\mu_{\Lambda_{d-k}}$.  For each $\sigma \in\Sigma$, the conditional measure $\mathscr{P}[\;\cdot \mid \sigma] = \delta_{(d^k(\sigma))^*}$, the Dirac measure on the dual of the co-boundary of $\sigma$ corresponding to taking the domain walls of the spin configuration. 
	\end{enumerate}
\end{proposition}

\subsection{The loop \texorpdfstring{O($n$)} \text{ } model and the  \texorpdfstring{$1+1=2$} \text{} coupling}\label{sec:loopO(n)}

The loop O($n$) model on the hexagonal lattice $\mathbb{H}$ is the statistical mechanics model with weights
$$
\ell_{G,x,n}[\eta]\propto \id_{\partial \eta=\emptyset} x^{|\eta|}n^{\kappa^0(\eta)},
$$
with $x,n\geq 0$, $G\subset \mathbb{H}$ finite and $\kappa^0(\eta)$ equal to the number of loops - i.e. components which are not isolated vertices. The loop O($n$) model was introduced as a continuous interpolation of first-order approximations to the spin O($n$) models. In the 80s, Nienhuis provided renormalization group arguments that gave a conjectural phase diagram of the model \cite{NieCrit}. As such, the loop O($n$) model is a unified setting for studying a bevy of planar statistical mechanics models, and it has become the subject of much study in the last decade.

In \cite{Glazman2021Log}, a coupling was introduced\footnote{Originally for $n=2=1+(2-1),$ but it straightforwardly generalizes to all  ways of splitting up $n=m+(n-m)$ with $m\leq n$.} where, by fixing $m\leq n$ and, conditionally on $\eta\sim \ell_{G,x,n}$, independently coloring loops of $\eta$ red with probability $\frac{m}{n}$ and blue otherwise, one gets a representation of $\ell_{G,x,n}$ as the union of independent $\eta^{\mathtt{red}}\sim\ell_{G,x,m}$ and $\eta^{\mathtt{blue}}\sim \ell_{G,x,n-m}$ configurations conditioned not to intersect. Just like for the classical Edwards-Sokal coupling, one can see this immediately by writing $n^{\kappa^0(\eta)}=(m+(n-m))^{\kappa^0(\eta)}$ and applying the binomial theorem. From this, it follows that the conditional distribution of $\eta$ given $\eta^{\mathtt{red}}$ is $\ell_{G\setminus \eta^{\mathtt{red}},x,n-m}$, a fact which was used in \cite{crawford2020macroscopic} to get XOR-invariance of the model away from $n=1$. 

It is noticeable that this, too, is a special case of \Cref{thm:generalcoupling} with 
$$
\Omega=\Sigma=\mathcal{E}_{\emptyset}(G),\quad \rho[\omega]\propto (n-m)^{\kappa^0(\omega)}x^{|\omega|},\quad \gamma[\eta]\propto \left(\frac{m}{n-m}\right)^{\kappa^0(\eta)},\quad  f(\omega)=\{\eta \in \Sigma \mid \eta\subset \omega \}.
$$

Tracing the marginals, we see that $\mathscr{P}_{\Omega}=\ell_{n}$, while
$
\frac{\textrm{d}\mathscr{P}_{\Sigma}}{\textrm{d} \ell_{m}}[\eta]\propto Z_{E\setminus \eta,n-m}
$
where $Z_{E\setminus \eta,n-m}$ denotes the partition function of the loop O($n-m$) model on the complement of $\eta$. For the conditional measures, one gets, by the previous discussion, that $\mathscr{P}[\;\cdot \mid \eta]=\ell_{G\setminus \eta,n-m}\cup\delta_{\eta}$ and that $\mathscr{P}[\cdot\mid \omega]=\mathbb{P}_{\mathcal{C}(\omega),\frac{m}{n}},$ by which we mean the measure obtained by retaining each cluster of $\omega$ independently with probability $\frac{m}{n}$.
\\
\\
%\subsection{Wilson's algorithm} \label{sec:Wilson}
Additionally, one may be able to view many other tricks in statistical mechanics through the lens of \Cref{thm:generalcoupling}, for example Wilson’s Algorithm for sampling the uniform spanning tree \cite{wilson1996generating}, the Temperley-bijection between spanning trees and perfect matchings \cite{kenyon1999trees}, the BKW-coupling relating the six-vertex and random-cluster models \cite{baxter1976equivalence}, the Kac-Ward formula \cite{cimasoni2010generalized,lis2016short},  Le-Jan's isomorphism \cite{jan2008markov}, Poisson thinning, Prüfer codes and super-symmetric representations \cite[Figure 3]{poudevigne2024h}. 

\end{appendix}

\section*{Acknowledgments}
We would like to thank Paul Duncan and Lorca Heeney for useful comments.

UTH's research was funded by the Austrian Science Fund (FWF) 10.55776/P34713.  JJ was supported by National Natural Science Foundation of China (No. 12226001 and No. 12271284). FRK was supported by the Carlsberg Foundation, grant CF24-0466.

For open access purposes, the author has applied a CC BY public copyright license to any author accepted manuscript version arising from this submission.

No data was used for this study and the authors have no relevant conflicts of interest.

\bibliographystyle{abbrv}
\bibliography{bibliography}

@incollection {DC16,
    AUTHOR = {Duminil-Copin, Hugo},
     TITLE = {Random currents expansion of the {I}sing model},
 BOOKTITLE = {European {C}ongress of {M}athematics},
     PAGES = {869--889},
 PUBLISHER = {Eur. Math. Soc., Z\"urich},
      YEAR = {2018},
      ISBN = {978-3-03719-176-7},
   MRCLASS = {82B20 (60K35)},
  MRNUMBER = {3890455},
}

@article {aizenman1988discontinuity,
    AUTHOR = {Aizenman, M. and Chayes, J. T. and Chayes, L. and Newman, C.
              M.},
     TITLE = {Discontinuity of the magnetization in one-dimensional
              {$1/|x-y|^2$} {I}sing and {P}otts models},
   JOURNAL = {J. Statist. Phys.},
  FJOURNAL = {Journal of Statistical Physics},
    VOLUME = {50},
      YEAR = {1988},
    NUMBER = {1-2},
     PAGES = {1--40},
      ISSN = {0022-4715,1572-9613},
   MRCLASS = {82A68 (82A25)},
  MRNUMBER = {939480},
MRREVIEWER = {Mary\ Lunn},
       DOI = {10.1007/BF01022985},
       URL = {https://doi.org/10.1007/BF01022985},
}

@incollection {Ioffe2009stochastic,
    AUTHOR = {Ioffe, Dmitry},
     TITLE = {Stochastic geometry of classical and quantum {I}sing models},
 BOOKTITLE = {Methods of contemporary mathematical statistical physics},
    SERIES = {Lecture Notes in Math.},
    VOLUME = {1970},
     PAGES = {87--127},
 PUBLISHER = {Springer, Berlin},
      YEAR = {2009},
      ISBN = {978-3-540-92795-2},
   MRCLASS = {82B20 (60K35 82-02)},
  MRNUMBER = {2581610},
MRREVIEWER = {Massimo\ Campanino},
       DOI = {10.1007/978-3-540-92796-9},
       URL = {https://doi.org/10.1007/978-3-540-92796-9},
}

@article{aizenman2025geometric,
  title={Geometric Analysis of {I}sing Models, Part {III}},
  author={Aizenman, Michael},
  journal={Math. Phys. Anal. Geom.},
  fjournal={Mathematical Physics, Analysis and Geometry},
  volume={28},
  number={4},
  pages={32},
  year={2025},
  publisher={Springer}
}

@article{kasai1988percolation,
  title={{Percolation problem describing $\pm J$ Ising spin glass system}},
  author={Kasai, Yasuhiro and Okiji, Ayao},
  journal={Prog. Theor. Phys.},
  fjournal={Progress of theoretical physics},
  volume={79},
  number={5},
  pages={1080--1094},
  year={1988},
  publisher={Oxford University Press}
}

@article {lis2017planar,
    AUTHOR = {Lis, Marcin},
     TITLE = {The planar {I}sing model and total positivity},
   JOURNAL = {J. Stat. Phys.},
  FJOURNAL = {Journal of Statistical Physics},
    VOLUME = {166},
      YEAR = {2017},
    NUMBER = {1},
     PAGES = {72--89},
      ISSN = {0022-4715,1572-9613},
   MRCLASS = {82B20 (05C50 05C81 60C05)},
  MRNUMBER = {3592851},
MRREVIEWER = {Longmin\ Wang},
       DOI = {10.1007/s10955-016-1690-x},
       URL = {https://doi.org/10.1007/s10955-016-1690-x},
}

@article{kenyon1999trees,
  title={Trees and matchings.},
  author={Kenyon, Richard W and Propp, James G and Wilson, David B},
  JOURNAL = {Electron. J. Combin.},
  fjournal={The Electronic Journal of Combinatorics},
  volume={7},
  year={2000}
}

@article {lis2016short,
    AUTHOR = {Lis, Marcin},
     TITLE = {A short proof of the {K}ac-{W}ard formula},
   JOURNAL = {Ann. Inst. Henri Poincar\'e{} D},
  FJOURNAL = {Annales de l'Institut Henri Poincar\'e{} D. Combinatorics,
              Physics and their Interactions},
    VOLUME = {3},
      YEAR = {2016},
    NUMBER = {1},
     PAGES = {45--53},
      ISSN = {2308-5827,2308-5835},
   MRCLASS = {05C90 (05A15 82B20)},
  MRNUMBER = {3462629},
MRREVIEWER = {Alessandra\ Bianchi},
       DOI = {10.4171/AIHPD/24},
       URL = {https://doi.org/10.4171/AIHPD/24},
}

@article {kenyon2014conformal,
    AUTHOR = {Kenyon, Richard},
     TITLE = {Conformal invariance of loops in the double-dimer model},
   JOURNAL = {Comm. Math. Phys.},
  FJOURNAL = {Communications in Mathematical Physics},
    VOLUME = {326},
      YEAR = {2014},
    NUMBER = {2},
     PAGES = {477--497},
      ISSN = {0010-3616,1432-0916},
   MRCLASS = {81T40},
  MRNUMBER = {3165463},
MRREVIEWER = {William\ Liu},
       DOI = {10.1007/s00220-013-1881-0},
       URL = {https://doi.org/10.1007/s00220-013-1881-0},
}

@article{deng2007cluster,
  title = {Cluster Simulations of Loop Models on Two-Dimensional Lattices},
  author = {Deng, Youjin and Garoni, Timothy M. and Guo, Wenan and Bl\"ote, Henk W. J. and Sokal, Alan D.},
  journal = {Phys. Rev. Lett.},
  FJOURNAL ={Physical review letters},
  volume = {98},
  issue = {12},
  pages = {120601},
  numpages = {4},
  year = {2007},
  publisher = {American Physical Society},
  doi = {10.1103/PhysRevLett.98.120601},
  url = {https://link.aps.org/doi/10.1103/PhysRevLett.98.120601}
}

@inproceedings {wilson1996generating,
    AUTHOR = {Wilson, David Bruce},
     TITLE = {Generating random spanning trees more quickly than the cover
              time},
 BOOKTITLE = {Proceedings of the {T}wenty-eighth {A}nnual {ACM} {S}ymposium
              on the {T}heory of {C}omputing ({P}hiladelphia, {PA}, 1996)},
     PAGES = {296--303},
 PUBLISHER = {ACM, New York},
      YEAR = {1996},
      ISBN = {0-89791-785-5},
   MRCLASS = {68R10},
  MRNUMBER = {1427525},
       DOI = {10.1145/237814.237880},
       URL = {https://doi.org/10.1145/237814.237880},
}

@article {pete2008corner,
    AUTHOR = {Pete, G\'abor},
     TITLE = {Corner percolation on {$\Bbb Z^2$} and the square root of 17},
   JOURNAL = {Ann. Probab.},
  FJOURNAL = {The Annals of Probability},
    VOLUME = {36},
      YEAR = {2008},
    NUMBER = {5},
     PAGES = {1711--1747},
      ISSN = {0091-1798,2168-894X},
   MRCLASS = {60K35},
  MRNUMBER = {2440921},
MRREVIEWER = {Timo\ Sepp\"al\"ainen},
       DOI = {10.1214/07-AOP373},
       URL = {https://doi.org/10.1214/07-AOP373},
}

@article {poudevigne2024h,
    AUTHOR = {Poudevigne-Auboiron, R\'emy and Wildemann, Peter},
     TITLE = {{$\Bbb H^{2|2}$}-model and vertex-reinforced jump process on
              regular trees: infinite-order transition and an intermediate
              phase},
   JOURNAL = {Comm. Math. Phys.},
  FJOURNAL = {Communications in Mathematical Physics},
    VOLUME = {405},
      YEAR = {2024},
    NUMBER = {8},
     PAGES = {Paper No. 196, 51},
      ISSN = {0010-3616,1432-0916},
   MRCLASS = {82C41 (60K35 82C20 82C27)},
  MRNUMBER = {4779981},
MRREVIEWER = {Stanislav\ Volkov},
       DOI = {10.1007/s00220-024-05070-y},
       URL = {https://doi.org/10.1007/s00220-024-05070-y},
}

@book {jan2008markov,
    AUTHOR = {Le Jan, Yves},
     TITLE = {Markov paths, loops and fields},
    SERIES = {Lecture Notes in Mathematics},
    VOLUME = {2026},
      NOTE = {Lectures from the 38th Probability Summer School held in
              Saint-Flour, 2008,
              \'Ecole d'\'Et\'e{} de Probabilit\'es de Saint-Flour.
              [Saint-Flour Probability Summer School]},
 PUBLISHER = {Springer, Heidelberg},
      YEAR = {2011},
     PAGES = {viii+124},
      ISBN = {978-3-642-21215-4},
   MRCLASS = {60J27 (60J45 60K35)},
  MRNUMBER = {2815763},
MRREVIEWER = {Achim\ Klenke},
       DOI = {10.1007/978-3-642-21216-1},
       URL = {https://doi.org/10.1007/978-3-642-21216-1},
}

@article{cimasoni2010generalized,
  title={A generalized {K}ac--{W}ard formula},
  author={Cimasoni, David},
  journal={J. Stat. Mech.: Theory Exp.},
  FJOURNAL={Journal of Statistical Mechanics: Theory and Experiment},
  volume={2010},
  number={07},
  pages={P07023},
  year={2010},
  publisher={IOP Publishing}
}

@article {jerrum1986random,
    AUTHOR = {Jerrum, Mark R. and Valiant, Leslie G. and Vazirani, Vijay V.},
     TITLE = {Random generation of combinatorial structures from a uniform
              distribution},
   JOURNAL = {Theoret. Comput. Sci.},
  FJOURNAL = {Theoretical Computer Science},
    VOLUME = {43},
      YEAR = {1986},
    NUMBER = {2-3},
     PAGES = {169--188},
      ISSN = {0304-3975,1879-2294},
   MRCLASS = {68Q15},
  MRNUMBER = {855970},
MRREVIEWER = {Claus-Peter\ Schnorr},
       DOI = {10.1016/0304-3975(86)90174-X},
       URL = {https://doi.org/10.1016/0304-3975(86)90174-X},
}

@article {frieze2017random,
    AUTHOR = {Frieze, Alan and Johansson, Tony},
     TITLE = {On random {$k$}-out subgraphs of large graphs},
   JOURNAL = {Random Struct. Algor.},
  FJOURNAL = {Random Structures \& Algorithms},
    VOLUME = {50},
      YEAR = {2017},
    NUMBER = {2},
     PAGES = {143--157},
      ISSN = {1042-9832,1098-2418},
   MRCLASS = {05C80 (05C38 05C45 60C05)},
  MRNUMBER = {3607119},
MRREVIEWER = {Mark\ R.\ Jerrum},
       DOI = {10.1002/rsa.20650},
       URL = {https://doi.org/10.1002/rsa.20650},
}

@book {knuth1997art,
    AUTHOR = {Knuth, Donald E.},
     TITLE = {The art of computer programming. {V}ol. 3},
   EDITION = {Second},
 PUBLISHER = {Addison-Wesley, Reading, MA},
      YEAR = {1998},
     PAGES = {xiv+780},
      ISBN = {0-201-89685-0},
   MRCLASS = {68-02},
  MRNUMBER = {3077154},
}

@article {GG06,
    AUTHOR = {Graham, Brightwell and Grimmett, Geoffrey},
     TITLE = {Influence and sharp-threshold theorems for monotonic measures},
   JOURNAL = {Ann. Probab.},
  FJOURNAL = {The Annals of Probability},
    VOLUME = {34},
      YEAR = {2006},
    NUMBER = {5},
     PAGES = {1726--1745},
      ISSN = {0091-1798,2168-894X},
   MRCLASS = {60E15 (60K35 82B43)},
  MRNUMBER = {2271479},
MRREVIEWER = {Elchanan\ Mossel},
       DOI = {10.1214/009117906000000278},
       URL = {https://doi.org/10.1214/009117906000000278},
}

@article{kavitha2009cycle,
  title={Cycle bases in graphs characterization, algorithms, complexity, and applications},
  author={Kavitha, Telikepalli and Liebchen, Christian and Mehlhorn, Kurt and Michail, Dimitrios and Rizzi, Romeo and Ueckerdt, Torsten and Zweig, Katharina A},
  journal={Computer Science Review},
  volume={3},
  number={4},
  pages={199--243},
  year={2009},
  publisher={Elsevier}
}

@article{savit1980duality,
  title={Duality in field theory and statistical systems},
  author={Savit, Robert},
  journal={Reviews of Modern Physics},
  volume={52},
  number={2},
  pages={453},
  year={1980},
  publisher={APS}
}

@article{higdon1998auxiliary,
  title={Auxiliary variable methods for {M}arkov chain {M}onte {C}arlo with applications},
  author={Higdon, David M},
  journal={J. Am. Stat. Assoc.},
  fjournal={Journal of the American statistical Association},
  volume={93},
  number={442},
  pages={585--595},
  year={1998},
  publisher={Taylor \& Francis}
}

@book {levin2017markov,
    AUTHOR = {Levin, David A. and Peres, Yuval},
     TITLE = {Markov chains and mixing times},
   EDITION = {Second},
 PUBLISHER = {American Mathematical Society, Providence, RI},
      YEAR = {2017},
     PAGES = {xvi+447},
      ISBN = {978-1-4704-2962-1},
   MRCLASS = {60J10 (60-01 60B15 60C05 60J27 60K35 68U20 82C22)},
  MRNUMBER = {3726904},
       DOI = {10.1090/mbk/107},
       URL = {https://doi.org/10.1090/mbk/107},
}

@inproceedings {guo2017random,
    AUTHOR = {Guo, Heng and Jerrum, Mark},
     TITLE = {Random cluster dynamics for the {I}sing model is rapidly
              mixing},
 BOOKTITLE = {Proceedings of the {T}wenty-{E}ighth {A}nnual {ACM}-{SIAM}
              {S}ymposium on {D}iscrete {A}lgorithms},
     PAGES = {1818--1827},
 PUBLISHER = {SIAM, Philadelphia, PA},
      YEAR = {2017},
      ISBN = {978-1-61197-478-2},
   MRCLASS = {82B20 (68W40)},
  MRNUMBER = {3627847},
       DOI = {10.1137/1.9781611974782.118},
       URL = {https://doi.org/10.1137/1.9781611974782.118},
}

@article {bjornberg2015vanishing,
    AUTHOR = {Bj\"ornberg, Jakob E.},
     TITLE = {Vanishing critical magnetization in the quantum {I}sing model},
   JOURNAL = {Comm. Math. Phys.},
  FJOURNAL = {Communications in Mathematical Physics},
    VOLUME = {337},
      YEAR = {2015},
    NUMBER = {2},
     PAGES = {879--907},
      ISSN = {0010-3616,1432-0916},
   MRCLASS = {81V10},
  MRNUMBER = {3339165},
       DOI = {10.1007/s00220-015-2299-7},
       URL = {https://doi.org/10.1007/s00220-015-2299-7},
}

@article {bjornberg2009phase,
    AUTHOR = {Bj\"ornberg, Jakob E. and Grimmett, Geoffrey},
     TITLE = {The phase transition of the quantum {I}sing model is sharp},
   JOURNAL = {J. Stat. Phys.},
  FJOURNAL = {Journal of Statistical Physics},
    VOLUME = {136},
      YEAR = {2009},
    NUMBER = {2},
     PAGES = {231--273},
      ISSN = {0022-4715,1572-9613},
   MRCLASS = {82B26 (82B10 82B20 82B27)},
  MRNUMBER = {2525245},
MRREVIEWER = {Peter\ Bernard\ Weichman},
       DOI = {10.1007/s10955-009-9788-z},
       URL = {https://doi.org/10.1007/s10955-009-9788-z},
}

@article{forsstrom2025current,
  title={A current expansion for {I}sing lattice gauge theory},
  author={Forsstr{\"o}m, Malin P and Viklund, Fredrik},
  journal={arXiv preprint arXiv:2502.19942},
  year={2025}
}

@article {grimmett2010random,
    AUTHOR = {Grimmett, Geoffrey and Janson, Svante},
     TITLE = {Random graphs with forbidden vertex degrees},
   JOURNAL = {Random Structures Algorithms},
  FJOURNAL = {Random Structures \& Algorithms},
    VOLUME = {37},
      YEAR = {2010},
    NUMBER = {2},
     PAGES = {137--175},
      ISSN = {1042-9832,1098-2418},
   MRCLASS = {60C05 (05C70)},
  MRNUMBER = {2676027},
MRREVIEWER = {Hai\ Yan\ Chen},
       DOI = {10.1002/rsa.20307},
       URL = {https://doi.org/10.1002/rsa.20307},
}

@incollection {duminil2022100,
    AUTHOR = {Duminil-Copin, Hugo},
     TITLE = {100 years of the (critical) {I}sing model on the hypercubic
              lattice},
 BOOKTITLE = {I{CM}---{I}nternational {C}ongress of {M}athematicians. {V}ol.
              1. {P}rize lectures},
     PAGES = {164--210},
 PUBLISHER = {EMS Press, Berlin},
      YEAR = {2023},
      ISBN = {978-3-98547-059-4; 978-3-98547-559-9; 978-3-98547-058-7},
   MRCLASS = {82B20 (01A60 01A61 60K35 82-03)},
  MRNUMBER = {4680248},
MRREVIEWER = {Jacob\ Richey},
}

@incollection {sokal1997monte,
    AUTHOR = {Sokal, Alan},
     TITLE = {Monte {C}arlo methods in statistical mechanics: foundations
              and new algorithms},
 BOOKTITLE = {Functional integration ({C}arg\`ese, 1996)},
    SERIES = {NATO Adv. Sci. Inst. Ser. B: Phys.},
    VOLUME = {361},
     PAGES = {131--192},
 PUBLISHER = {Plenum, New York},
      YEAR = {1997},
      ISBN = {0-306-45617-6},
   MRCLASS = {82B80 (82B05 82B26 82B27)},
  MRNUMBER = {1477456},
MRREVIEWER = {Emilio\ N. M. Cirillo},
}

@inproceedings{evertz2002new,
  title={New cluster method for the {I}sing model},
  author={Evertz, Hans Gerd and Erkinger, HM and Von der Linden, Wolfgang},
  booktitle={Computer Simulation Studies in Condensed-Matter Physics XIV: Proceedings of the Fourteenth Workshop, Athens, GA, USA, February 19--24, 2001},
  pages={123--133},
  year={2002},
  organization={Springer}
}

@article {ben1990critical,
    AUTHOR = {Ben-Av, Radel and Kandel, Daniel and Katznelson, Ehud and Lauwers, Paul
              G. and Solomon, Sorin},
     TITLE = {Critical acceleration of lattice gauge simulations},
   JOURNAL = {J. Stat. Phys.},
  FJOURNAL = {Journal of Statistical Physics},
    VOLUME = {58},
      YEAR = {1990},
    NUMBER = {1-2},
     PAGES = {125--139},
      ISSN = {0022-4715,1572-9613},
   MRCLASS = {81E25 (81-08 82-08)},
  MRNUMBER = {1035619},
       DOI = {10.1007/BF01020288},
       URL = {https://doi.org/10.1007/BF01020288},
}

@book {MadrasBook,
    AUTHOR = {Madras, Neal and Slade, Gordon},
     TITLE = {The self-avoiding walk},
    SERIES = {Modern Birkh\"auser Classics},
      NOTE = {Reprint of the 1993 original},
 PUBLISHER = {Birkh\"auser/Springer, New York},
      YEAR = {2013},
     PAGES = {xvi+425},
      ISBN = {978-1-4614-6024-4; 978-1-4614-6025-1},
   MRCLASS = {82Bxx (60K35)},
  MRNUMBER = {2986656},
       DOI = {10.1007/978-1-4614-6025-1},
       URL = {https://doi.org/10.1007/978-1-4614-6025-1},
}

@article {zhang2020loop,
    AUTHOR = {Zhang, Lei and Michel, Manon and El\c{c}i, Eren M. and Deng,
              Youjin},
     TITLE = {Loop-cluster coupling and algorithm for classical statistical
              models},
   JOURNAL = {Phys. Rev. Lett.},
  FJOURNAL = {Physical Review Letters},
    VOLUME = {125},
      YEAR = {2020},
    NUMBER = {20},
     PAGES = {200603, 6},
      ISSN = {0031-9007,1079-7114},
   MRCLASS = {82B20 (82B27)},
  MRNUMBER = {4180383},
       DOI = {10.1103/physrevlett.125.200603},
       URL = {https://doi.org/10.1103/physrevlett.125.200603},
}

@article {bauerschmidt2017local,
    AUTHOR = {Bauerschmidt, Roland and Knowles, Antti and Yau, Horng-Tzer},
     TITLE = {Local semicircle law for random regular graphs},
   JOURNAL = {Comm. Pure Appl. Math.},
  FJOURNAL = {Communications on Pure and Applied Mathematics},
    VOLUME = {70},
      YEAR = {2017},
    NUMBER = {10},
     PAGES = {1898--1960},
      ISSN = {0010-3640,1097-0312},
   MRCLASS = {05C80 (05C50 60B20)},
  MRNUMBER = {3688032},
MRREVIEWER = {Lyuben\ R.\ Mutafchiev},
       DOI = {10.1002/cpa.21709},
       URL = {https://doi.org/10.1002/cpa.21709},
}

@article {Bauerschmidt2024PercoForest,
    AUTHOR = {Bauerschmidt, Roland and Crawford, Nicholas and Helmuth,
              Tyler},
     TITLE = {Percolation transition for random forests in {$d\geqslant 3$}},
   JOURNAL = {Invent. Math.},
  FJOURNAL = {Inventiones Mathematicae},
    VOLUME = {237},
      YEAR = {2024},
    NUMBER = {2},
     PAGES = {445--540},
      ISSN = {0020-9910,1432-1297},
   MRCLASS = {60K35 (05C85 82B20 82B43)},
  MRNUMBER = {4768630},
MRREVIEWER = {Rinaldo\ Schinazi},
       DOI = {10.1007/s00222-024-01263-3},
       URL = {https://doi.org/10.1007/s00222-024-01263-3},
}

@article {Bauerschmidt2021PlaneForest,
    AUTHOR = {Bauerschmidt, Roland and Crawford, Nicholas and Helmuth, Tyler
              and Swan, Andrew},
     TITLE = {Random spanning forests and hyperbolic symmetry},
   JOURNAL = {Comm. Math. Phys.},
  FJOURNAL = {Communications in Mathematical Physics},
    VOLUME = {381},
      YEAR = {2021},
    NUMBER = {3},
     PAGES = {1223--1261},
      ISSN = {0010-3616,1432-0916},
   MRCLASS = {60K35 (05C80)},
  MRNUMBER = {4218682},
       DOI = {10.1007/s00220-020-03921-y},
       URL = {https://doi.org/10.1007/s00220-020-03921-y},
}

@article {CaraccioloForests,
    AUTHOR = {Caracciolo, Sergio and Jacobsen, Jesper Lykke and Saleur,
              Hubert and Sokal, Alan D. and Sportiello, Andrea},
     TITLE = {Fermionic field theory for trees and forests},
   JOURNAL = {Phys. Rev. Lett.},
  FJOURNAL = {Physical Review Letters},
    VOLUME = {93},
      YEAR = {2004},
    NUMBER = {8},
     PAGES = {080601, 4},
      ISSN = {0031-9007,1079-7114},
   MRCLASS = {81T10 (05C05)},
  MRNUMBER = {2110547},
       DOI = {10.1103/PhysRevLett.93.080601},
       URL = {https://doi.org/10.1103/PhysRevLett.93.080601},
}

@article {BaloghKrFree,
    AUTHOR = {Balogh, J\'ozsef and Morris, Robert and Samotij, Wojciech and
              Warnke, Lutz},
     TITLE = {The typical structure of sparse {$K_{r+1}$}-free graphs},
   JOURNAL = {Trans. Amer. Math. Soc.},
  FJOURNAL = {Transactions of the American Mathematical Society},
    VOLUME = {368},
      YEAR = {2016},
    NUMBER = {9},
     PAGES = {6439--6485},
      ISSN = {0002-9947,1088-6850},
   MRCLASS = {05C80 (05A16 05C30 05C35)},
  MRNUMBER = {3461039},
MRREVIEWER = {David\ Burns},
       DOI = {10.1090/tran/6552},
       URL = {https://doi.org/10.1090/tran/6552},
}

@article {WegDual,
    AUTHOR = {Wegner, F. J.},
     TITLE = {Duality in generalized {I}sing models and phase transitions
              without local order parameters},
   JOURNAL = {J. Math. Phys.},
  FJOURNAL = {Journal of Mathematical Physics},
    VOLUME = {12},
      YEAR = {1971},
     PAGES = {2259--2272},
      ISSN = {0022-2488,1089-7658},
   MRCLASS = {82.46},
  MRNUMBER = {289087},
MRREVIEWER = {L.\ Finkelstein},
       DOI = {10.1063/1.1665530},
       URL = {https://doi.org/10.1063/1.1665530},
}

@article {KogPotts,
    AUTHOR = {Kogut, J. B. and Pearson, R. B. and Shigemitsu, J. and
              Sinclair, D. K.},
     TITLE = {{$Z\sb{N}$}\ and {$N$}-state {P}otts lattice gauge theories:
              phase diagrams, first-order transitions, {$\beta $}\
              functions, and {$1/N$}\ expansions},
   JOURNAL = {Phys. Rev. D (3)},
  FJOURNAL = {Physical Review. D. Particles and Fields. Third Series},
    VOLUME = {22},
      YEAR = {1980},
    NUMBER = {10},
     PAGES = {2447--2464},
      ISSN = {0556-2821},
   MRCLASS = {82A25 (82A67 82A68)},
  MRNUMBER = {594232},
MRREVIEWER = {Gheorghe\ Nenciu},
       DOI = {10.1103/PhysRevD.22.2447},
       URL = {https://doi.org/10.1103/PhysRevD.22.2447},
}

@article {PromelNoTriangles,
    AUTHOR = {Pr\"omel, Hans J\"urgen and Steger, Angelika},
     TITLE = {On the asymptotic structure of sparse triangle free graphs},
   JOURNAL = {J. Graph Theory},
  FJOURNAL = {Journal of Graph Theory},
    VOLUME = {21},
      YEAR = {1996},
    NUMBER = {2},
     PAGES = {137--151},
      ISSN = {0364-9024,1097-0118},
   MRCLASS = {05C80 (05C75)},
  MRNUMBER = {1368739},
MRREVIEWER = {Lyuben\ R.\ Mutafchiev},
       DOI = {10.1002/(sici)1097-0118(199602)21:2<137::aid-jgt3>3.3.co;2-k},
       URL =
              {https://doi.org/10.1002/(sici)1097-0118(199602)21:2<137::aid-jgt3>3.3.co;2-k},
}

@article{OsthusNoTriangles,
    AUTHOR = {Osthus, Deryk and Pr\"omel, Hans J\"urgen and Taraz, Anusch},
     TITLE = {For which densities are random triangle-free graphs almost
              surely bipartite?},
   JOURNAL = {Combinatorica},
  FJOURNAL = {Combinatorica. An International Journal on Combinatorics and
              the Theory of Computing},
    VOLUME = {23},
      YEAR = {2003},
    NUMBER = {1},
     PAGES = {105--150},
      ISSN = {0209-9683,1439-6912},
   MRCLASS = {05C80 (05A16 05C35)},
  MRNUMBER = {1996629},
MRREVIEWER = {W.\ G.\ Brown},
       DOI = {10.1007/s00493-003-0016-1},
       URL = {https://doi.org/10.1007/s00493-003-0016-1},
}

@article {GerkeNoK4,
    AUTHOR = {Gerke, S. and Pr\"omel, H. J. and Schickinger, T. and Steger,
              A. and Taraz, A.},
     TITLE = {{$K_4$}-free subgraphs of random graphs revisited},
   JOURNAL = {Combinatorica},
  FJOURNAL = {Combinatorica. An International Journal on Combinatorics and
              the Theory of Computing},
    VOLUME = {27},
      YEAR = {2007},
    NUMBER = {3},
     PAGES = {329--365},
      ISSN = {0209-9683,1439-6912},
   MRCLASS = {05C80 (05C35)},
  MRNUMBER = {2345813},
MRREVIEWER = {Yoshiharu\ Kohayakawa},
       DOI = {10.1007/s00493-007-2010-5},
       URL = {https://doi.org/10.1007/s00493-007-2010-5},
}

@incollection {DC17,
    AUTHOR = {Duminil-Copin, Hugo},
     TITLE = {Lectures on the {I}sing and {P}otts models on the hypercubic
              lattice},
 BOOKTITLE = {Random graphs, phase transitions, and the {G}aussian free
              field},
    SERIES = {Springer Proc. Math. Stat.},
    VOLUME = {304},
     PAGES = {35--161},
 PUBLISHER = {Springer, Cham},
      YEAR = {[2020] \copyright 2020},
      ISBN = {978-3-030-32011-9; 978-3-030-32010-2},
   MRCLASS = {82B20 (60K35)},
  MRNUMBER = {4043224},
MRREVIEWER = {Jianping\ Jiang},
       DOI = {10.1007/978-3-030-32011-9\_2},
       URL = {https://doi.org/10.1007/978-3-030-32011-9_2},
}

@article {DCSAW,
    AUTHOR = {Duminil-Copin, Hugo and Smirnov, Stanislav},
     TITLE = {The connective constant of the honeycomb lattice equals
              {$\sqrt{2+\sqrt{2}}$}},
   JOURNAL = {Ann. of Math. (2)},
  FJOURNAL = {Annals of Mathematics. Second Series},
    VOLUME = {175},
      YEAR = {2012},
    NUMBER = {3},
     PAGES = {1653--1665},
      ISSN = {0003-486X,1939-8980},
   MRCLASS = {82B41 (60J67 60K35 82D60)},
  MRNUMBER = {2912714},
MRREVIEWER = {Dmitry\ Beliaev},
       DOI = {10.4007/annals.2012.175.3.14},
       URL = {https://doi.org/10.4007/annals.2012.175.3.14},
}

@article {PemantleUST,
    AUTHOR = {Pemantle, Robin},
     TITLE = {Choosing a spanning tree for the integer lattice uniformly},
   JOURNAL = {Ann. Probab.},
  FJOURNAL = {The Annals of Probability},
    VOLUME = {19},
      YEAR = {1991},
    NUMBER = {4},
     PAGES = {1559--1574},
      ISSN = {0091-1798,2168-894X},
   MRCLASS = {60C05 (05C05 60K35)},
  MRNUMBER = {1127715},
MRREVIEWER = {David\ J.\ Aldous},
       URL =
              {http://links.jstor.org/sici?sici=0091-1798(199110)19:4<1559:CASTFT>2.0.CO;2-E&origin=MSN},
}

@article {KenyonUST,
    AUTHOR = {Kenyon, Richard},
     TITLE = {The asymptotic determinant of the discrete {L}aplacian},
   JOURNAL = {Acta Math.},
  FJOURNAL = {Acta Mathematica},
    VOLUME = {185},
      YEAR = {2000},
    NUMBER = {2},
     PAGES = {239--286},
      ISSN = {0001-5962,1871-2509},
   MRCLASS = {82B41 (05B45 05B50 31C20 60J10 82B20)},
  MRNUMBER = {1819995},
MRREVIEWER = {Jan\ de Gier},
       DOI = {10.1007/BF02392811},
       URL = {https://doi.org/10.1007/BF02392811},
}

@article {angel2021uniform,
    AUTHOR = {Angel, Omer and Ray, Gourab and Spinka, Yinon},
     TITLE = {Uniform even subgraphs and graphical representations of
              {I}sing as factors of i.i.d.},
   JOURNAL = {Electron. J. Probab.},
  FJOURNAL = {Electronic Journal of Probability},
    VOLUME = {29},
      YEAR = {2024},
     PAGES = {Paper No. 39, 31},
      ISSN = {1083-6489},
   MRCLASS = {60K35 (05C63 37A60 82B20)},
  MRNUMBER = {4718436},
MRREVIEWER = {Glauco\ Valle},
       DOI = {10.1214/24-ejp1082},
       URL = {https://doi.org/10.1214/24-ejp1082},
}

@article {GMM18,
    AUTHOR = {Garet, Olivier and Marchand, R\'egine and Marcovici, Ir\`ene},
     TITLE = {Does {E}ulerian percolation on {$\Bbb Z^2$} percolate?},
   JOURNAL = {ALEA Lat. Am. J. Probab. Math. Stat.},
  FJOURNAL = {ALEA. Latin American Journal of Probability and Mathematical
              Statistics},
    VOLUME = {15},
      YEAR = {2018},
    NUMBER = {1},
     PAGES = {279--294},
      ISSN = {1980-0436},
   MRCLASS = {60K35 (82B43)},
  MRNUMBER = {3795450},
MRREVIEWER = {Rapha\"el\ Cerf},
       DOI = {10.30757/alea.v15-13},
       URL = {https://doi.org/10.30757/alea.v15-13},
}

@article {grimmett2007random,
    AUTHOR = {Grimmett, Geoffrey and Janson, Svante},
     TITLE = {Random even graphs},
   JOURNAL = {Electron. J. Combin.},
  FJOURNAL = {Electronic Journal of Combinatorics},
    VOLUME = {16},
      YEAR = {2009},
    NUMBER = {1},
     PAGES = {Research Paper, 46, 19},
      ISSN = {1077-8926},
   MRCLASS = {05C80 (60K35)},
  MRNUMBER = {2491648},
MRREVIEWER = {Dirk\ Oliver\ Theis},
       DOI = {10.37236/135},
       URL = {https://doi.org/10.37236/135},
}

@article {aizenman1982geometric,
    AUTHOR = {Aizenman, Michael},
     TITLE = {Geometric analysis of $\varphi^{4}$ fields and {I}sing
              models. {I}, {II}},
   JOURNAL = {Comm. Math. Phys.},
  FJOURNAL = {Communications in Mathematical Physics},
    VOLUME = {86},
      YEAR = {1982},
    NUMBER = {1},
     PAGES = {1--48},
      ISSN = {0010-3616,1432-0916},
   MRCLASS = {81E25 (82A68)},
  MRNUMBER = {678000},
MRREVIEWER = {C.\ A.\ Hurst},
       URL = {http://projecteuclid.org/euclid.cmp/1103921614},
}

@incollection {HS16,
    AUTHOR = {Hiraoka, Yasuaki and Shirai, Tomoyuki},
     TITLE = {Tutte polynomials and random-cluster models in {B}ernoulli
              cell complexes},
 BOOKTITLE = {Stochastic analysis on large scale interacting systems},
    SERIES = {RIMS K\^oky\^uroku Bessatsu},
    VOLUME = {B59},
     PAGES = {289--304},
 PUBLISHER = {Res. Inst. Math. Sci. (RIMS), Kyoto},
      YEAR = {2016},
   MRCLASS = {60C05 (05C31 05C45 05C80 57Q05)},
  MRNUMBER = {3675939},
}

@article {NieCrit,
    AUTHOR = {Nienhuis, Bernard},
     TITLE = {Exact critical point and critical exponents of {${\rm O}(n)$}\
              models in two dimensions},
   JOURNAL = {Phys. Rev. Lett.},
  FJOURNAL = {Physical Review Letters},
    VOLUME = {49},
      YEAR = {1982},
    NUMBER = {15},
     PAGES = {1062--1065},
      ISSN = {0031-9007},
   MRCLASS = {81E25 (82A68)},
  MRNUMBER = {675241},
       DOI = {10.1103/PhysRevLett.49.1062},
       URL = {https://doi.org/10.1103/PhysRevLett.49.1062},
}

@article {Glazman2021Log,
    AUTHOR = {Glazman, Alexander and Manolescu, Ioan},
     TITLE = {Uniform {L}ipschitz functions on the triangular lattice have
              logarithmic variations},
   JOURNAL = {Comm. Math. Phys.},
  FJOURNAL = {Communications in Mathematical Physics},
    VOLUME = {381},
      YEAR = {2021},
    NUMBER = {3},
     PAGES = {1153--1221},
      ISSN = {0010-3616,1432-0916},
   MRCLASS = {60K35 (05C60 26A16 30L99 82B20 82B41)},
  MRNUMBER = {4218681},
       DOI = {10.1007/s00220-020-03920-z},
       URL = {https://doi.org/10.1007/s00220-020-03920-z},
}

@incollection {Newman1990Rep,
    AUTHOR = {Newman, Charles M.},
     TITLE = {Ising models and dependent percolation},
 BOOKTITLE = {Topics in statistical dependence ({S}omerset, {PA}, 1987)},
    SERIES = {IMS Lecture Notes Monogr. Ser.},
    VOLUME = {16},
     PAGES = {395--401},
 PUBLISHER = {Inst. Math. Statist., Hayward, CA},
      YEAR = {1990},
      ISBN = {0-940600-23-4},
   MRCLASS = {60K35 (05C80 82B20 82B43)},
  MRNUMBER = {1193993},
       DOI = {10.1214/lnms/1215457575},
       URL = {https://doi.org/10.1214/lnms/1215457575},
}

@incollection {Newman1994Rep,
    AUTHOR = {Newman, Charles M.},
     TITLE = {Disordered {I}sing systems and random cluster representations},
 BOOKTITLE = {Probability and phase transition ({C}ambridge, 1993)},
    SERIES = {NATO Adv. Sci. Inst. Ser. C: Math. Phys. Sci.},
    VOLUME = {420},
     PAGES = {247--260},
 PUBLISHER = {Kluwer Acad. Publ., Dordrecht},
      YEAR = {1994},
      ISBN = {0-7923-2720-9},
   MRCLASS = {82B44 (60K35 82B20 82D30)},
  MRNUMBER = {1283186},
MRREVIEWER = {Massimo\ Campanino},
       DOI = {10.1007/978-94-015-8326-8\_15},
       URL = {https://doi.org/10.1007/978-94-015-8326-8_15},
}

@article {duncanPRCM1,
    AUTHOR = {Duncan, Paul and Schweinhart, Benjamin},
     TITLE = {Topological {P}hases in the {P}laquette {R}andom-{C}luster
              {M}odel and {P}otts {L}attice {G}auge {T}heory},
   JOURNAL = {Comm. Math. Phys.},
  FJOURNAL = {Communications in Mathematical Physics},
    VOLUME = {406},
      YEAR = {2025},
    NUMBER = {6},
     PAGES = {Paper No. 145},
      ISSN = {0010-3616,1432-0916},
   MRCLASS = {99-06},
  MRNUMBER = {4913966},
       DOI = {10.1007/s00220-025-05322-5},
       URL = {https://doi.org/10.1007/s00220-025-05322-5},
}

@article{duncanPRCM2,
  title={A Sharp Deconfinement Transition for {P}otts Lattice Gauge Theory in Codimension Two},
  author={Duncan, Paul and Schweinhart, Benjamin},
  journal={arXiv preprint arXiv:2308.07534},
  year={2023}
}

@article {hansen2023uniform,
    AUTHOR = {Hansen, Ulrik Thinggaard and Kj{\ae}r, Boris and Klausen,
              Frederik Ravn},
     TITLE = {The {U}niform {E}ven {S}ubgraph and {I}ts {C}onnection to
              {P}hase {T}ransitions of {G}raphical {R}epresentations of the
              {I}sing {M}odel},
   JOURNAL = {Comm. Math. Phys.},
  FJOURNAL = {Communications in Mathematical Physics},
    VOLUME = {406},
      YEAR = {2025},
    NUMBER = {6},
     PAGES = {Paper No. 124},
      ISSN = {0010-3616,1432-0916},
   MRCLASS = {82B20 (60K35 82B05 82B26 82B43)},
  MRNUMBER = {4902845},
       DOI = {10.1007/s00220-025-05297-3},
       URL = {https://doi.org/10.1007/s00220-025-05297-3},
}

@article{baxter1976equivalence,
  title={Equivalence of the {P}otts model or {W}hitney polynomial with an ice-type model},
  author={Baxter, Rodney J and Kelland, Stewart B and Wu, Frank Y},
  journal={J. Phys. A Math. Gen.},
  fjournal={Journal of Physics A: Mathematical and General},
  volume={9},
  number={3},
  pages={397},
  year={1976},
  publisher={IOP Publishing}
}

@book {Gri06,
    AUTHOR = {Grimmett, Geoffrey},
     TITLE = {The Random-Cluster Model},
    SERIES = {Grundlehren der mathematischen Wissenschaften [Fundamental
              Principles of Mathematical Sciences]},
    VOLUME = {333},
 PUBLISHER = {Springer-Verlag, Berlin},
      YEAR = {2006},
     PAGES = {xiv+377},
      ISBN = {978-3-540-32890-2; 3-540-32890-4},
   MRCLASS = {60K35 (60-02 82-02 82B20 82B43)},
  MRNUMBER = {2243761},
MRREVIEWER = {Olivier\ Garet},
       DOI = {10.1007/978-3-540-32891-9},
       URL = {https://doi.org/10.1007/978-3-540-32891-9},
}

@article {griffiths1970concavity,
    AUTHOR = {Griffiths, Robert B. and Hurst, C. A. and Sherman, S.},
     TITLE = {Concavity of magnetization of an {I}sing ferromagnet in a
              positive external field},
   JOURNAL = {J. Math. Phys.},
  FJOURNAL = {Journal of Mathematical Physics},
    VOLUME = {11},
      YEAR = {1970},
     PAGES = {790--795},
      ISSN = {0022-2488,1089-7658},
   MRCLASS = {81.05},
  MRNUMBER = {266507},
       DOI = {10.1063/1.1665211},
       URL = {https://doi.org/10.1063/1.1665211},
}

@article {duminil2020exponential,
    AUTHOR = {Duminil-Copin, Hugo and Goswami, Subhajit and Raoufi, Aran},
     TITLE = {Exponential decay of truncated correlations for the {I}sing
              model in any dimension for all but the critical temperature},
   JOURNAL = {Comm. Math. Phys.},
  FJOURNAL = {Communications in Mathematical Physics},
    VOLUME = {374},
      YEAR = {2020},
    NUMBER = {2},
     PAGES = {891--921},
      ISSN = {0010-3616,1432-0916},
   MRCLASS = {82B20 (60K35)},
  MRNUMBER = {4072233},
MRREVIEWER = {Jianping\ Jiang},
       DOI = {10.1007/s00220-019-03633-y},
       URL = {https://doi.org/10.1007/s00220-019-03633-y},
}

@article {kramers1941statistics,
    AUTHOR = {Kramers, H. A. and Wannier, G. H.},
     TITLE = {Statistics of the two-dimensional ferromagnet. {I}},
   JOURNAL = {Phys. Rev. (2)},
  FJOURNAL = {Physical Review. Series II},
    VOLUME = {60},
      YEAR = {1941},
     PAGES = {252--262},
      ISSN = {0031-899X,1536-6065},
   MRCLASS = {79.0X},
  MRNUMBER = {4803},
MRREVIEWER = {L.\ W.\ Nordheim},
}

@article {Lis,
    AUTHOR = {Lis, Marcin},
     TITLE = {Spins, percolation and height functions},
   JOURNAL = {Electron. J. Probab.},
  FJOURNAL = {Electronic Journal of Probability},
    VOLUME = {27},
      YEAR = {2022},
     PAGES = {Paper No. 35, 21},
      ISSN = {1083-6489},
   MRCLASS = {82B20 (60K35 82B43)},
  MRNUMBER = {4387843},
MRREVIEWER = {Malin\ P.\ Forsstr\"om},
       DOI = {10.1214/22-ejp761},
       URL = {https://doi.org/10.1214/22-ejp761},
}

@article {crawford2020macroscopic,
    AUTHOR = {Crawford, Nicholas and Glazman, Alexander and Harel, Matan and
              Peled, Ron},
     TITLE = {Macroscopic loops in the loop {$O(n)$} model via the {XOR}
              trick},
   JOURNAL = {Ann. Probab.},
  FJOURNAL = {The Annals of Probability},
    VOLUME = {53},
      YEAR = {2025},
    NUMBER = {2},
     PAGES = {478--508},
      ISSN = {0091-1798,2168-894X},
   MRCLASS = {82B43 (82B26)},
  MRNUMBER = {4888137},
       DOI = {10.1214/24-aop1712},
       URL = {https://doi.org/10.1214/24-aop1712},
}

@article {edwards1988generalization,
    AUTHOR = {Edwards, Robert G. and Sokal, Alan D.},
     TITLE = {Generalization of the {F}ortuin-{K}asteleyn-{S}wendsen-{W}ang
              representation and {M}onte {C}arlo algorithm},
   JOURNAL = {Phys. Rev. D (3)},
  FJOURNAL = {Physical Review. D. Particles and Fields. Third Series},
    VOLUME = {38},
      YEAR = {1988},
    NUMBER = {6},
     PAGES = {2009--2012},
      ISSN = {0556-2821},
   MRCLASS = {82-04 (81-08 81E25 82-08 82A68)},
  MRNUMBER = {965465},
       DOI = {10.1103/PhysRevD.38.2009},
       URL = {https://doi.org/10.1103/PhysRevD.38.2009},
}

@article {klausen2021monotonicity,
    AUTHOR = {Klausen, Frederik Ravn},
     TITLE = {On monotonicity and couplings of random currents and the
              loop-{O}(1)-model},
   JOURNAL = {ALEA Lat. Am. J. Probab. Math. Stat.},
  FJOURNAL = {ALEA. Latin American Journal of Probability and Mathematical
              Statistics},
    VOLUME = {19},
      YEAR = {2022},
    NUMBER = {1},
     PAGES = {151--161},
      ISSN = {1980-0436},
   MRCLASS = {82B20 (60K35 82B43)},
  MRNUMBER = {4359789},
       DOI = {10.30757/alea.v19-07},
       URL = {https://doi.org/10.30757/alea.v19-07},
}

@article {lupu2016note,
    AUTHOR = {Lupu, Titus and Werner, Wendelin},
     TITLE = {A note on {I}sing random currents, {I}sing-{FK}, loop-soups
              and the {G}aussian free field},
   JOURNAL = {Electron. Commun. Probab.},
  FJOURNAL = {Electronic Communications in Probability},
    VOLUME = {21},
      YEAR = {2016},
     PAGES = {Paper No. 13, 7},
      ISSN = {1083-589X},
   MRCLASS = {60K35 (82B20 82B44)},
  MRNUMBER = {3485382},
MRREVIEWER = {Ji\v r\'i\ \v Cern\'y},
       DOI = {10.1214/16-ECP4733},
       URL = {https://doi.org/10.1214/16-ECP4733},
}

@article {kozma2013lower,
    AUTHOR = {Kozma, Gady and Sidoravicius, Vladas},
     TITLE = {Lower bound for the escape probability in the {L}orentz mirror
              model on {$\Bbb{Z}^2$}},
   JOURNAL = {Israel J. Math.},
  FJOURNAL = {Israel Journal of Mathematics},
    VOLUME = {209},
      YEAR = {2015},
    NUMBER = {2},
     PAGES = {683--685},
      ISSN = {0021-2172,1565-8511},
   MRCLASS = {60K35 (82B43)},
  MRNUMBER = {3430256},
MRREVIEWER = {Hans-Otto\ Georgii},
       DOI = {10.1007/s11856-015-1233-1},
       URL = {https://doi.org/10.1007/s11856-015-1233-1},
}

@article {gunaratnam2024existence,
    AUTHOR = {Gunaratnam, Trishen S. and Krachun, Dmitrii and Panagiotis,
              Christoforos},
     TITLE = {Existence of a tricritical point for the {B}lume-{C}apel model
              on {$\Bbb{Z}^d$}},
   JOURNAL = {Probab. Math. Phys.},
  FJOURNAL = {Probability and Mathematical Physics},
    VOLUME = {5},
      YEAR = {2024},
    NUMBER = {3},
     PAGES = {785--845},
      ISSN = {2690-0998,2690-1005},
   MRCLASS = {60K35 (82B43)},
  MRNUMBER = {4765498},
MRREVIEWER = {Adrian\ Muntean},
       DOI = {10.2140/pmp.2024.5.785},
       URL = {https://doi.org/10.2140/pmp.2024.5.785},
}

@article{hansen2024nonuniquenessphasetransitionsgraphical,
  title={Non-uniqueness of phase transitions for graphical repre-sentations of the {I}sing model on tree-like graphs},
  author={Hansen, Ulrik Thinggaard and Klausen, Frederik Ravn and Wildemann, Peter},
  journal={ALEA Lat. Am. J. Probab. Math. Stat.},
  volume={22},
  pages={889--904},
  year={2025}
}

@article{gunaratnam2025supercritical,
  title={The supercritical phase of the $\varphi^{4}$ model is well behaved},
  author={Gunaratnam, Trishen S and Panagiotis, Christoforos and Panis, Romain and Severo, Franco},
  journal={arXiv preprint arXiv:2501.05353},
  year={2025}
}

@article{chayes1998graphical,
  title={Graphical representations and cluster algorithms II},
  author={Chayes, Lincoln and Machta, Jon},
  journal={Phys. A Stat. Mech. Appl.},
  fjournal={Physica A: Statistical Mechanics and its Applications},
  volume={254},
  number={3-4},
  pages={477--516},
  year={1998},
  publisher={Elsevier}
}

@article{duminil2021conformal,
  title={Conformal invariance of double random currents {II}: tightness and properties in the discrete},
  author={Duminil-Copin, Hugo and Lis, Marcin and Qian, Wei},
  journal={arXiv preprint arXiv:2107.12880},
  year={2021}
}

@article {duminil2021conformal2,
    AUTHOR = {Duminil-Copin, Hugo and Lis, Marcin and Qian, Wei},
     TITLE = {Conformal invariance of double random currents {I}:
              {I}dentification of the limit},
   JOURNAL = {Proc. Lond. Math. Soc. (3)},
  FJOURNAL = {Proceedings of the London Mathematical Society. Third Series},
    VOLUME = {130},
      YEAR = {2025},
    NUMBER = {1},
     PAGES = {Paper No. e70022},
      ISSN = {0024-6115,1460-244X},
   MRCLASS = {82B20 (81T40)},
  MRNUMBER = {4846862},
       DOI = {10.1112/plms.70022},
       URL = {https://doi.org/10.1112/plms.70022},
}

@article{swendsen1987nonuniversal,
  title={Nonuniversal critical dynamics in Monte Carlo simulations},
  author={Swendsen, Robert H and Wang, Jian-Sheng},
  JOURNAL = {Phys. Rev. Lett.},
  fjournal={Physical review letters},
  volume={58},
  number={2},
  pages={86},
  year={1987},
  publisher={APS}
}

@article {aizenman2019emergent,
    AUTHOR = {Aizenman, Michael and Duminil-Copin, Hugo and Tassion, Vincent
              and Warzel, Simone},
     TITLE = {Emergent planarity in two-dimensional {I}sing models with
              finite-range interactions},
   JOURNAL = {Invent. Math.},
  FJOURNAL = {Inventiones Mathematicae},
    VOLUME = {216},
      YEAR = {2019},
    NUMBER = {3},
     PAGES = {661--743},
      ISSN = {0020-9910,1432-1297},
   MRCLASS = {82B20 (60K35)},
  MRNUMBER = {3955708},
MRREVIEWER = {Rongfeng\ Sun},
       DOI = {10.1007/s00222-018-00851-4},
       URL = {https://doi.org/10.1007/s00222-018-00851-4},
}

@article {DCsharpness,
    AUTHOR = {Duminil-Copin, Hugo and Raoufi, Aran and Tassion, Vincent},
     TITLE = {Sharp phase transition for the random-cluster and {P}otts
              models via decision trees},
   JOURNAL = {Ann. of Math. (2)},
  FJOURNAL = {Annals of Mathematics. Second Series},
    VOLUME = {189},
      YEAR = {2019},
    NUMBER = {1},
     PAGES = {75--99},
      ISSN = {0003-486X,1939-8980},
   MRCLASS = {60K35},
  MRNUMBER = {3898174},
MRREVIEWER = {Stanislav\ Volkov},
       DOI = {10.4007/annals.2019.189.1.2},
       URL = {https://doi.org/10.4007/annals.2019.189.1.2},
}

@article {HistoricPRCM4,
    AUTHOR = {Maritan, A. and Omero, C},
     TITLE = {On the gauge {P}otts model and the plaquette percolation
              problem},
   JOURNAL = {Nuclear Phys. B},
  FJOURNAL = {Nuclear Physics. B. Theoretical, Phenomenological, and
              Experimental High Energy Physics. Quantum Field Theory and
              Statistical Systems},
    VOLUME = {210},
      YEAR = {1982},
    NUMBER = {4},
     PAGES = {553--566},
      ISSN = {0550-3213,1873-1562},
   MRCLASS = {82A68 (81E25 82A43)},
  MRNUMBER = {686898},
       DOI = {10.1016/0550-3213(82)90179-1},
       URL = {https://doi.org/10.1016/0550-3213(82)90179-1},
}

@article {HistoricPRCM1,
    AUTHOR = {Aizenman, Michael and Fr\"ohlich, J\"urg},
     TITLE = {Topological anomalies in the {$n$}\ dependence of the
              {$n$}-states {P}otts lattice gauge theory},
   JOURNAL = {Nuclear Phys. B},
  FJOURNAL = {Nuclear Physics. B. Theoretical, Phenomenological, and
              Experimental High Energy Physics. Quantum Field Theory and
              Statistical Systems},
    VOLUME = {235},
      YEAR = {1984},
    NUMBER = {1},
     PAGES = {1--18},
      ISSN = {0550-3213,1873-1562},
   MRCLASS = {81E25 (82A68)},
  MRNUMBER = {740980},
MRREVIEWER = {Claus\ Montonen},
       DOI = {10.1016/0550-3213(84)90144-5},
       URL = {https://doi.org/10.1016/0550-3213(84)90144-5},
}

@article{HistoricPRCM2,
title = {The correct extension of the Fortuin-Kasteleyn result to plaquette percolation},
journal = {Nuclear Physics B},
volume = {235},
number = {1},
pages = {19-23},
year = {1984},
issn = {0550-3213},
doi = {https://doi.org/10.1016/0550-3213(84)90145-7},
url = {https://www.sciencedirect.com/science/article/pii/0550321384901457},
author = {J.T. Chayes and L. Chayes}
}

@article{HistoricPRCM3,
title = {Large $q$ expansions for $q$-state gauge-matter Potts models in lagrangian form},
journal = {Nuclear Physics B},
volume = {170},
number = {3},
pages = {409-432},
year = {1980},
issn = {0550-3213},
doi = {https://doi.org/10.1016/0550-3213(80)90419-8},
url = {https://www.sciencedirect.com/science/article/pii/0550321380904198},
author = {Paul Ginsparg and Yadin Y. Goldschmidt and Jean-Bernard Zuber}
}

\end{document}